\documentclass[]{siamart0516}

\usepackage{lipsum}
\usepackage{amsfonts}
\usepackage{graphicx}
\usepackage{epstopdf,mathtools}
\usepackage{algorithmic}
\ifpdf
  \DeclareGraphicsExtensions{.eps,.pdf,.png,.jpg}
\else
  \DeclareGraphicsExtensions{.eps}
\fi

\newcommand{\TheTitle}{Dual virtual element method for discrete fractures networks}
\newcommand{\TheAuthors}{A. Fumagalli, E. Keilegavlen}

\headers{\TheTitle}{\TheAuthors}

\title{{\TheTitle}\thanks{Submitted to the editors DATE.
\funding{This work was partially funded by the ANIGMA project from the Research
Council of Norway (project no. 244129/E20) through the ENERGIX program.}}}

\author{
    Alessio Fumagalli\thanks{Department of Mathematics,
    University of Bergen, Bergen, Norway (\email{alessio.fumagalli@uib.no},
    \email{eirik.keilegavlen@uib.no}).}
    \and
    Eirik Keilegavlen\footnotemark[2]
}

\usepackage{amsopn}

\ifpdf
\hypersetup{
  pdftitle={\TheTitle},
  pdfauthor={\TheAuthors}
}
\fi

\graphicspath{{./Parts/Images/},{./Parts/Images/four_fractures/}}

\RequirePackage{amscd, amsfonts, amssymb}
\RequirePackage{mathrsfs, mathtools}
\RequirePackage{stmaryrd, array}
\RequirePackage{bm}

\newcommand{\defeq}{\vcentcolon=}

\newcommand{\jump}[2]{\left\llbracket {#1} \right\rrbracket_{#2}}

\newcommand{\ie}{\textit{i.e.~}}

\newcommand{\eg}{\textit{e.g.~}}

\newcommand{\dof}{\textit{dof}}

\newcommand{\dofs}{\textit{d.o.f.}'s~}

\newcommand{\abs}[1]{\left| #1 \right|}

\newcommand{\norm}[1]{\Vert #1 \Vert}

\newtheorem{problem}{Problem}

\newtheorem{remark}{Remark}

\DeclareMathOperator*{\base}{base}

\newcommand{\FIX}[2]{#2}

\begin{document}
\maketitle




\begin{abstract}
    Discrete fracture networks is a key ingredient in the simulation of physical
    processes which involve fluid flow in the underground, when the surrounding
    rock matrix is considered impervious.  In this paper we present two
    different models to compute the pressure field and Darcy velocity in the
    system. The first allows a normal flow out of a fracture at the
    intersections, while the second grants also a tangential flow along the
    intersections.  For the numerical discretization, we use the mixed virtual
    \FIX{finite}{} element method as it is known to handle grid elements of, almost, any
    arbitrary shape. The flexibility of the discretization allows us to loosen the requirements
    on grid construction, and thus significantly simplify the flow discretization
    compared to traditional discrete fracture network models.
    A coarsening algorithm, from the algebraic multigrid
    literature, is also considered to further speed up the computation.
    The performance of the method is validated by numerical experiments.
\end{abstract}









\section{Introduction}


In many hard rocks fractures, naturally occurring or engineered, are of a
paramount importance to understand and simulate flow paths.  The construction of
efficient simulation models for flow in fracture networks is therefore of
relevance for applications such as energy recovery and storage, waste disposal
(nuclear and CO$_2$), just to name a few
\cite{bear1993flow,Karimi-Fard2006,Hui2008}.

Flow in the fracture planes is usually modeled by Darcy's law
\cite{Martin2005,Karimi-Fard2006}, or Forchheimer's law if the Reynold number is
sufficiently high \cite{Frih2008,Knabner2014}, and it may also be of crucial
interest to include flow along intersections in the models
\cite{Xiang2010,Formaggia2012,Fumagalli2012g}.
Because of hydraulic aperture, which is several order of magnitude smaller than
other characteristic sizes of the problem, from a modeling perspective a
fracture may be represented as a two-dimensional plane (generally manifold)
\cite{Alboin2000a,Martin2005,Amir2005}, embedded in a three-dimensional domain.
The set of, possibly intersecting, fractures forms the network where the flow
may take place and, usually, is referred as a discrete fracture network (DFN).
Depending on how the fractures were created, they may intersect in an arbitrary
manner, leading to highly complex simulation geometries.  The flow properties of
a fracture are determined by its geo-mechanical and geo-chemical history
\cite{faybishenko2000dynamics}, and there can be significant heterogeneities as
well as permeability anisotropy in the fracture plane. Infilling processes or
geological movements may alter also the composition and orientation of the
material presents in the intersections, leaving the latter as a privileged
patterns or obstacles for the flow.

One of the main challenge in DFN simulations in complex fracture geometries is
the construction of the computational mesh, combined with the subsequent
discretization of the flow equations.  Depending on the properties of the chosen
numerical method, this may put significant constraints on the mesh algorithm, in
particular in meshing of fracture intersections
\cite{Mustapha2007,Hyman2014,Hyman2015}.  Compared to two-dimensional gridding
problems, which by themselves can be non-trivial, gridding of DFN problems is
complicated by the requirement, commonly made, that the grids in different
fracture planes are mutually conforming.  That is, hanging nodes are not allowed
in fracture intersections.  The gridding can be simplified by allowing for a
relaxed interpretation of the fracture planes \cite{Karimi-Fard2016}, in effect
locally assigning a curvature to the fracture plane.

If this requirement is loosened, and some of the burden is transferred to an
appropriately chosen numerical method, the gridding problem becomes
significantly simpler and standard algorithms may be applied to create the
global mesh.  For instance, the extended finite element method handles
discontinuities internal to elements by enriching the approximation space
\cite{Hansbo2002,DAngelo2011,Fumagalli2012g}.  This removes all requirements on
the mesh conforming to fracture intersections, and thus significantly simplifies
the gridding \cite{Fumagalli2012g,Berrone2013a,Flemisch2016,Pieraccini2016}.

A reasonable compromise between burdening mesh generation and discretization is
to require that meshes conform to fracture intersections, but also allow for
hanging nodes. This approach allows for the independent construction of a set of
bi-dimensional meshes in the fracture planes, which is considerably less
difficult than gridding the whole DFN model fully coupled. See \cref{fig:real}
and \cref{fig:real_zoom} as an example.
\begin{figure}[tbp]
    \centering
    \includegraphics[width=0.5\textwidth]{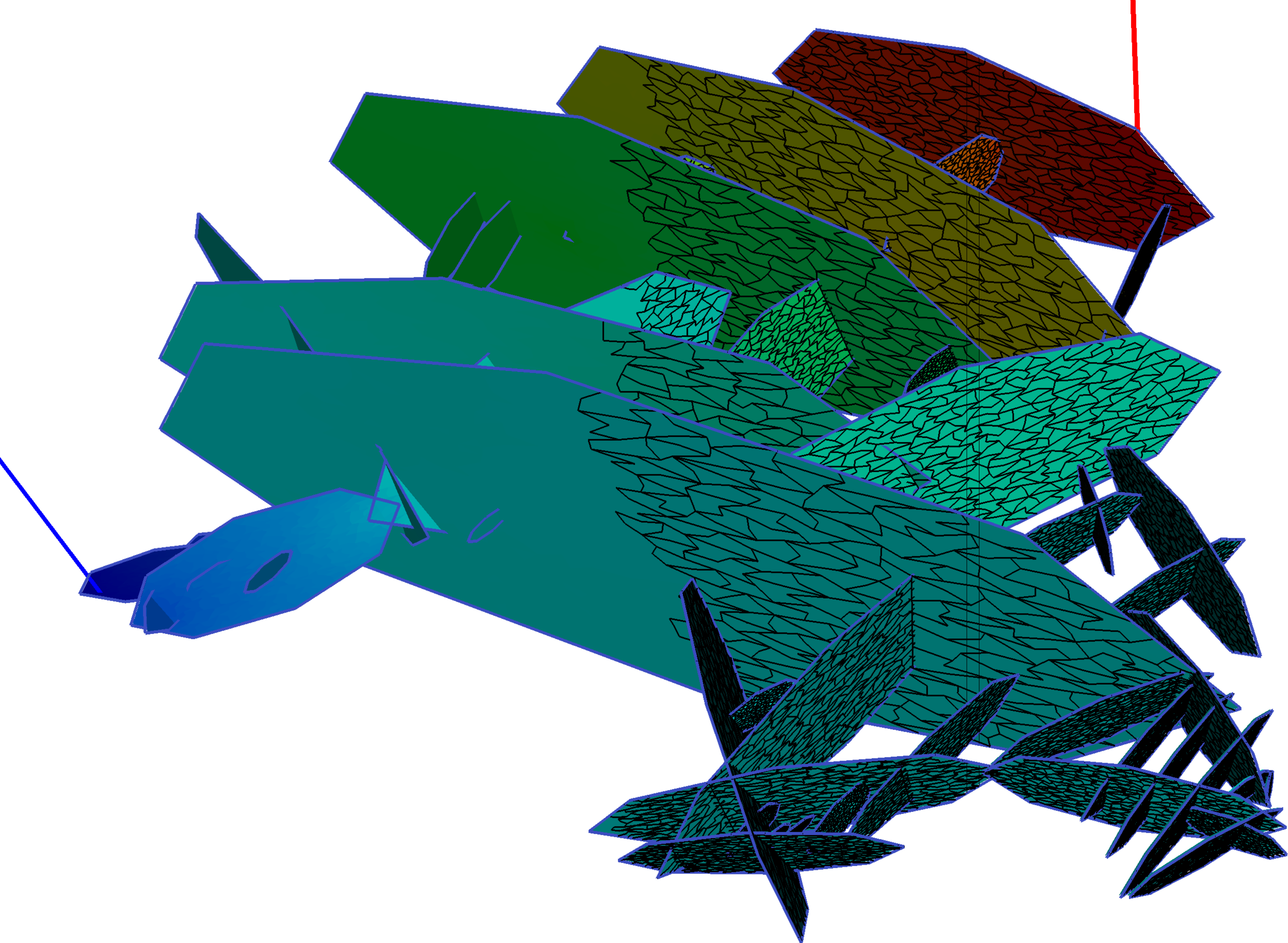}%
    \caption{Example of pressure field in a realistic DFN with several
    intersecting fractures, represented as ellipses. The fractures are
    represented by one co-dimensional objects. The blue line represents an
    injection well, while the red line a production well. Note that the
    represented mesh is \FIX{non}{} conforming at the intersections. A detailed
    representation of the grid is reported in \cref{fig:real_zoom}.}%
    \label{fig:real}
\end{figure}
\begin{figure}[tbp]
    \centering
    \includegraphics[width=0.475\textwidth]{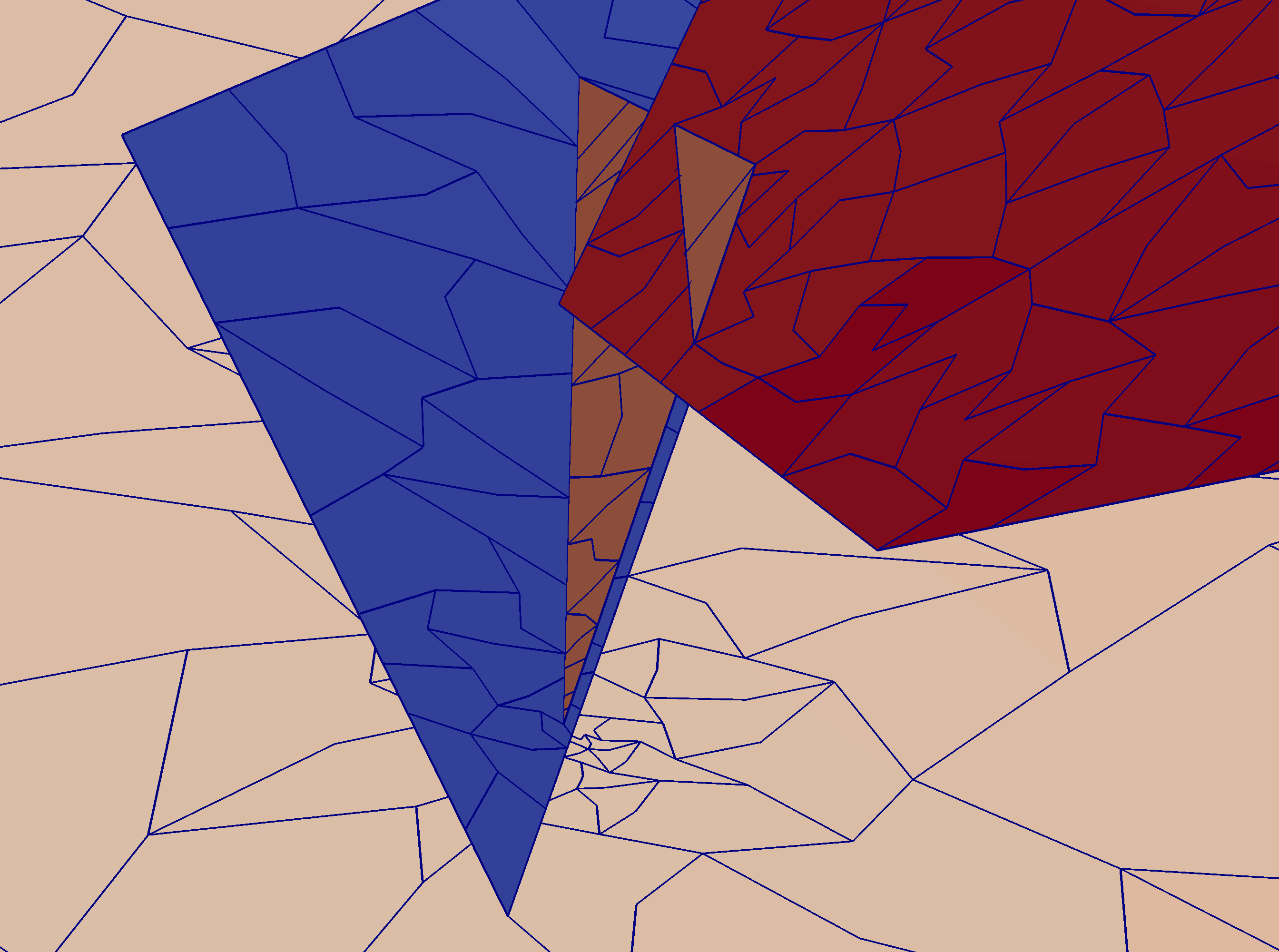}%
    \hfill%
    \includegraphics[width=0.475\textwidth]{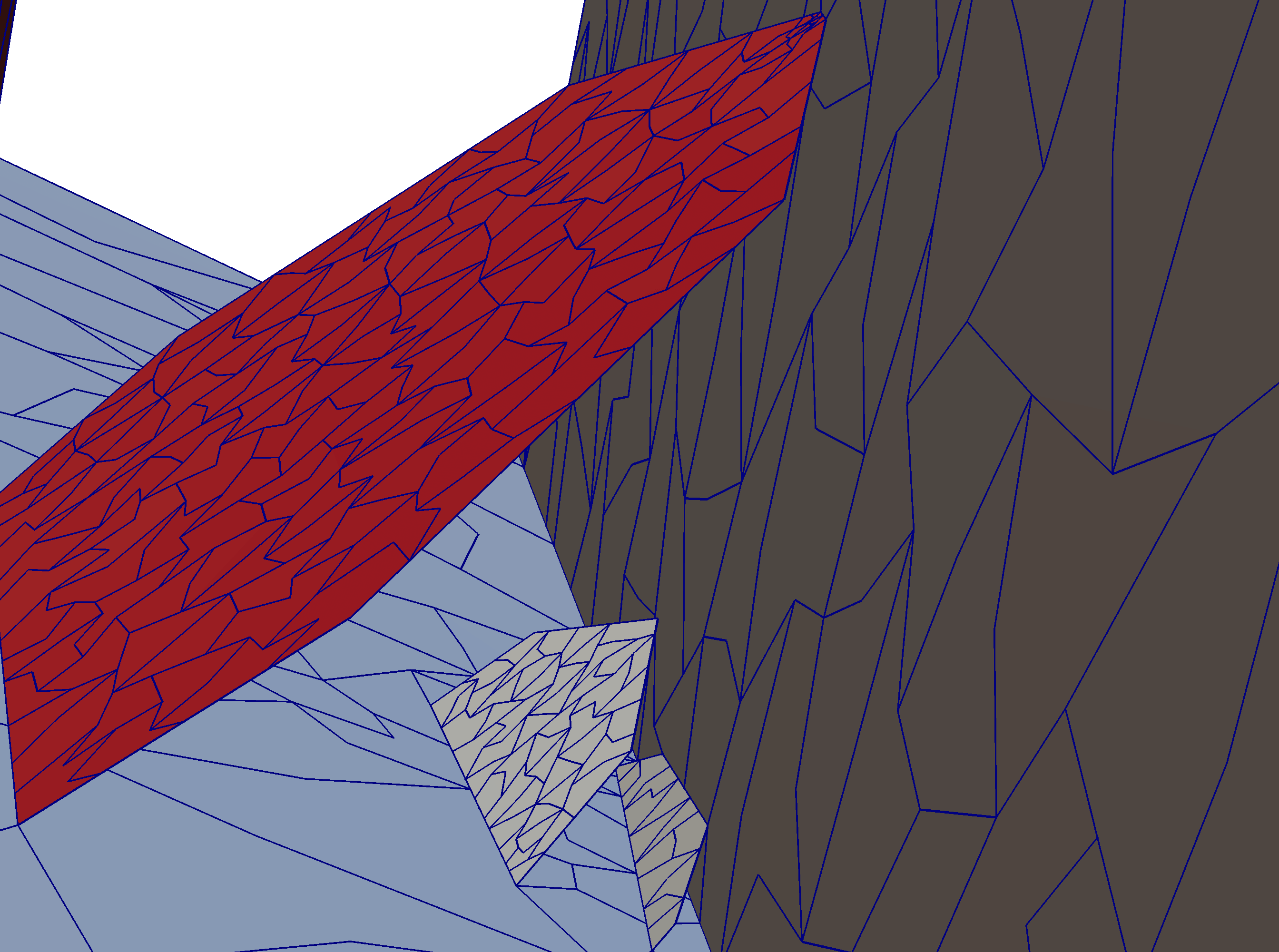}
    \includegraphics[width=0.475\textwidth]{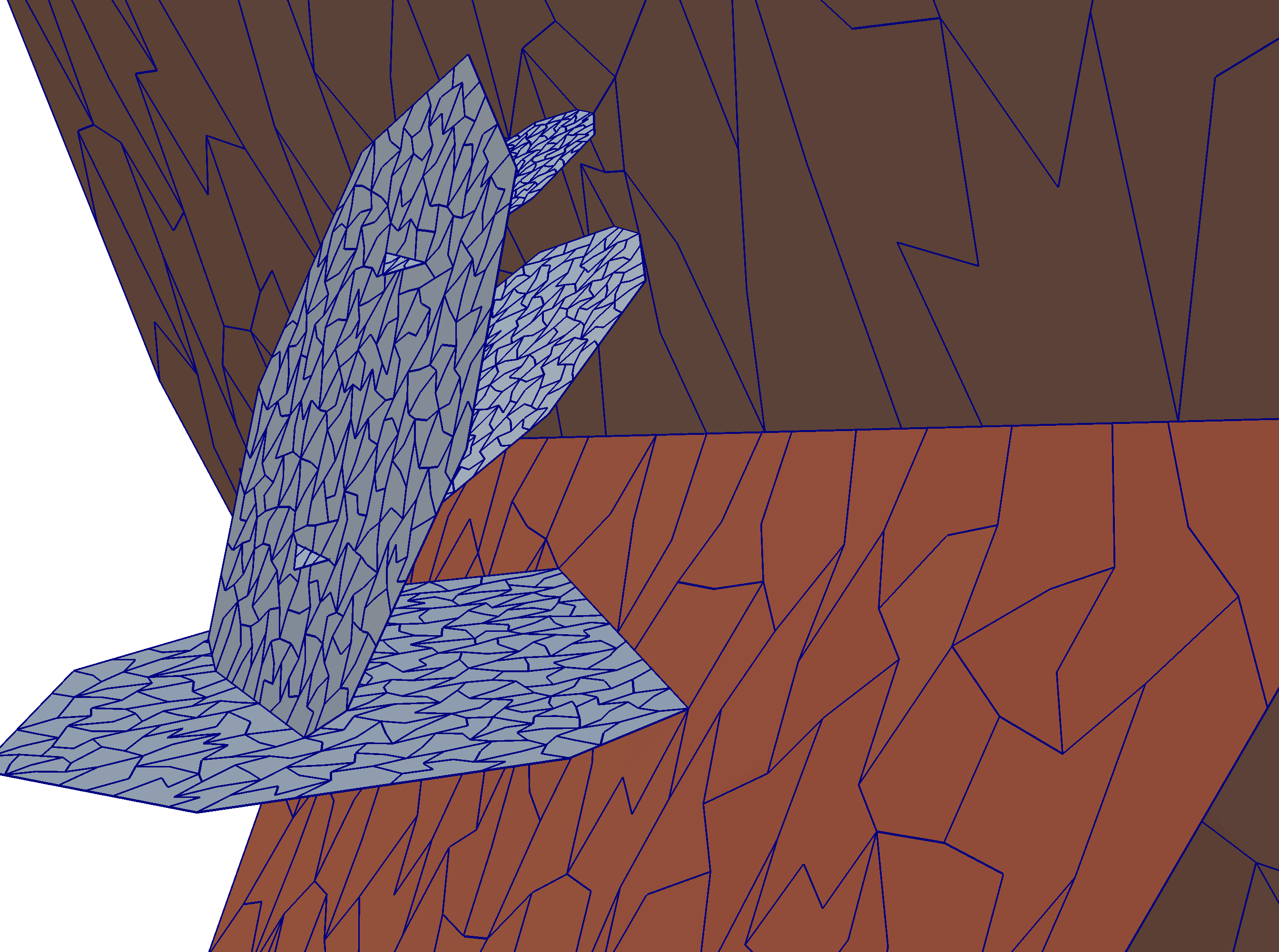}%
    \hfill%
    \includegraphics[width=0.475\textwidth]{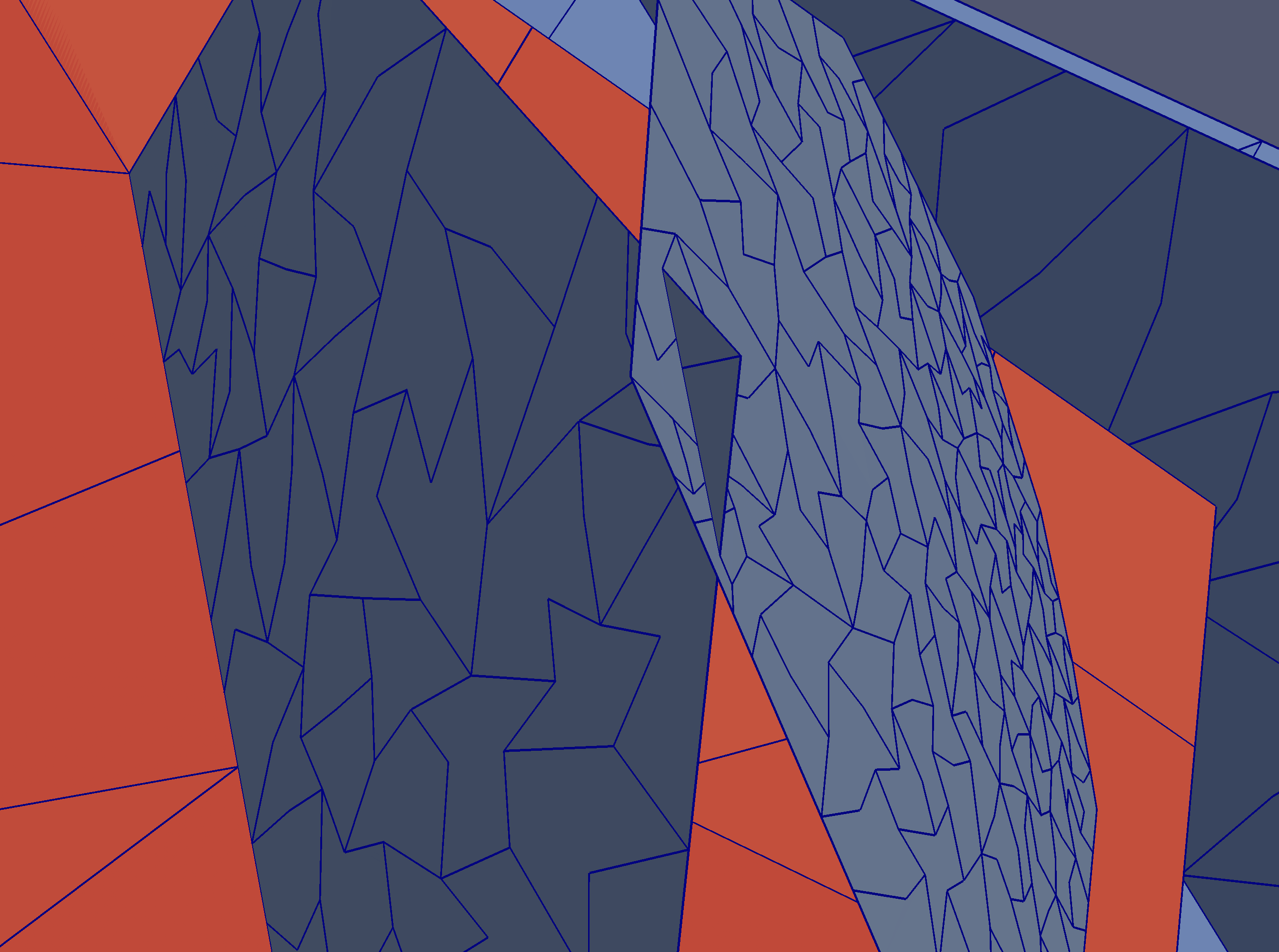}%
    \caption{Four different details of the grid for the example represented in
    \cref{fig:real}. The method considered is able to handle these
    intersection among fractures as conforming.}%
    \label{fig:real_zoom}
\end{figure}
To further increase
the applicability of discretization scheme, it should then allow for general
polygonal cells, including hanging nodes.  In particular, we note that the
virtual element method (VEM) has successfully been applied in the DFN setting
\cite{Benedetto2014,Benedetto2016}. See
\cite{BeiraodaVeiga2013,BeiraodaVeiga2014a,Brezzi2014,BeiraodaVeiga2014b,Antonietti2014a,BeiraoVeiga2016} for a
presentation of VEM in various contexts.

In this work, we consider different models to describe the single-phase flow in
the network of fractures, with and without the tangential contribution of the
intersections included. A reduced model is thus provided, with suitable normal and
tangential effective permeabilities, to describe the flow in a two-codimensional
framework. We consider a dual formulation of the generalized Darcy problem where
the virtual element method is extend to handle the proposed modes. Local mass
conservation is thus guaranteed and the normal flux, from each edge in the mesh,
may be directly used to simulate the heat or tracer transport in the DFN.  A
gridding strategy is also presented to generate a conforming mesh, in the sense
that the edges may be split to ensure conformity along the intersections of the
fractures.  The proposed scheme is thus a flexible and robust tool to perform
simulations in this context. We present several numerical tests to validate the
approach, focusing the study on the order of convergence for the pressure and
the velocity in different scenarios.

The paper is organized as follows. The mathematical models and their analyses to
describe the fluid flow in the DFN are presented in
\cref{sec:continuous}.\FIX{}{The \cref{sec:well_posedness} is devoted to
prove the well posedness of the continuous models considered.}
In \cref{sec:discrete} we review the implementation of the
mixed virtual element method, and its implementation in a DFN setting.
The gridding and coarsening strategy are presented in
\cref{sec:coarsening}. Experimental results presented in \cref{sec:examples} mainly
consider the error decay of the numerical method presented. Conclusions
follow in \cref{sec:conclusions}. \cref{sec:coarseapp} briefly describes the
coarsening strategy while \cref{sec:appendix} contains tables related to the
examples.





\section{Continuous \FIX{M}{m}odel} \label{sec:continuous}


This section is devoted to the presentation of the physical model.
In \cref{subsec:single_phase} the setting for single-phase flow is introduced.
\cref{subsec:fracture_flow} presents the model for a single fracture, while the
different couplings among fractures are described in
\cref{subsec:coupling_between_fractures}.
\FIX{In \cref{subsec:well_posedness} the weak formulation of the problem
and the related functional spaces are introduced. In the same section well
posedness results are shown.}{}


\subsection{Single-phase flow in DFN} \label{subsec:single_phase}


We make use of the symbol $\sharp(A) \in \mathbb{N}$ indicating the counting
measure of $A$.  We define $\mathcal{N}$ the set of indexes associated to the
network of fractures, one value identify a fracture.  Let us consider $\sharp(
\mathcal{N} )$ distinct and planar domains $\Omega_i$, for $i \in \mathcal{N}$,
embedded in $\mathbb{R}^3$ such that their union is indicated by $\Omega \defeq
\cup_{i \in \mathcal{N}} \Omega_i$. Each polygon represents a fracture and
$\Omega$ is the discrete fracture network. We indicate by $\mathcal{I}$ the set
of indexes associated to the intersection among fractures, one value identify
one intersection.  We define also the set of one co-dimension intersections
among fractures as $\gamma$, which can be viewed also as a disjoint union of
$\sharp( \mathcal{I} )$ lines $\gamma_k$ such that $\gamma = \cup_{k \in
\mathcal{I}} \gamma_k$.  We indicate, for each fracture, the finite set of
indexes $\mathcal{G}_i \defeq \left\{ k \in \mathbb{N}: \gamma_k \cap \Omega_i
\neq \emptyset \right\}$. We suppose that if $k \in \mathcal{G}_i$ then
$\gamma_k$ belongs to the internal part of $\Omega_i$, \FIX{}{that is, we assume
fractures do not intersect along fracture boundaries}. We indicate by
$\mathcal{P}$ the set of indexes associated with the intersection of
one-codimensional objects, one value identifies one intersection.  Finally the set
of two co-dimensional intersections among fractures is indicated by $\xi$, which
can be viewed as a union of $\sharp( \mathcal{P} )$ points $\xi_l$ such that
$\xi = \cup_{l \in \mathcal{P} } \xi_l$. We suppose that if $\xi_l \subset
\gamma_k$, for some $l$ and $k$, then $\xi_l$ belongs to the internal part of
$\gamma_k$. Note that  $ \gamma \cup \xi = \cap_{i \in \mathcal{N}} \Omega_i$.
The aforementioned sets of indexes are ordered in a natural way.  See
\cref{fig:example_fracts} as an example.  Throughout the paper generally,
subscripts $i$ and $j$ denote quantities related with fracture planes
$\Omega_i$, subscript $k$ is associated with one co-dimensional intersections
$\gamma_k$, while subscript $l$ is used for two co-dimensional intersections
$\xi_l$.  We indicate data and unknowns defined on $\gamma$ with
$\hat{\cdot}$.
\begin{figure}[tbp]
    \centering
    \resizebox{0.5\textwidth}{!}{\fontsize{24pt}{7.2}\selectfont%
\begingroup%
  \makeatletter%
  \providecommand\color[2][]{%
    \errmessage{(Inkscape) Color is used for the text in Inkscape, but the package 'color.sty' is not loaded}%
    \renewcommand\color[2][]{}%
  }%
  \providecommand\transparent[1]{%
    \errmessage{(Inkscape) Transparency is used (non-zero) for the text in Inkscape, but the package 'transparent.sty' is not loaded}%
    \renewcommand\transparent[1]{}%
  }%
  \providecommand\rotatebox[2]{#2}%
  \ifx\svgwidth\undefined%
    \setlength{\unitlength}{415.26142578bp}%
    \ifx\svgscale\undefined%
      \relax%
    \else%
      \setlength{\unitlength}{\unitlength * \real{\svgscale}}%
    \fi%
  \else%
    \setlength{\unitlength}{\svgwidth}%
  \fi%
  \global\let\svgwidth\undefined%
  \global\let\svgscale\undefined%
  \makeatother%
  \begin{picture}(1,0.7406406)%
    \put(0,0){\includegraphics[width=\unitlength,page=1]{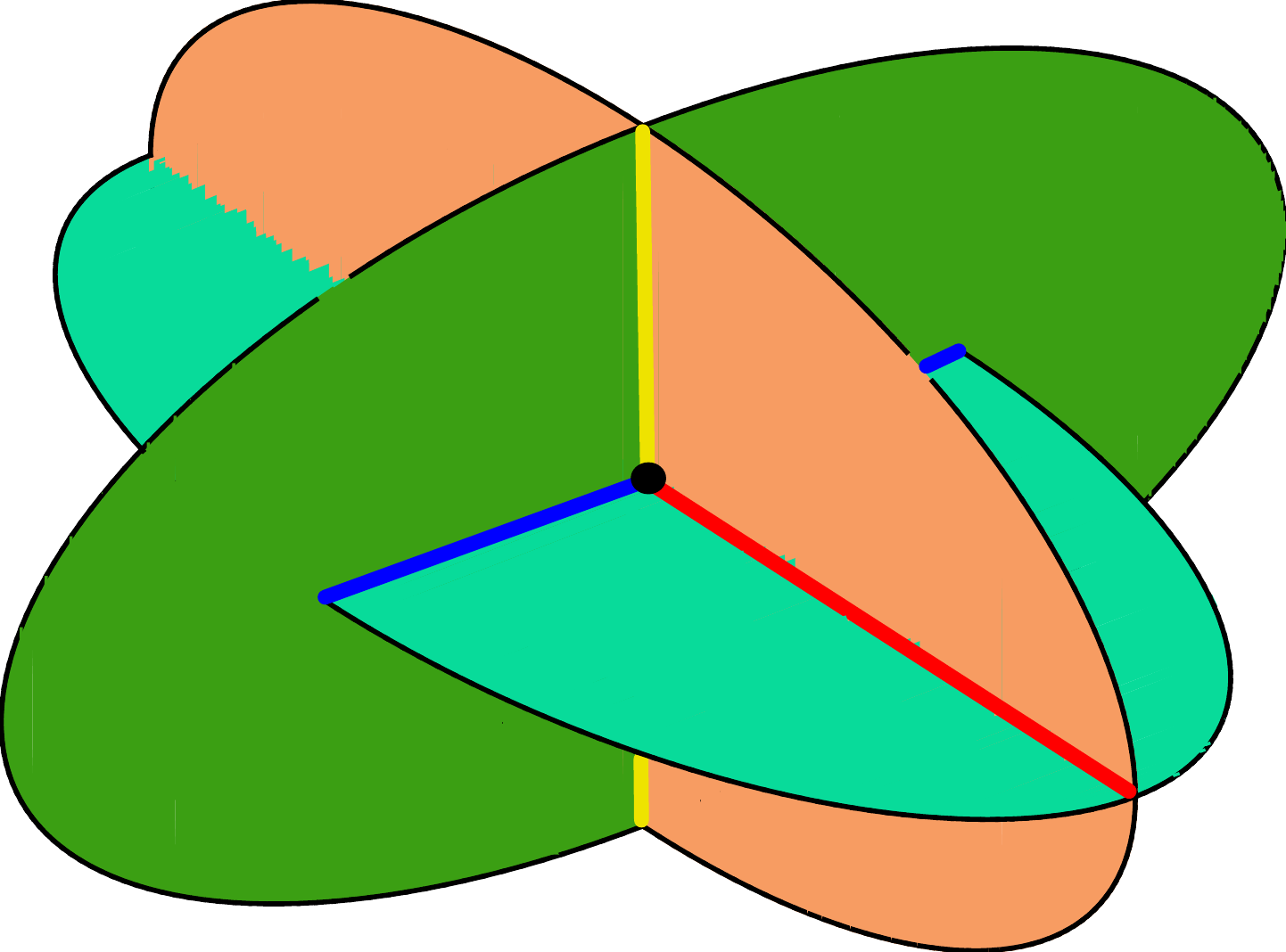}}%
    \put(0.76491029,0.56441822){\color[rgb]{0,0,0}\makebox(0,0)[lb]{\smash{$\Omega_1$}}}%
    \put(0,0){\includegraphics[width=\unitlength,page=2]{fracts.pdf}}%
    \put(0.53996008,0.16225442){\color[rgb]{0,0,0}\makebox(0,0)[lb]{\smash{$\Omega_2$}}}%
    \put(0.16101266,0.65340768){\color[rgb]{0,0,0}\makebox(0,0)[lb]{\smash{$\Omega_3$}}}%
    \put(0.43581661,0.49416917){\color[rgb]{0,0,0}\makebox(0,0)[lb]{\smash{$\gamma_1$}}}%
    \put(0.62014344,0.30536745){\color[rgb]{0,0,0}\makebox(0,0)[lb]{\smash{$\gamma_3$}}}%
    \put(0.51801581,0.37188476){\color[rgb]{0,0,0}\makebox(0,0)[lb]{\smash{$\xi$}}}%
    \put(0.35474607,0.28117938){\color[rgb]{0,0,0}\makebox(0,0)[lb]{\smash{$\gamma_1$}}}%
    \put(0,0){\includegraphics[width=\unitlength,page=3]{fracts.pdf}}%
  \end{picture}%
\endgroup%
    }%
    \caption{Example of a DFN with three intersecting fractures, represented as
    ellipses. The fractures form three one co-dimensional objects, the straight
    lines, and one two co-dimensional object, the black dot.}%
    \label{fig:example_fracts}
\end{figure}

Our objective is to compute the pressure field, indicated by $p$ and $\hat{p}$,
and the Darcy velocity field, indicated by $\bm{u}$ and $\hat{\bm{u}}$, in
$\Omega$ such that the Darcy equations are fulfilled. To simplify the presentation
we start by considering a single fracture $\Omega_i$ without intersections. The
normal of the fracture is indicated by $\bm{n}_i$, and we define the normal
projection matrix as $N_i \defeq \bm{n}_i \otimes \bm{n}_i$ and the tangential
projection matrix as $T_i \defeq I - N_i$.  In the following, for each fracture,
we make use of tangential divergence and gradient, defined as
\begin{gather*}
    \nabla_{T_i} \cdot \defeq T_i : \nabla
    \quad \text{ and } \quad
    \nabla_{T_i} \defeq T_i \nabla.
\end{gather*}
In the sequel, we will sometimes drop the subscript on $T_i$ when there
should be no room for confusion.


\subsection{\FIX{}{Fracture flow}} \label{subsec:fracture_flow}


Following the work of \cite{Martin2005,Formaggia2012} we assume that the
fracture permeability can be written as a full elliptic
second order tensor for the permeability in the tangential space. We adopt a reduced
order, or hybrid dimensional, model, to describe $p_i$ and $\bm{u}_i$ in
the fracture. The model explicitly takes into account the hydraulic aperture
$d_i$ of the fracture, which can change in space along the fracture.
\FIX{}{We require that exist $d_* > 0$ such that $d_i \geq d_*$ for all $i$,
i.e., the fracture does not degenerate. In the forthcoming models we assume that
the permeability and fracture aperture are constant in each cell of the
computational grid, this is a practical assumption when the discretization is
applied.} The
generalized Darcy equation and the conservation of mass in $\Omega_i$ can be
written as
\begin{gather}\label{eq:darcy_frac}
    \begin{aligned}
        &\bm{u}_i + \FIX{\kappa_i}{\lambda_i} \nabla p_i = \bm{0} \\
        &\nabla \cdot \bm{u}_i = \FIX{q_i}{f_i}
    \end{aligned}
    \quad \text{ in } \Omega_i
    \FIX{,}{%
    \quad\text{and}\quad
    p_i = g_i \text{ on } \partial \Omega_i,
    }
\end{gather}
where\FIX{}{, for each $\Omega_i$,} $\FIX{\kappa_i}{\lambda_i} \defeq d_i k_i $
is the effective permeability\FIX{$\Omega_i$}{, $k_i$ the tangential
permeability, $f_i$ the scalar sink or source, and $g_i$ the boundary
condition}.  \FIX{Suitable boundary conditions can be considered to
\cref{eq:darcy_frac} such that the problem is well posed, see
\cite{Brezzi1991}.}{For simplicity of the proofs, we assume Dirichlet conditions
on all fractures, although it would suffice to require only Dirichlet conditions
on a non-zero measure portion of the boundary. The case of where some fractures with only Neumann
conditions are imposed is studied in the numerical examples. The proofs can also be
extended to this case to the price of increased technicalities.  Under these
assumptions, problem \cref{eq:darcy_frac} is well posed, see \cite{Brezzi1991}.}


\subsection{\FIX{}{Coupling Flow Between Fractures}}%
\label{subsec:coupling_between_fractures}


Now, consider two fractures $\Omega_i$ and $\Omega_j$ such that a one
co-dimensional intersection occurs $\gamma_{k} = \Omega_i \cap \Omega_j$.  We
indicate by $\bm{n}_k$ the set of unit vectors orthogonal to $\gamma_k$.  In the
internal part of $\Omega_i$, \ie in $\Omega_i \setminus \gamma_k$, equation
\cref{eq:darcy_frac} is applied and for each internal point of $\gamma_k$, far
from its possible internal ending, we locally identify two sides of $\Omega_i$
indicated $\Omega_i^+$ and $\Omega_i^-$.  Since this definition is local, for
simplicity we keep uniform notation $\Omega_i^+$ and $\Omega_i^-$ for ``the same
side'' of $\Omega_i$.  We indicate data and unknowns restricted to one side of a
fracture with a $+$ or $-$ superscript. \FIX{}{For convenience $\partial
\Omega_i$ represents the outer part of the boundary, i.e., without the
intersection.} We indicate by $T_i \bm{n}_k$ the unit
normal vector of $\gamma_k$ lying in the plane of $\Omega_i$ and pointing from
$\Omega_i^+$ to $\Omega_i^-$.  $T_j \bm{n}_k$ is defined similarly. See
\cref{fig:example_fr} as an example.
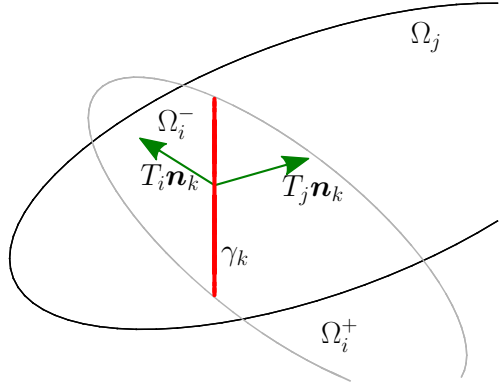
\begin{figure}[tbp]
    \centering
    \resizebox{0.5\textwidth}{!}{\fontsize{20pt}{8}\selectfont%
\begingroup%
  \makeatletter%
  \providecommand\color[2][]{%
    \errmessage{(Inkscape) Color is used for the text in Inkscape, but the package 'color.sty' is not loaded}%
    \renewcommand\color[2][]{}%
  }%
  \providecommand\transparent[1]{%
    \errmessage{(Inkscape) Transparency is used (non-zero) for the text in Inkscape, but the package 'transparent.sty' is not loaded}%
    \renewcommand\transparent[1]{}%
  }%
  \providecommand\rotatebox[2]{#2}%
  \ifx\svgwidth\undefined%
    \setlength{\unitlength}{332.2895752bp}%
    \ifx\svgscale\undefined%
      \relax%
    \else%
      \setlength{\unitlength}{\unitlength * \real{\svgscale}}%
    \fi%
  \else%
    \setlength{\unitlength}{\svgwidth}%
  \fi%
  \global\let\svgwidth\undefined%
  \global\let\svgscale\undefined%
  \makeatother%
  \begin{picture}(1,0.77575234)%
    \put(0,0){\includegraphics[width=\unitlength,page=1]{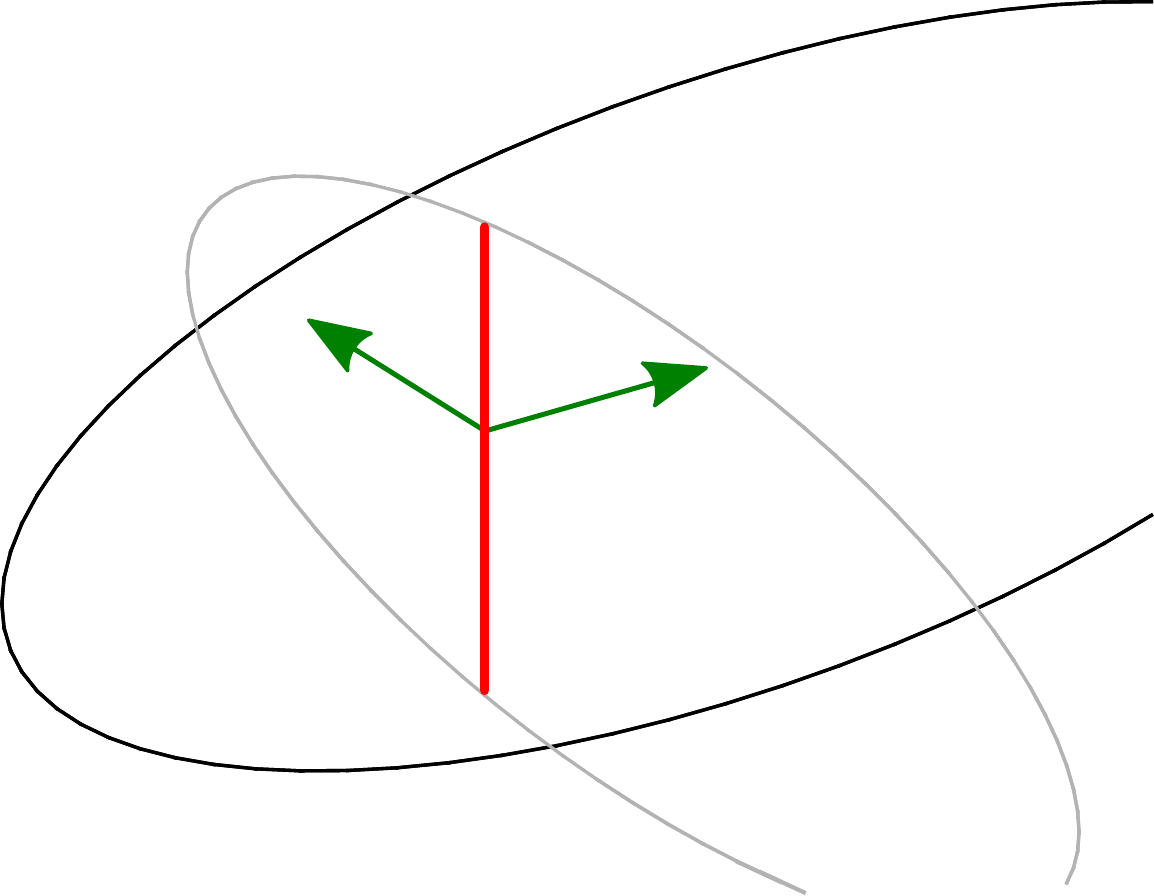}}%
    \put(0.55818362,0.3797893){\color[rgb]{0,0,0}\makebox(0,0)[lb]{\smash{$T_j \bm{n}_k$}}}%
    \put(0.26862133,0.40189945){\color[rgb]{0,0,0}\makebox(0,0)[lb]{\smash{$T_i \bm{n}_k$}}}%
    \put(0.82078876,0.69271686){\color[rgb]{0,0,0}\makebox(0,0)[lb]{\smash{$\Omega_j$}}}%
    \put(0.63793315,0.07873308){\color[rgb]{0,0,0}\makebox(0,0)[lb]{\smash{$\Omega_i^+$}}}%
    \put(0.30403482,0.5129695){\color[rgb]{0,0,0}\makebox(0,0)[lb]{\smash{$\Omega_i^-$}}}%
    \put(0.43388422,0.24905498){\color[rgb]{0,0,0}\makebox(0,0)[lb]{\smash{$\gamma_k$}}}%
  \end{picture}%
\endgroup%
    }%
    \caption{Example of two fractures with the notation introduced in the
    derivation of the models.}%
    \label{fig:example_fr}
\end{figure}
Following \cite{Alboin2000a,Amir2005}, at the intersection we assume the
coupling conditions along $\gamma_k$ as
\begin{gather}\label{eq:darcy_coupling_cont}
    \begin{aligned}
        &\sum_{m = i,j} \jump{\bm{u}_{m} \cdot T_m \bm{n}_{k}}{\gamma_k} = 0\\
        &\FIX{p_i |_{\gamma_k} = p_j |_{\gamma_k}}%
        {p_i^+ |_{\gamma_k} = p_i^- |_{\gamma_k} =
        p_j^+ |_{\gamma_k} = p_j^- |_{\gamma_k}}
    \end{aligned}
    \quad \text{ on } \gamma_k,
\end{gather}
where $\cdot |_{\gamma_k}$ is the trace operator from a fracture, $\Omega_i$ or
$\Omega_j$, to $\gamma_k$
and $\jump{\cdot}{\gamma_k}$ denotes the jump operator across $\gamma_k$ defined
as $\jump{ \bm{u}_m \cdot T_m \bm{n}_k}{\gamma_k} \defeq \bm{u}_m^+ \cdot T_m
\bm{n}_k |_{\gamma_k} - \bm{u}_m^- \cdot T_m \bm{n}_k |_{\gamma_k}$, for $m=i,j$.
Thus, while we allow for flow from $\Omega_i$ to $\Omega_j$, the permeability
of the intersection is so high, or the width of the fracture so small, that
the pressure drop is negligible, see \cite{Martin2005} for further considerations.

In subsurface flow, intersections between fractures, or faults, may be
significant conduits for fluid flow \FIX{}{\cite{Rotevatn2009,Peacock2017}}.
It is therefore of interest to extend the
model \cref{eq:darcy_coupling_cont} to allow for tangential flow along
$\gamma_k$ and for pressure jumps over the intersection, as studied in \eg
\cite{Formaggia2012,Schwenck2015,Boon2016}. Flow in the intersection
is still modeled by a generalized Darcy's law with an additional
source term from the surrounding fractures
\begin{subequations} \label{eq:darcy_coupling_disc}
\begin{gather} \label{eq:darcy_coupling_inters}%
    \begin{aligned}
        &\hat{\bm{u}}_{k} + \FIX{\hat{\kappa}_k}{\hat{\lambda}_k} \nabla
        \hat{{p}}_k = \bm{0} \\
        &\nabla \cdot \hat{\bm{u}}_k = \FIX{\hat{q}_k}{\hat{f}_k} +
        \sum_{m = i,j} \jump{ \bm{u}_m \cdot T_m \bm{n}_{k}}{\gamma_k}
    \end{aligned}
    \text{in } \gamma_{k}
    \FIX{,}{%
    \quad\text{and}\quad
    \begin{aligned}
        &\hat{p}_k = \hat{g}_k & \text{ on } \partial \gamma_k\\
        &\hat{\bm{u}}_k \cdot \bm{n}_{\gamma_k} = 0 & \text{ on } \partial
        \gamma_k^{in}
    \end{aligned}
    }
\end{gather}
where $\FIX{\hat{\kappa}_k}{\hat{\lambda}_k} \defeq d_i d_j \hat{k}(\gamma_k)$
is the effective permeability
along $\gamma_k$, $\hat{k}(\gamma_k)$ is the tangential permeability of
$\gamma_k$\FIX{}{, and $\hat{g}_k$ is the boundary condition on the, possibly
empty, outer part
(i.e., included in $\partial \Omega_i$ or $\partial \Omega_j$) of
the boundary of $\gamma_k$, for simplicity, indicated by $\partial \gamma_k$.
We define as $\gamma_k^{in}$ the, possible empty, part of the boundary of
$\gamma_k$ such that it lies internally to $\Omega_i$ or $\Omega_j$. We
identify by $\bm{n}_{\gamma_k}$ the unit vector tangential to the
intersection pointing outward with respect to $\partial\gamma_k^{in}$.} Note
that $d_i$, respectively $d_j$, is the aperture of fracture $\Omega_i$,
respectively $\Omega_j$. Their product represents the measure of the area normal
to the intersection. The differential operators are now defined along $\gamma_k$
and can be written in terms of local coordinates.  The coupling condition
\cref{eq:darcy_coupling_cont} is generalized, for $m=i,j$, as
\begin{gather} \label{eq:darcy_coupling_discontinuous}%
    \begin{aligned}
        &\bm{u}_m^+ \cdot T_m \bm{n}_{k} |_{\gamma_k}
        = - \FIX{\widetilde{\kappa}_{m}}{\widetilde{\lambda}_{m}}
        \left( \hat{p}_k - p_m^+ |_{\gamma_k}\right)\\
        &\bm{u}_m^- \cdot T_m \bm{n}_{k}|_{\gamma_k}
        = - \FIX{\widetilde{\kappa}_{m}}{\widetilde{\lambda}_{m}}
        \left( p_m^- |_{\gamma_k}- \hat{p}_k \right)
    \end{aligned}
    \quad \text{in } \gamma_k
\end{gather}
\end{subequations}
where $\FIX{\widetilde{\kappa}_{m}}{\widetilde{\lambda}_{m}}
\defeq \widetilde{k}(\gamma_k) / d_m$ is the effective
permeability of the fracture $\Omega_m$ in the direction normal to the intersection $\gamma_k$, and
$\widetilde{k}(\gamma_k)$ is the normal permeability of $\gamma_k$.
\FIX{We assume zero flow
boundary conditions at the  tips of the intersections.}{} Note that it is
possible to have different normal permeability for each fracture and even for
each side of a fracture, however to avoid confusion we assume a unique value for
$\widetilde{k}(\gamma_k)$.
Following \cite{DAngelo2011,Formaggia2012} the coupling
conditions \cref{eq:darcy_coupling_discontinuous} can be written as
\begin{gather*}
    \begin{aligned}
        &\jump{p_m}{\gamma_k} =
        \FIX{\widetilde{\kappa}_{m}^{-1}}{\widetilde{\lambda}_{m}^{-1}}
        \left(
        \bm{u}_m^+ \cdot T_m \bm{n}_{k} |_{\gamma_k} + \bm{u}_m^- \cdot T_m
        \bm{n}_{k}|_{\gamma_k} \right)\\
        &\jump{\bm{u}_m \cdot T_m \bm{n}_{k}}{\gamma_k} =
        \FIX{\widetilde{\kappa}_{m}}{\widetilde{\lambda}_{m}}
        \left( p_m^+ + p_m^- \right)
    \end{aligned}
    \quad \text{in } \gamma_k
\end{gather*}
where the jump of pressure and velocity across $\gamma_k$ are explicitly
written.

Finally, considering a single two co-dimensional intersection between fractures,
we have $\xi_l = \gamma_k \cap \gamma_n$. Following the idea proposed in
\cite{Boon2016} we require the same model presented in
\cref{eq:darcy_coupling_cont} at the intersection, we have
\begin{gather} \label{eq:two_codimensional}
    \begin{aligned}
        &\sum_{m = k,n} \jump{\hat{\bm{u}}_{m} \cdot T_m \bm{n}_l}{\xi_l} = 0\\
        &\FIX{\hat{p}_k |_{\xi_l} = \hat{p}_n |_{\xi_l}}
        {\hat{p}_k^+ |_{\xi_l} = \hat{p}_k^- |_{\xi_l} =
        \hat{p}_n^+ |_{\xi_l} = \hat{p}_n^- |_{\xi_l}}
    \end{aligned}
    \quad \text{in } \xi_l
\end{gather}
where $\cdot |_{\xi_l}$ is the trace operator from a one co-dimensional
intersection to $\xi_l$ and $T_m \bm{n}_l$ is a unit vector, pointing outward
from $\xi_l$, and lying on $\gamma_m$.  While it is possible to extend the model
to allow for jumps in pressure and velocity over $\xi_l$, see
\cite{Formaggia2012,Fumagalli2012g}, this is not considered herein.

We can now define two problems defined on the network $\Omega$ which will
be used in the sequel. The first problem considers continuous coupling conditions
at the fracture intersections, it is the following.
\begin{problem}[DFN Darcy flow - continuous coupling]\label{pb:cc}
    Considering the aforementioned data and assumptions on $\Omega$, the model
    problem with continuous coupling conditions is to find $(\bm{u}, p)$ in
    $\Omega$ such that \cref{eq:darcy_frac} is valid for all
    $i \in \mathcal{N}$\FIX{,}{ and } \cref{eq:darcy_coupling_cont} is valid for
    all $k \in \mathcal{I}$\FIX{, and \cref{eq:two_codimensional} is valid for
    all $l \in \mathcal{P}$}{}.
\end{problem}
The second problem allows a flow tangential to the fracture
intersection and a possible pressure jump across the one co-dimensional
intersections.
\begin{problem}[DFN Darcy flow - discontinuous coupling]\label{pb:dc}
    Considering the aforementioned data and assumptions on $\Omega$,
    the model problem with discontinuous coupling conditions is to
    find $(\bm{u}, p)$ in
    $\Omega$ such that \cref{eq:darcy_frac} is valid for all
    $i \in \mathcal{N}$, \cref{eq:darcy_coupling_disc} is valid for
    all $k \in \mathcal{I}$, and \cref{eq:two_codimensional} is valid for
    all $l \in \mathcal{P}$.
\end{problem}

\begin{remark}
    In the case that $\gamma_k$ is the one co-dimensional intersection of more
    than two fractures, coupling conditions \cref{eq:darcy_coupling_cont} can be
    generalized extending the summation of the normal fluxes on all the
    fractures and assuming the pressure continuity at the intersection.  A
    similar consideration can be done for the coupling conditions
    \cref{eq:darcy_coupling_disc} where the summation is extended to all the
    fractures involved and the definition of $\FIX{\hat{\kappa}_k}{\hat{\lambda}_k}$
    should include the
    measure of the cross section of $\gamma_k$, now simply represented as a
    rectangle.
\end{remark}

\begin{remark}
    In the case that $\xi_l$ is the two co-dimensional intersection of more
    than two fractures, coupling conditions \cref{eq:two_codimensional} can be
    generalized assuming the pressure continuity at the intersection for
    all the one co-dimensional objects.
\end{remark}



\section{\FIX{}{Well posedness}} \label{sec:well_posedness}


\FIX{In this part we analyse the well posedness of the weak form of \cref{pb:cc} and
\cref{pb:dc}.}{In this part the weak formulation of the problem \cref{pb:cc} and
\cref{pb:dc} and the related functional spaces are introduced. We provide also
well posedness results for both problems.}


\subsection{\FIX{}{DFN Darcy flow - continuous coupling}}


We define $\norm{\cdot}_E : L^2(E) \rightarrow \mathbb{R}$
as the usual $L^2$-norm and $\left( \cdot, \cdot \right)_E : L^2(E) \times
L^2(E) \rightarrow \mathbb{R}$ which is the usual $L^2$-scalar product.  To simplify
the presentation we use this notation for both scalar and vector functions.
We indicate by $[\cdot]_i$ and $[\cdot]_{ij}$ the $i$-th, respectively $ij$-th,
component of the vector, respectively of the matrix, in the square brackets. We
assume also pressure boundary conditions assigned to each $\partial \Omega_i$ by
a function $\FIX{f_i}{g_i} \in H^{\frac{1}{2}}(\partial \Omega_i)$, a source term
$\FIX{q_i}{f_i} \in
L^2(\Omega_i)$, and $[\FIX{\kappa_i}{\lambda_i}]_{jk} \in L^\infty( \Omega_i) $.
Given a regular domain
$E \in \Omega_i$, we define the functional spaces $Q(E) \defeq L^2(E)$ for the scalar fields
and $V(E) \defeq \left\{ \bm{v} \in [L^2(E)]^3: \nabla \cdot \bm{v} \in
L^2(E)  \right\}$ for the vector fields, which are Hilbert spaces. It is worth to
remember that we are dealing with differential operators on the tangent planes of the
fractures. With an abuse
of notation we write $V_i = V(\Omega_i)$ and $Q_i = Q(\Omega_i)$ for each
fracture.  We endow $Q_i$ with the usual $L^2$-norm and $V_i$ with the $H_{\rm
div}$-norm on the tangent space, defined as $\norm{\bm{v}}_{V_i}^2 \defeq
\norm{\bm{v}}_{\Omega_i}^2 + \norm{\nabla \cdot \bm{v}}_{\Omega_i}^2$.  The
global spaces for the network $\Omega$ are defined as $Q(\Omega) \defeq
\prod_{i \in \mathcal{N}} Q_i$ with norm  and $V(\Omega) \defeq
\prod_{i \in \mathcal{N}} V_i$,
with an abuse of notation we indicate $V = V(\Omega)$ and $Q = Q(\Omega)$.
Their norms are defined as $\norm{\FIX{v}{q}}_{Q}^2 \defeq
\sum_{i \in \mathcal{N}} \norm{\FIX{v_i}{q_i}}_{Q_i}^2$ and
$\norm{\bm{v}}_{V}^2 \defeq
\sum_{i \in \mathcal{N}} \norm{\bm{v}_i}_{V_i}^2$. Following the standard
procedure we can derive the weak formulation of \cref{pb:cc} which requires the
definition of the following bilinear forms
\begin{gather*}
    a \left( \cdot, \cdot \right): V \times V \rightarrow \mathbb{R}:
    \quad a \left( \bm{u}, \bm{v} \right) \defeq \sum_{i \in \mathcal{N}}
    a_{i} \left( \bm{u}_i,
    \bm{v}_i \right), \,\,
    a_{i} \left( \bm{u}_i, \bm{v}_i \right) \defeq \left(
    \FIX{\kappa_i^{-1}}{\lambda_i^{-1}}
    \bm{u}_i, \bm{v}_i \right)_{\Omega_i}
    \\
    b \left( \cdot, \cdot \right): V \times Q \rightarrow \mathbb{R}:
    \quad b \left( \bm{u}, \FIX{v}{q} \right) \defeq \sum_{i \in \mathcal{N}}
    b_{i} \left( \bm{u}_i, \FIX{v_i}{q_i} \right),\,\,
    b_{i} \left( \bm{u}_i, \FIX{v_i}{q_i} \right) \defeq -
    \left( \nabla \cdot
    \bm{u}_i, \FIX{v_i}{q_i} \right)_{\Omega_i}
\end{gather*}
and the following functionals which include the boundary data of the problem and
the scalar source term
\begin{gather*}
    B \left(\cdot  \right) : V \rightarrow \mathbb{R}:
    \quad B \left( \bm{v} \right) \defeq - \sum_{i \in \mathcal{N}} \langle
    \bm{v}_i \cdot \bm{n}_{\partial \Omega_i}, \FIX{f_i}{g_i}
    \rangle_{\partial \Omega_i}\\
    F \left(\cdot  \right) : Q \rightarrow \mathbb{R}:
    \quad F \left( \FIX{v}{q} \right) \defeq - \sum_{i \in \mathcal{N}} \left(
    \FIX{v_i}{q_i}, \FIX{q_i}{f_i} \right)_{\Omega_i}
\end{gather*}
where the dual pairing is defined as $\langle
\cdot, \cdot \rangle_{\partial \Omega_i}: H^{-\frac{1}{2}}(\partial \Omega_i)
\times H^{\frac{1}{2}}(\partial \Omega_i) \rightarrow \mathbb{R}$ and
$\bm{n}_{\partial \Omega_i}$ is the unit normal pointing outward from $\partial
\Omega_i$ and tangent to $\Omega_i$. With this definitions we introduce the
following problem.
\begin{problem}[Weak formulation of \cref{pb:cc}]\label{pb:weak_cc} The weak formulation of
\cref{pb:cc} is to find $(\bm{u}, p) \in V \times Q$ such that
    \begin{gather}
        \begin{aligned}
            &a \left( \bm{u}, \bm{v} \right) + b \left( \bm{v}, p \right) = B \left(
            \bm{v} \right)
            & \forall \bm{v} \in V \\
            & b \left( \bm{u}, \FIX{v}{q} \right) = F \left( \FIX{v}{q} \right)
            & \forall \FIX{v}{q} \in Q
        \end{aligned}
    \end{gather}
\end{problem}
\begin{theorem}\label{teo:well_posed_cc}
    \cref{pb:weak_cc} is well posed.
\end{theorem}
\begin{proof}
    Using standard arguments, mainly Cauchy-Schwarz inequality, it is possible
    to prove the boundedness, and then the continuity, of the bilinear forms $a$
    and $b$ on their spaces as well as the functionals $F$ and $B$.  To show
    coercivity of $a$, on the kernel of $b$ we consider a function $\bm{w} \in
    V$ such that $b\left( \bm{w}, \FIX{v}{q} \right) = 0$ for all $\FIX{v}{q} \in Q$.  We then
    have $\nabla_{T_i} \cdot \bm{w}_i = 0$ almost everywhere in $\Omega_i$ and
    $\norm{ \bm{w} }_V^2 = \sum_{i \in \mathcal{N}} \norm{ \bm{w}_i
    }_{\Omega_i}^2$, thus
    \begin{gather*}
        a\left( \bm{w}, \bm{w} \right) = \sum_{i \in \mathcal{N}} \left(
        \FIX{\kappa_i^{-1}}{\lambda_i^{-1}} \bm{w}_i, \bm{w}_i \right)_{\Omega_i} \gtrsim
        \norm{\bm{w}}_V^2.
    \end{gather*}
    To simplify the proof we now consider only two fractures, being the case of multiple
    fractures a straightforward extension where the following techniques are
    properly applied to each fracture intersection.
    To verify the inf-sup condition on $b$ \FIX{,}{we can have two possibilities:
    either the intersection does not reach the boundary of the fracture or its
    contrary. In both cases,} let $q$ be a function in $Q$
    and consider the following auxiliary problem
    \begin{gather*}
        \begin{aligned}
            & - \nabla_{T_i} \cdot \nabla_{T_i} \varphi_i = q_i & \text{in }
            \Omega_i\\
            & \varphi_i = 0 & \text{on } \partial \Omega_i \cup \left\{ \gamma \cap
            \Omega_i \right\}
        \end{aligned},
    \end{gather*}
    which are decoupled problems, one for each fracture. \FIX{Following
    \cite{Dziuk1988} the previous problem admits a unique solution
    $\varphi_i \in H^2(\Omega_i)$ $ \forall i$, such that
    $\norm{\varphi_i}_{H^2(\Omega_i)} \lesssim \norm{q_i}_{\Omega_i}$.
    We consider now $\bm{v} \in V$ such that $\bm{v}_i =
    \nabla_{T_i} \varphi_i$, we have $- \nabla_{T_i} \cdot \bm{v}_i =
    q_i$ and}{If the
    intersection completely cuts $\Omega_i$ then, following \cite{Dziuk1988} the
    previous problem admits a unique solution on each disconnected part of
    $\Omega_i$, called
    $\Omega_i^*$. We have $\varphi_i^* \in H^2(\Omega_i^*)$ $ \forall i$, such
    that $\norm{\varphi_i^*}_{H^2(\Omega_i^*)} \lesssim \norm{q_i}_{\Omega_i^*}$.
    With an abuse of notation, indicating the $H^2$-broken norm with the same
    symbol, we obtain $\norm{\varphi_i}_{H^2(\Omega_i)} \lesssim
    \norm{q_i}_{\Omega_i}$. We consider now $\bm{v} \in V$ such that
    $\bm{v}|_{\Omega_i^*} =
    \nabla_{T_i} \varphi_i^*$, we have $- \nabla_{T_i} \cdot \bm{v}|_{\Omega_i^*} =
    q_i$ in $\Omega_i^*$ and}
    \begin{gather*}
        \norm{\bm{v}}_{V}^2 = \sum_{i \in \mathcal{N} }
        \norm{\nabla_{T_i} \varphi_i}^2_{\Omega_i} + \norm{q_i}^2_{\Omega_i}
        \lesssim \norm{q}_Q^2.
    \end{gather*}
    With this choice of $\bm{v}$ we obtain the boundedness from below of the
    bilinear form $b$
    \begin{gather*}
        b\left( \bm{v}, q \right) = \sum_{i \in \mathcal{N}} - \left(
        \nabla_{T_i} \cdot \bm{v}, q \right)_{\Omega_i} =
        \norm{ q }_Q^2 \gtrsim \norm{q}_Q \norm{\bm{v}}_V.
    \end{gather*}
    Following \cite{Brezzi1991} we conclude that \FIX{}{in the case where the intersection
        boundaries coincide with the boundary of the fracture planes},
    \cref{pb:weak_cc} is well posed.
    \FIX{}{We consider now the problem of partially immersed intersection, \ie,
    one ending part touches a fracture boundary and the other is immersed in the
    fracture. The fully immersed intersection is a straightforward extension,
    where the same technique is applied to both the ending parts.
    If the intersection is immersed in the fracture planes, following
    \cite{Dziuk1988,Grisvard1992} the previous auxiliary problem admits a unique
    solution $\varphi_i \in H^{\frac{3}{2}-\epsilon}(\Omega_i)$, $\epsilon > 0$
    and for all $i \in \mathcal{N}$. Define a regularized kernel $\theta \in
    {C}^\infty(T \Omega_i)$ with compact support in $T
    \Omega_i$, that is, the tangent space of $\Omega_i$, centred in the immersed
    part of $\gamma$. We introduce also the
    corresponding family of mollifiers $\theta_\zeta(x) = \zeta^{-n}
    \theta(x/\zeta)$ for $\zeta > 0$ and $n \in \mathbb{N}$. We consider now
    $\bm{v}_\zeta \in V \cap C^\infty(T \Omega_i)$ such that
    $[\bm{v}_\zeta]_j = \phi_\zeta \ast [\nabla_{T_i}
    \varphi_i]_j$ for $j=1,2$. We have
    \begin{gather*}
        \norm{\bm{v}_\zeta}_{V}^2 =
        \norm{\theta_\zeta \ast \nabla_{T_i} \varphi_i}^2_{\Omega_i} +
        \norm{ \nabla_{T_i} \cdot \theta_\zeta \ast \nabla_{T_i} \varphi_i}^2_{\Omega_i}
    \end{gather*}
    where, with an abuse of notation, the mollifier is applied to each
    component of $\bm{v}_i$. Using the properties of the mollifiers, we can
    bound the first term as
    \begin{gather*}
        \norm{\theta_\zeta \ast \nabla_{T_i} \varphi_i}_{\Omega_i} \leq
        \norm{\nabla_{T_i} \varphi_i}_{\Omega_i} \leq
        \norm{\varphi_i}_{H^1(\Omega_i)} \leq \norm{q}_{\Omega_i},
    \end{gather*}
    while the second term can be estimated as
    \begin{gather*}
        \norm{ \nabla_{T_i} \cdot \theta_\zeta \ast \nabla_{T_i}
        \varphi_i}_{\Omega_i} = \norm{ \theta_\zeta \ast \nabla_{T_i} \cdot
        \nabla_{T_i} \varphi_i}_{\Omega_i} \leq \norm{ \nabla_{T_i} \cdot
        \nabla_{T_i} \varphi_i}_{\Omega_i} = \norm{q}_{\Omega_i}.
    \end{gather*}
    We obtain $\norm{\bm{v}_\zeta}_{V} \lesssim \norm{q_i}_Q$.
    With this choice we have for a $\Omega_i$
    \begin{gather*}
        b_i\left( \bm{v}_\zeta, q_i \right) = - \left(
        \nabla_{T_i} \cdot \bm{v}_\zeta, q_i \right)_{\Omega_i} = - \left(
        \nabla_{T_i} \cdot \theta_\zeta \ast \nabla_{T_i} \varphi_i, q_i
        \right)_{\Omega_i} = \\
        = - \left(\theta_\zeta \ast
        \nabla_{T_i} \cdot  \nabla_{T_i} \varphi_i, q_i
        \right)_{\Omega_i} = \left(\theta_\zeta \ast q_i, q_i
        \right)_{\Omega_i}.
    \end{gather*}
    Using again the property of the mollifiers we obtain that for $n \rightarrow
    \infty$: $\bm{v}_\zeta \xrightarrow{L^2} \bm{v}$, with $\bm{v} \in V$, and
    $\theta_\zeta \ast q_i\xrightharpoonup{L^2}q_i $.
    Following \cite{Brezzi1991} we conclude also in this case that \cref{pb:weak_cc} is well
    posed.
    }
\end{proof}


\subsection{\FIX{}{DFN Darcy flow - discontinuous coupling}}

We consider now the functional setting to present the weak formulation of
problem \cref{pb:dc}. Referring to \cref{eq:darcy_coupling_inters} and
\cref{eq:darcy_coupling_discontinuous}, we
assume pressure boundary conditions assigned to each ending point $\partial
\gamma_k$ by a scalar $\FIX{\hat{f}_i}{\hat{g}_i} \in \mathbb{R}$, a source term defined in
$\FIX{\hat{q}_k}{\hat{f}_k} \in L^2(\gamma_k)$, an effective tangential
permeability in $\FIX{\hat{\kappa}_k}{\hat{\lambda}_k}
\in L^\infty( \gamma_k) $, and an effective normal permeability with regularity
$\FIX{\widetilde{\kappa}_m}{\widetilde{\lambda}_m} \in  L^\infty( \gamma_k)$.
Motivated by \eg \cite{Martin2005}, we introduce a new family of spaces,
one for each fracture $\Omega_i$ by
\begin{gather*}
    W_i \defeq \left\{\bm{v} \in V_i: \bm{v}^+ \cdot T_i \bm{n}_k |_{\gamma_k} \in L^2( \gamma_k )
    \text{ and } \bm{v}^- \cdot T_i \bm{n}_k |_{\gamma_k} \in L^2( \gamma_k ),
    \forall k \in \mathcal{G}_i \right\}
\end{gather*}
and their composition for the vector fields $W \defeq
\prod_{i \in \mathcal{N}} W_i$, which are Hilbert spaces endowed with norms
\begin{gather*}
    \norm{\bm{v}}_{W_i}^2 \defeq \norm{\bm{v}}_{V_i}^2 + \sum_{k \in \mathcal{G}_i}
    \norm{\bm{v}^+ \cdot T_i \bm{n}_k}_{\gamma_k}^2 +
    \norm{\bm{v}^- \cdot T_i \bm{n}_k}_{\gamma_k}^2
    \quad \text{ and } \quad \norm{\bm{v}}_W^2 \defeq \sum_{i \in \mathcal{N} } \norm{ \bm{v}_i
}_{W_i}^2.
\end{gather*}
It is worth to notice that we require more regularity on the intersections for
$W_i$ than $V_i$, where implicitly we assume $H^{-1/2}$-regularity, in order to
properly take into account the coupling conditions
\cref{eq:darcy_coupling_discontinuous}. This assumption is related to Robin-type
boundary conditions for problem in mixed form, see \cite{Martin2005} for a more
detailed investigation.  For the intersections we introduce the functional
spaces for both scalar fields as $\hat{Q}_k \defeq L^2(\gamma_k)$, with the
usual $L^2$-norm, and vector fields as
\begin{gather*}
    \hat{V}_k \defeq \left\{ \hat{\bm{v}} \in [ L^2( \gamma_k ) ]^3: \nabla \cdot
    \hat{\bm{v}} \in
    L^2(\gamma_k) \right\}
    \quad \text{with} \quad \norm{\hat{\bm{v}}}_{\hat{V}_k}^2 \defeq
    \norm{\hat{\bm{v}}}_{\gamma_k}^2 +
    \norm{\nabla \cdot \hat{\bm{v}}}_{\gamma_k}^2.
\end{gather*}
The global spaces are $\hat{V} \defeq \prod_{k \in \mathcal{I} } \hat{V}_k$,
with norm $\norm{\hat{\bm{v}}}_{\hat{V}}^2 \defeq \sum_{k \in \mathcal{I} }
\norm{\hat{\bm{v}}_k}_{\hat{V}_k}^2$ for the vector fields, and $\hat{Q} \defeq
\prod_{k \in \mathcal{I} } \hat{Q}_k$, with norm
$\norm{\FIX{\hat{v}}{\hat{q}}}_{\hat{Q}}^2 \defeq \sum_{k \in \mathcal{I}}
\norm{\FIX{\hat{v}_k}{\hat{q}_k}}^2_{\hat{Q}_k}$ for the scalar fields. Finally
the spaces for the coupled problem are $U \defeq W \times \hat{V}$, with induced
norm from $W$ and $\hat{V}$, and $O \defeq Q \times \hat{Q}$, with induced norm
from $Q$ and $\hat{Q}$. All the aforementioned spaces are Hilbert spaces. We
introduce the bilinear forms associated with the intersections as
\begin{gather*}
    \hat{a} \left( \cdot, \cdot \right): \hat{V} \times \hat{V} \rightarrow \mathbb{R}:
    \quad \hat{a} \left( \hat{\bm{u}}, \hat{\bm{v}} \right) \defeq
    \sum_{k \in \mathcal{I}}
    \hat{a}_{k} \left( \hat{\bm{u}}_k,
    \hat{\bm{v}}_k \right), \,\,
    \hat{a}_{k} \left( \hat{\bm{u}}_k, \hat{\bm{v}}_k \right) \defeq
    \left( \FIX{\hat{\kappa}_k^{-1}}{\hat{\lambda}_k^{-1}}
    \hat{\bm{u}}_k, \hat{\bm{v}}_k \right)_{\gamma_k}
    \\
    \hat{b} \left( \cdot, \cdot \right): \hat{V} \times \hat{Q} \rightarrow \mathbb{R}:
    \quad \hat{b} \left( \hat{\bm{u}}, \FIX{\hat{v}}{\hat{q}} \right) \defeq
    \sum_{k \in \mathcal{I}}
    \hat{b}_{k} \left( \hat{\bm{u}}_k, \FIX{\hat{v}_k}{\hat{q}_k} \right),\,\,
    \hat{b}_{k} \left( \hat{\bm{u}}_k, \FIX{\hat{v}_k}{\hat{q}_k} \right) \defeq -
    \left( \nabla \cdot
    \hat{\bm{u}}_k, \FIX{\hat{v}_k}{\hat{q}_k} \right)_{\gamma_k}.
\end{gather*}
The global bilinear forms for the coupled problem are
$\alpha \left( \cdot, \cdot \right): U \times U \rightarrow
\mathbb{R}$ and $\beta \left( \cdot, \cdot \right): U \times O
\rightarrow \mathbb{R}$
\begin{gather*}
    \alpha \left( \left( \bm{u}, \hat{\bm{u}} \right), \left( \bm{v},
    \hat{\bm{v}} \right) \right)
    \defeq
    a \left( \bm{u},
    \bm{v} \right) +
    \hat{a} \left( \hat{\bm{u}}, \hat{\bm{v}} \right)+
    cc^+_1( \bm{u}, \bm{v} ) + cc^-_1( \bm{u}, \bm{v} )
    \\
    \beta \left( \left( \bm{u}, \hat{\bm{u}} \right), \left( \FIX{v}{q},
    \FIX{\hat{v}}{\hat{q}} \right)
    \right) \defeq
    b \left(
    \bm{u}, \FIX{v}{q} \right)
    +\hat{b} \left( \hat{\bm{u}}, \FIX{\hat{v}}{\hat{q}} \right)
    + cc_2( \bm{u}, \FIX{\hat{v}}{\hat{q}} )
\end{gather*}
where the bilinear forms associated with the coupling conditions are
\begin{gather} \label{eq:cc_weak}
    cc^+_1\left( \cdot, \cdot \right) : V \times V \rightarrow \mathbb{R}:
    \quad
    cc^+_1( \bm{u}, \bm{v} ) \defeq
    \sum_{i \in \mathcal{N}} \sum_{j \in \mathcal{G}_i}
    \left( \FIX{\widetilde{\kappa}_i^{-1}}{\widetilde{\lambda}_i^{-1}}
    \bm{u}_i^+ \cdot T_i \bm{n}_j, \bm{v}_i^+ \cdot T_i \bm{n}_j
    \right)_{\gamma_j}\\
    cc_2\left( \cdot, \cdot \right) : V \times \hat{Q} \rightarrow
    \mathbb{R}:\quad
    cc_2( \bm{u}, \FIX{\hat{v}}{\hat{q}} ) \defeq \sum_{i \in \mathcal{N}} \sum_{j \in
    \mathcal{G}_i}
    \left( \jump{\bm{u}_i \cdot T_i \bm{n}_j }{\gamma_j},
    \FIX{\hat{v}_j}{\hat{q}_j}\right)_{\gamma_j},
\end{gather}
$cc_1^-(\cdot, \cdot)$ follow immediatly.
The functional $B$ is extended naturally on $U$ and we define
\begin{gather*}
    \Phi: O \rightarrow \mathbb{R}:\quad \Phi \left( \left( \FIX{v}{q},
    \FIX{\hat{v}}{\hat{q}} \right)
    \right) \defeq -
    \sum_{i \in \mathcal{N}} \left( \FIX{v_i}{q_i}, \FIX{q_i}{f_i} \right)_{\Omega_i} -
    \sum_{k \in \mathcal{I}} \left( \FIX{\hat{v}_k}{\hat{q}_k},
    \FIX{\hat{q}_k}{\hat{f}_k} \right)_{\gamma_k}
\end{gather*}
\begin{problem}[Weak formulation of \cref{pb:dc}]\label{pb:weak_dc} The weak formulation of
\cref{pb:dc} is to find $(\left( \bm{u}, \hat{\bm{u}} \right), \left( p, \hat{p} \right))
\in U \times O$ such that
    \begin{gather}
        \begin{aligned}
            &\alpha  \left( \left( \bm{u}, \hat{\bm{u}} \right), \left( \bm{v},
            \hat{\bm{v}} \right) \right) + \beta \left( \left( \bm{v},
            \hat{\bm{v}} \right), \left( p, \hat{p} \right)
            \right) = B \left(\left( \bm{v}, \hat{\bm{v}} \right)
            \right)
            & \forall \left( \bm{v}, \hat{\bm{v}} \right) \in U \\
            & \beta \left( \left( \bm{u}, \hat{\bm{u}} \right),
            \left( \FIX{v}{q}, \FIX{\hat{v}}{\hat{q}} \right) \right)
            = \Phi \left( \left( \FIX{v}{q},
            \FIX{\hat{v}}{\hat{q}}  \right)\right)
            & \forall \left( \FIX{v}{q}, \FIX{\hat{v}}{\hat{q}} \right) \in O
        \end{aligned}
    \end{gather}
\end{problem}
\begin{theorem}
    \cref{pb:weak_dc} is well posed.
\end{theorem}
\begin{proof}
    Using standard arguments, mainly Cauchy-Schwarz inequality, it is possible
    to prove the boundedness, and then the continuity, of the bilinear forms
    $\alpha$
    and $\beta$ on their spaces as well as the functionals $\Phi$ and $B$.
    Considering a function
    $\left( \bm{w}, \hat{\bm{w}} \right) \in U$ such that $\beta
    \left( \left( \bm{w}, \hat{\bm{w}} \right), \left( \FIX{v}{q},
    \FIX{\hat{v}}{\hat{q}} \right)
    \right) = 0$ for all $\left(\FIX{v}{q}, \FIX{\hat{v}}{\hat{q}} \right) \in O$, we have
    $\nabla_{T_i} \cdot \bm{w}_i = 0$ almost everywhere in $\Omega_i$ and $
    \nabla \cdot \hat{\bm{w}}_k = 0$ almost everywhere in $\gamma_k$.  The norm
    of $\left( \bm{w}, \hat{\bm{w}} \right)$ becomes
    \begin{gather*}
        \norm{ \left( \bm{w}, \hat{\bm{w}}\right) }_{U}^2 =
        \sum_{i \in \mathcal{N}} \norm{ \bm{w}_i }_{\Omega_i}^2 + \sum_{k \in
        \mathcal{G}_i}
        \norm{\bm{w}^+_i \cdot T_i \bm{n}_k}_{\gamma_k}^2 +
        \norm{\bm{w}^-_i \cdot T_i \bm{n}_k}_{\gamma_k}^2 +
        \sum_{k \in \mathcal{I}} \norm{ \hat{\bm{w}}_k }_{\gamma_k}^2.
    \end{gather*}
    It is possible to show that $\alpha \left( \left( \bm{w}, \hat{\bm{w}}
    \right), \left( \bm{w}, \hat{\bm{w}} \right) \right) \gtrsim \norm{ \left(
    \bm{w}, \hat{\bm{w}} \right) }_{U}^2$, thus $\alpha
    $ is coercive on the kernel of $\beta$.

    Given a function $(q,
    \hat{q} ) \in O$, taking inspiration from \cite{Martin2005, Formaggia2012}, we
    introduce the following auxiliary problems for fractures
    \begin{gather*}
        \begin{aligned}
            &-\nabla_{T_i} \cdot \nabla_{T_i} \varphi_i = q_i & \text{in }
            \Omega_i\\
            &\nabla_{T_i} \varphi_i^+ \cdot T_i \bm{n}_j =
            \hat{q}_j &
            \text{on } \gamma_j, j \in \mathcal{G}_i\\
            &\nabla_{T_i} \varphi_i^- \cdot T_i \bm{n}_j = -
            \hat{q}_j &
            \text{on } \gamma_j, j \in \mathcal{G}_i\\
            & \jump{\varphi_i}{\gamma_j} = 0 &
            \text{on } \gamma_j, j \in \mathcal{G}_i\\
            & \varphi_i = 0 & \text{on } \partial \Omega_i
        \end{aligned}
    \end{gather*}
    and for the intersections
    \begin{gather*}
        \begin{aligned}
            & - \nabla \cdot \nabla \hat{\varphi}_k = \hat{q}_k & \text{in }
            \gamma_k \\
            & \hat{\varphi}_k = 0 & \text{on } \partial \gamma_k
        \end{aligned}.
    \end{gather*}
    \FIX{}{We consider only the case where the intersection reaches the fracture
    boundary; intersection endpoints inside the fracture plane can be handled
    similar to the proof of Theorem
    \ref{teo:well_posed_cc}.}
    Following \cite{Dziuk1988} the previous problems admits a unique solution
    \FIX{}{on each disconnected part of $\Omega_i$, indicated by $\Omega_i^*$}.
    In the first case we have \FIX{$\varphi_i \in H^2(\Omega_i)$}{
    $\varphi_i^* \in H^2(\Omega_i^*)$} $\forall i$, such that \FIX{
    $\norm{\varphi_i}_{H^2(\Omega_i)} \lesssim \norm{q_i}_{\Omega_i} + \sum_{j
    \in \mathcal{G}_i}
    \norm{\hat{q}_j}_{\gamma_j}$}{$\norm{\varphi_i^*}_{H^2(\Omega_i^*)}
    \lesssim \norm{q_i}_{\Omega_i^*} + \sum_{j\in \mathcal{G}_i}
    \norm{\hat{q}_j}_{\gamma_j}$}.
    We consider now $\bm{v} \in W$ such that \FIX{$\bm{v}_i =
    \nabla_{T_i} \varphi_i$}{$\bm{v}|_{\Omega_i^*} = \nabla_{T_i} \varphi_i^*$},
    we have \FIX{$- \nabla_{T_i} \cdot \bm{v}_i =
    q_i$}{$- \nabla_{T_i} \cdot \bm{v}|_{\Omega_i^*} = q_i$ in $\Omega_i^*$},
    \FIX{$\bm{v}_i^+ \cdot T_i \bm{n}_j = \hat{q}_j$ and $\bm{v}_i^-
    \cdot T_i
    \bm{n}_j = \hat{q}_j$}{$\bm{v}|_{\Omega_i^*} \cdot T_i \bm{n}_j = \hat{q}_j$}
    for all $j \in \mathcal{G}_i$, and
    \begin{gather*}
        \norm{\bm{v}_i}_{W_i}^2 =
        \norm{\nabla_{T_i} \varphi_i}^2_{\Omega_i} + \norm{q_i}^2_{\Omega_i}
        + 2 \sum_{j \in \mathcal{G}_i} \norm{\hat{q}_j}_{\gamma_j}^2
        \lesssim \norm{q_i}_{Q_i}^2 + \sum_{j \in \mathcal{G}_i}
        \norm{\hat{q}_j}_{Q_j}^2
        \leq\norm{q}_{Q}^2 + \norm{\hat{q}}_{\hat{Q}}^2.
    \end{gather*}
    For the second family of problems we obtain the existence of
    $\hat{\varphi}_k \in H^2( \gamma_k )$ $\forall k$, such that
    $\norm{\hat{\varphi}_k}_{H^2( \gamma_k )} \lesssim \norm{
    \hat{q}_k}_{\gamma_k}$.
    Considering $\hat{\bm{v}} \in \hat{V}$ such that $\hat{\bm{v}}_k = \nabla
    \hat{\varphi}_k$, we have $- \nabla \cdot \hat{\bm{v}}_k =
    \hat{q}_k$
    and
    \begin{gather*}
        \norm{\hat{\bm{v}}}_{\hat{V}}^2 = \sum_{k \in \mathcal{I}}
        \norm{\nabla_{T_i} \hat{\varphi}_k}^2_{\gamma_k} +
        \norm{\hat{q}_k}^2_{\gamma_k}
        \lesssim \norm{\hat{q}}_{\hat{Q}}^2.
    \end{gather*}
    With this choice of $\bm{v}$ and $\hat{\bm{v}}$ we obtain the boundedness
    from below of the bilinear form $\beta$
    \begin{gather*}
        \beta \left( \left( \bm{v}, \hat{\bm{v}} \right), \left( q,
        \hat{q} \right) \right) =
        \sum_{i \in \mathcal{N}} \norm{q_i}_{Q_i}^2
        + 2\sum_{j \in \mathcal{G}_i} \norm{\hat{q}_j}_{\gamma_j}^2
        + \sum_{k \in \mathcal{I}} \norm{\hat{q}_k}_{\gamma_k}^2
        \gtrsim \norm{(q, \hat{q})}_{O} \norm{(\bm{v},\hat{\bm{v}})}_{U}.
    \end{gather*}
    Thus the inf-sup condition is fulfilled, and following \cite{Brezzi1991}
    we conclude that \cref{pb:weak_dc} is well posed.
\end{proof}





\section{\FIX{Discrete Model}{Discrete approximation}} \label{sec:discrete}


In this part we present the numerical discretization of \cref{pb:weak_cc} and
\cref{pb:weak_dc}. We extend the virtual element method for mixed problem
presented in
\cite{Brezzi2014,BeiraodaVeiga2014b,BeiraoVeiga2016} to the DFN setting.
For simplicity we start the discretization of \cref{pb:weak_cc} considering only
the network composed by a single fracture $\Omega = \Omega_i$, the extension of more
than one fracture is trivial but requires \FIX{an}{a} heavier notation, while the
approximation of \cref{eq:darcy_coupling_cont} will be discussed below.
For realistic fracture networks, the solutions will commonly have low
regularity due to heterogeneities, and we therefore limit ourselves to lowest
order methods.


\subsection{\FIX{}{Discrete setting}}


Let us consider $\mathcal{T}(\Omega)$ a tessellation of a single fracture
$\Omega$ into non-overlapping polygons $E$ and the union of the edges $e$ as
$\mathcal{E}(\Omega) = \left\{ e \in \partial E \right\}$. We denote by $h_E$
the diameter of $E$, by $\bm{x}_E$ the centre of $E$, by $\bm{n}_e$ with $e \in
\mathcal{E}(E)$ the unit normal of $e$ pointing outward with respect to the
internal part of $E$, and by $h$ the maximum diameter in $\mathcal{T}(\Omega)$.
When there is no room for confusion we write $\mathcal{T}$ instead of
$\mathcal{T}(\Omega)$ and $\mathcal{E}$ instead of $\mathcal{E}(\Omega)$. We
assume that there exist $\rho_E \in \mathbb{R}^+$ such that $E$ is star-shaped
with respect to every point of a disc of radius $\rho_E h_E$ and $\abs{e} \geq
\rho_E h_E$, for all $e \in \mathcal{E}(E)$. An example of elements is presented
in \cref{fig:example_cells}.
\begin{figure}[tbp]
    \centering
    \resizebox{0.35\textwidth}{!}{\fontsize{30pt}{8}\selectfont%
\begingroup%
  \makeatletter%
  \providecommand\color[2][]{%
    \errmessage{(Inkscape) Color is used for the text in Inkscape, but the package 'color.sty' is not loaded}%
    \renewcommand\color[2][]{}%
  }%
  \providecommand\transparent[1]{%
    \errmessage{(Inkscape) Transparency is used (non-zero) for the text in Inkscape, but the package 'transparent.sty' is not loaded}%
    \renewcommand\transparent[1]{}%
  }%
  \providecommand\rotatebox[2]{#2}%
  \ifx\svgwidth\undefined%
    \setlength{\unitlength}{563.86801758bp}%
    \ifx\svgscale\undefined%
      \relax%
    \else%
      \setlength{\unitlength}{\unitlength * \real{\svgscale}}%
    \fi%
  \else%
    \setlength{\unitlength}{\svgwidth}%
  \fi%
  \global\let\svgwidth\undefined%
  \global\let\svgscale\undefined%
  \makeatother%
  \begin{picture}(1,0.94147246)%
    \put(0,0){\includegraphics[width=\unitlength,page=1]{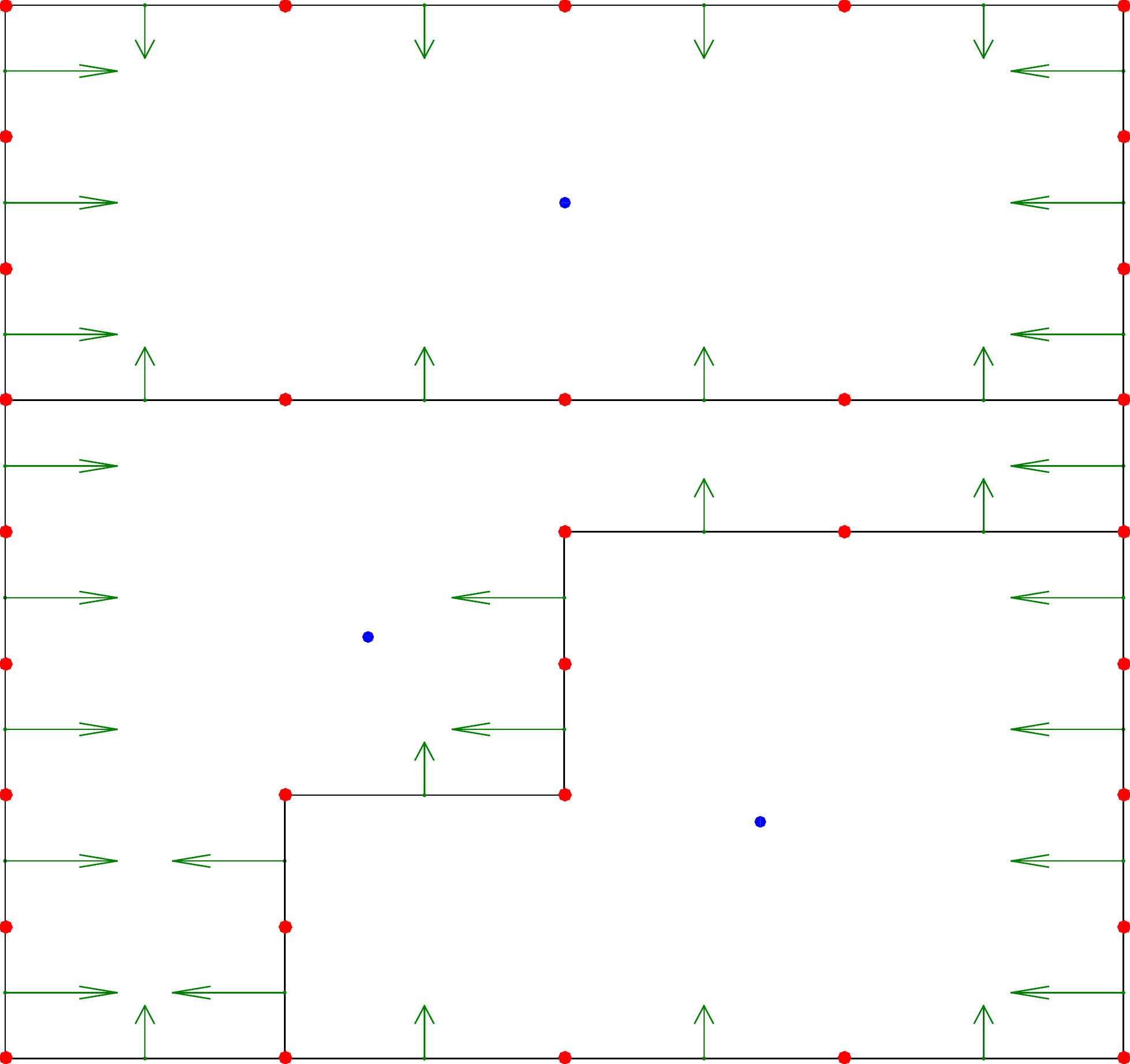}}%
    \put(0.46129512,0.78054506){\color[rgb]{0,0,0}\makebox(0,0)[lb]{\smash{$E_1$}}}%
    \put(0.28400788,0.39815937){\color[rgb]{0,0,0}\makebox(0,0)[lb]{\smash{$E_2$}}}%
    \put(0.62435359,0.22798643){\color[rgb]{0,0,0}\makebox(0,0)[lb]{\smash{$E_3$}}}%
  \end{picture}%
\endgroup%
    }%
    \caption{Example of three elements admitted in the discretization. The blue
    points represent the centre $\bm{x}_E$ for each element $E$, the red points define
    the edges $e$, and the green arrows are the normal vector to the edges. $E_1$ is a
    rectangular element but composed of 14 edges and $E_2$ is a non-convex
    element.}%
    \label{fig:example_cells}
\end{figure}
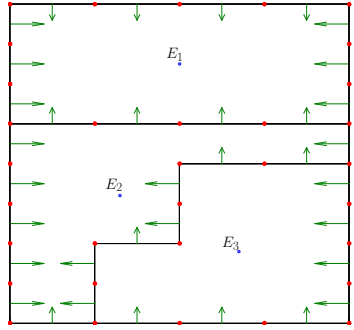
We introduce the
following local spaces for the element $E$ of $\mathcal{T}$: for the scalar
fields $Q_h(E) \defeq \left\{ \FIX{v}{q} \in Q(E): \FIX{v}{q} \in \mathbb{P}_0(E)
\right\}$ and for the vector fields
\begin{gather*}
    V_h(E) \defeq \left\{ \bm{v} \in V(E): \bm{v} \cdot \bm{n}_{e_i} |_{e_i} \in
    \mathbb{P}_0(e_i) \forall e_i \in \mathcal{E}(E), \nabla \cdot \bm{v}
    \in \mathbb{P}_0(E), \nabla \times \bm{v} = \bm{0} \right\}.
\end{gather*}
It is worth to notice that we do not reconstruct the velocity inside an element,
since in most of the applications the normal flux from a face is the physical
quantity of interest. \FIX{However $\bm{v}$ can be explicitly computed using the last
request of $V_h(E)$.}{} The global spaces for the fracture $\Omega$ are defined by
\begin{gather*}
    Q_h(\Omega) \defeq \left\{ \FIX{v}{q} \in Q(\Omega): \FIX{v|_E}{q|_E}
    \in Q_h(E) \, \forall E \in
    \mathcal{T} \right\} \\
    V_h(\Omega) \defeq \left\{ \bm{v} \in V(\Omega): \bm{v}|_E \in V_h(E) \, \forall E \in
    \mathcal{T} \right\}.
\end{gather*}
The degrees of freedom, simply \dofs\!\!, for $Q_h(\Omega)$ are constant value
in each element of the mesh, while for $\bm{v} \in V_h(\Omega)$ we consider
\begin{gather*}
    \left( \bm{v} \cdot \bm{n}_e, q \right)_e \quad \forall e \in \mathcal{E}
    \text{ and } \forall q \in \mathbb{P}_0(e).
\end{gather*}
We indicate by $N_{\dof} = \sharp \mathcal{E}(E)$ the number of \dofs for a generic element $E$.
To simplify the presentation, we construct the approximation
locally to each element $E \in \mathcal{T}$ and then use the fact that the
velocity \dofs are single valued for all $e \in \mathcal{E}$, the global
approximation can be build up. When it is
clear from the context, we consider the bilinear form introduced in the previous
section restricted to an element of the mesh.  For the element $E \in
\mathcal{T}$ we define the local canonical base for both spaces: for $Q_h(E)$
is trivial, while for $V_h(E)$ we indicate with $\base(V_h(E)) = \left\{
\bm{\varphi}_\omega \right\}_{\omega=1}^{N_\dof}$ such that $\bm{\varphi}_\omega \in V_h(E)$
and given $\bm{v} \in V_h(E)$ we have $\bm{v} = \sum_{\omega=1}^{N_\dof}
[\bm{v}]_\omega
\bm{\varphi}_\omega$, with $[\bm{v}]_\omega \in \mathbb{R}$. As a consequence of these
facts we have $\left( \bm{\varphi}_\omega \cdot \bm{n}_{e_j}, 1 \right)_{e_j} =
\delta_{\omega j}$ which means $\abs{e_\omega} \bm{\varphi}_\omega \cdot \bm{n}_{e_\omega} =
1$, as
well as $\left( \nabla \cdot \bm{\varphi}_\omega, 1 \right)_E = 1$ which means
$\abs{E}\nabla \cdot \bm{\varphi}_\omega = 1$.
With these choices, it is possible
to compute directly the bilinear form $b(\cdot, \cdot)$ as well as the
functionals $F(\cdot)$ and $B(\cdot)$. However the computation of $a(\cdot,
\cdot)$ is more tricky and requires the introduction of a projection operator
$\Pi_0$, which gives an approximation of $a(\cdot, \cdot) \approx a_h(\cdot,
\cdot)$. We will see that the explicit evaluation of the basis functions
$\bm{\varphi}_i$ is not needed in the internal part of $E$, but the request that
the order of convergence is preserved when using $a_h(\cdot, \cdot)$ instead of
$a(\cdot, \cdot)$. We introduce the following local space
\begin{gather*}
    {\mathcal{V}}(E) \defeq \left\{ \bm{v} \in V_h(E): \bm{v} =
    \FIX{\kappa}{\lambda} \nabla v,
    \text{ for } v \in \mathbb{P}_1(E) \right\},
\end{gather*}
where $\nabla$ stands for the tangential gradient on the current fracture.
We define the projection operator $\Pi_0: V(E) \rightarrow \mathcal{V}(E)$ such
that $a( \bm{v} - \Pi_0 \bm{v}, \bm{w} ) = 0 $ for all $\bm{w} \in
\mathcal{V}(E)$. To introduce the discrete bilinear forms, we start by having
for $\bm{u}, \bm{v} \in V_h(E)$
\begin{gather} \label{eq:a_split}
    a( \bm{u}, \bm{v} ) = a( \Pi_0 \bm{u}, \Pi_0 \bm{v} ) + a( (I-\Pi_0) \bm{u},
    (I-\Pi_0) \bm{v} ),
\end{gather}
since the cross terms are zero due to $\Pi_0$.


\subsection{\FIX{}{Computation of the local $H_{\rm div}$-mass matrix}}


We focus now on the approximation of the first term \FIX{}{in \eqref{eq:a_split}}.
Since $\Pi_0 \bm{u} \in
V_h(E)$, it can be expanded in the canonical base $\Pi_0 \bm{u} = \sum_{\omega =
1}^{N_\dof} [\Pi_0 \bm{u}]_\omega \bm{\varphi}_\omega$ with $\bm{\varphi}_\omega
= \FIX{\kappa}{\lambda} \nabla \varphi_\omega$, $\varphi_\omega \in \mathbb{P}_1(E)$.  The
trial function $\Pi_0 \bm{v} \in \mathcal{V}$ then $\exists v \in \mathbb{P}_1(E)$
such that $\Pi_0 \bm{v} = \FIX{\kappa}{\lambda} \nabla v$.  Moreover we can write
$\varphi_\omega$, as well as $v$, using a monomial expansion as
\begin{gather*}
    \varphi_\omega = \sum_{m=1}^2 s_\omega^m m_m
    \quad \text{with} \quad m_m( \bm{x} ) \defeq \dfrac{[\bm{x}]_m -
    [\bm{x}_E]_m}{h_E} \quad \text{and} \quad s_\omega^m \in \mathbb{R}.
\end{gather*}
It is worth to notice that the coordinates in the definition of the monomials
are defined in the tangential space of $\Omega$. The reason why $m_0(\bm{x}) = 1$
is not included will become clear below. We obtain for $\omega$
\begin{gather*}
    \left( \FIX{\kappa^{-1}}{\lambda^{-1}} \bm{\varphi}_\omega, \bm{v}
    \right)_E = - \left( \nabla \cdot
    \bm{\varphi}_\omega, v \right)_E + \sum_{e \in \mathcal{E}(E)} \left(
    \bm{\varphi} \cdot \bm{n}_e, v \right)_e = - \dfrac{1}{\abs{E}} \left(1,
    v\right)_E  + \dfrac{1}{\abs{e_\omega}}\left( 1, v\right)_{e_\omega},
\end{gather*}
which is computable using the monomial expansion of $v$. \FIX{It is possible to use
directly the monomial expansion, for both $\varphi_\omega$ and $v$, and}
{Indeed, by using the monomial expansion for both $\varphi_\omega$ and $v$, we} obtain
\begin{gather*}
    \left( \FIX{\kappa^{-1}}{\lambda^{-1}} \bm{\varphi}_\omega,
    \FIX{\kappa}{\lambda} \nabla m_j \right)_E =
    \left( \nabla \varphi_\omega,  \FIX{\kappa}{\lambda} \nabla m_j \right)_E =
    \sum_{i=1}^2 s_\omega^i \left( \nabla m_i, \FIX{\kappa}{\lambda}
    \nabla m_j \right)_E
    \quad \text{for} \quad j =1,2
\end{gather*}
the scalar product in the previous expression is computable. Defining the
matrices and vectors $G \in \mathbb{R}^{2\times 2}$, $\bm{f}_\omega,
\bm{s}_\omega \in \mathbb{R}^2$, $F \in \mathbb{R}^{2 \times N_\dof}$, and
$\Pi^* \in \mathbb{R}^{2 \times N_\dof}$ as
\begin{gather*}
    [G]_{ij} \defeq \left( \FIX{\kappa}{\lambda} \nabla m_i, \nabla m_j \right)_E, \quad
    [\bm{f}_\omega]_i \defeq - \dfrac{1}{\abs{E}} \left(1,
    m_i\right)_E  + \dfrac{1}{\abs{e_\omega}}\left( 1, m_i\right)_{e_\omega},\\
    \bm{s}_\omega \defeq [s^1_\omega, s^2_\omega ]^\top = G^{-1} \bm{f}_\omega, \quad
    F \defeq [ \bm{f}_1 | \ldots | \bm{f}_{N_\dof}], \quad
    \Pi^* \defeq G^{-1} F,
\end{gather*}
we have $s_\omega^k = [\Pi^*]_{\omega k}$. \FIX{It is now possible
to given the final expression of the term}{The final expression thus reads}
\begin{gather*}
    a( \Pi_0 \bm{\varphi}_\omega, \Pi_0 \bm{\varphi}_\theta ) = \sum_{i =1
    }^{N_\dof} \sum_{j=1}^{N_\dof} s_\omega^i s_\theta^j \left(
    \FIX{\kappa}{\lambda} \nabla
    m_i, \nabla m_j \right)_E = \sum_{i = 1
    }^{N_\dof} \sum_{j=1}^{N_\dof} [\Pi^*]_{\omega i} [G]_{ij} [\Pi^*]_{\theta
    j}=\\ =[ (\Pi^*)^\top  G \Pi^* ]_{\omega \theta}.
\end{gather*}
\FIX{}{For the second, stability, term in \eqref{eq:a_split},}
we follow \cite{Brezzi2014,BeiraodaVeiga2014b,BeiraoVeiga2016} and approximate the term as
\begin{gather*}
    a( (I-\Pi_0) \bm{\varphi}_\omega, (I-\Pi_0) \bm{\varphi}_\theta ) \approx
    s \left( (I-\Pi_0) \bm{\varphi}_\omega,
    (I-\Pi_0) \bm{\varphi}_\theta  \right)  \defeq \\ \defeq \varsigma
    \sum_{i=1}^{N_\dof} \dof_i\left( (I-\Pi_0) \bm{\varphi}_\omega \right)
    \dof_i\left( (I-\Pi_0) \bm{\varphi}_\theta \right),
\end{gather*}
where $s(\cdot, \cdot): V_h(E) \times V_h(E) \rightarrow \mathbb{R}$ is the
bilinear form associated to the stabilization and $\varsigma \in \mathbb{R}$ is a
suitable parameter that will be explained later on.  We give an expression of
$\dof_i\left( (I-\Pi_0) \bm{\varphi}_\omega \right)$. Since $\mathcal{V} \subset
V_h(E)$ we expand each projected element of the canonical base on the canonical
base itself, \ie $\Pi_0 \bm{\varphi}_\omega = \sum_{i=1}^{N_\dof} \pi_{\omega}^i
\bm{\varphi}_i $ for $\omega=1,\ldots,N_\dof$ with $\pi_\omega^j = \dof_j( \Pi_0
\bm{\varphi}_\omega )$, which is also
\begin{gather*}
    \Pi_0 \bm{\varphi}_\omega = \sum_{j=1}^2 s_\omega^j \FIX{\kappa}{\lambda}
    \nabla m_j =
    \sum_{j=1}^2
    s_\omega^j \sum_{i=1}^{N_\dof} \dof_i(\FIX{\kappa}{\lambda}
    \nabla m_j) \bm{\varphi}_i =
    \sum_{i=1}^{N_\dof}
    \left( \sum_{j=1}^2 s_\omega^j \dof_i( \FIX{\kappa}{\lambda}
    \nabla m_j) \right)
    \bm{\varphi}_i,
\end{gather*}
obtaining $\pi_\omega^i$ equal to the term in the brackets. This is computable
introducing the matrix $D \in \mathbb{R}^{2\times 2}$, with $[D]_{ij} \defeq
\dof_i (\FIX{\kappa}{\lambda} \nabla m_j)$, as $\pi_\omega^i = [D \Pi^*]_{i\omega}$. For the
stabilization we finally obtain $s \left( (I-\Pi_0) \bm{\varphi}_\omega,
(I-\Pi_0) \bm{\varphi}_\theta  \right) = \varsigma [ ( I -
D\Pi^*)^\top ( I - D \Pi^*)]_{\omega \theta}$. Introducing the bilinear form
$a_h(\cdot, \cdot): V_h(E) \times V_h(E) \rightarrow \mathbb{R}$ as
\begin{gather*}
    a_h( \bm{\varphi}_\omega, \bm{\varphi}_\theta) \defeq a\left(
    \Pi_0\bm{\varphi}_\omega, \Pi_0 \bm{\varphi}_\theta\right) +
    s\left( (I-\Pi_0) \bm{\varphi}_\omega, (I-\Pi_0) \bm{\varphi}_\theta \right)
    = \\ = \left[ (\Pi^*)^\top  G \Pi^* + \varsigma ( I -
    D\Pi^*)^\top ( I - D \Pi^*) \right]_{\omega \theta},
\end{gather*}
the local approximation of \cref{pb:weak_cc} for a single fracture is
\begin{problem}[Local discrete formulation of \cref{pb:weak_cc}]\label{pb:disc_weak_cc}
    the discrete approximation of the weak problem in $E$ is find $\left(\bm{u},
    p \right) \in V_h(E) \times Q_h(E)$ such that
    \begin{gather*}
        \begin{aligned}
            &a_h \left( \bm{u}, \bm{\varphi}_\theta \right) +
            b\left( p, \bm{\varphi}_\theta \right) = - \left( \bm{\varphi}_\theta
            \cdot \bm{n}_{\partial \Omega}, \FIX{f}{g} \right)_{\partial \Omega \cap
            e_\theta}
            \quad \forall \bm{\varphi}_\theta \in \base(V_h(E)) \\
            &b \left( 1, \bm{u} \right) = - (\FIX{q_i}{f_i}, 1)_E
        \end{aligned}.
    \end{gather*}
\end{problem}


\subsection{\FIX{}{Fracture intersection}}


We consider two intersecting fractures $\Omega_i$ and $\Omega_j$ such that
$\gamma_k = \Omega_i \cap \Omega_j$. The general case is
extendible with an analogous procedure. We discretize
$\gamma_k$ as the union of consecutive edges of $\mathcal{E}(\Omega_i) \cap
\mathcal{E}(\Omega_j)$, indicated by $\mathcal{T}(\gamma_k)$. For each fracture
we double the velocity \dofs in $\mathcal{T}(\gamma_k)$. We enforce
\cref{eq:darcy_coupling_cont} by using Lagrange multipliers along
$\mathcal{T}(\gamma_k)$, one for each edge involved.

If two intersecting fractures are present, then the coupled model \cref{pb:dc}
can be considered, and we need to introduce a proper
discretization also for the intersection as well as for the coupling condition
\cref{eq:darcy_coupling_discontinuous}. The procedure for the former is similar
to the derivation of the discrete system for the fractures but in a
mono-dimensional framework, \ie in $\mathcal{T}(\gamma)$. For simplicity we
consider a single two-codimensional object indicated by $\gamma$, the
extension to multiple intersections is trivial but requires \FIX{an}{a} heavier notation.
The approximation of \cref{eq:two_codimensional} will be discussed below.
We consider the discrete
spaces $\hat{Q}_h(E) \defeq \left\{ \FIX{\hat{v}}{\hat{q}} \in \hat{Q}(E): \,
\FIX{\hat{v}}{\hat{q}} \in
\mathbb{P}_0(E) \right\}$ with $E \in \mathcal{T}(\gamma)$ for the pressure,
for the velocity we have
\begin{gather*}
    \hat{V}_h(E) \defeq \left\{ \hat{\bm{v}} \in \hat{V}(E): \hat{\bm{v}} \cdot
    \bm{n}_{e_i} |_{e_i} \in
    \mathbb{R} \,\forall e_i \in \mathcal{E}(E), \nabla \cdot \hat{\bm{v}}
    \in \mathbb{P}_0(E), \nabla \times \hat{\bm{v}} = \bm{0} \right\},
\end{gather*}
where $\bm{n}_{e_i}$ is the unit tangential vector of $\gamma$ which point
outward from $E$.
The global spaces $\hat{Q}_h(\gamma)$ and $\hat{V}_h(\gamma)$ follow naturally.
The \dofs for $\hat{Q}_h(\Omega)$ are piece-wise
constant in each element of $\mathcal{T}(\gamma)$, while for $\hat{\bm{v}}
\in \hat{V}_h(\gamma)$, given an ending point $e$ of an element $E$, we consider
$\hat{\bm{v}} \cdot \bm{n}_e |_e$. Also in this case, we introduce a projection operator
$\hat{\Pi}_0: \hat{V}(E) \rightarrow \hat{\mathcal{V}}(E)$ such that $\hat{a}(
\hat{\bm{v}} - \hat{\Pi}_0 \hat{\bm{v}}, \hat{\bm{w}} ) =0 $ for all
$\hat{\bm{w}} \in \hat{\mathcal{V}}(E)$ and where
\begin{gather*}
    \hat{\mathcal{V}}(E) \defeq \left\{ \hat{\bm{v}} \in \hat{V}_h(E): \hat{\bm{v}} =
    \FIX{\hat{\kappa}}{\hat{\lambda}} \nabla \hat{v},
    \text{ for } \hat{v} \in \mathbb{P}_1(E) \right\}.
\end{gather*}
Following the same process described for an element belonging to a fracture
mesh, we obtain the local matrix formulation for the $H_{\rm div}$-mass matrix
\begin{gather*}
    \hat{a}( \hat{\bm{\varphi}}_\omega, \hat{\bm{\varphi}}_\theta ) \approx
    \hat{a}_h( \hat{\bm{\varphi}}_\omega, \hat{\bm{\varphi}}_\theta ) =
    \left[
    ( \hat{\Pi}^* )^\top \hat{G} \hat{\Pi}^* + \hat{\varsigma} ( I - \hat{D}
    \hat{\Pi}^* )^\top ( I - \hat{D} \hat{\Pi}^* )
    \right]_{\omega \theta}
\end{gather*}
where $\hat{\bm{\varphi}}_\omega, \hat{\bm{\varphi}}_\theta$ are elements of the
base for $\hat{V}_h (E)$, $\hat{\varsigma} \in \mathbb{R}$ is a proper
stabilization parameter, and $\hat{\Pi}^*, \hat{G}$, and $\hat{D}$ are suitable
matrices. The approximation of the $\hat{b}(\cdot, \cdot)$ bilinear form follows from the
definition of the \dofs for the velocity and the pressure.
The discrete formulation of the coupling condition presented in
\cref{eq:cc_weak} can be derived using the \dofs introduced previously. In
particular, for each intersection $\gamma_j$ of a fracture $\Omega_i$, we double
the \textit{d.o.f.} associated
with the velocity obtaining the discrete form of $\bm{u}^+ \cdot T_i \bm{n}_j
|_{\gamma_j}$ and $\bm{u}^- \cdot T_i \bm{n}_j |_{\gamma_j}$. With this choice
the implementation of \cref{eq:cc_weak} is immediate.

In the case where two intersections meet in a point $\xi$, coupling conditions
\cref{eq:two_codimensional} should be adopted. For each intersection
we double the reduced velocity \dofs in $\xi$ and we enforce
the condition by using Lagrange multipliers in $\xi$.


\subsection{\FIX{}{Stabilization term}}


To conclude, we discuss now the stabilization parameters $\varsigma$ and
$\hat{\varsigma}$ introduced previously. Following
\cite{Brezzi2014,BeiraodaVeiga2014b,BeiraoVeiga2016}, to obtain a proper error
decay we require that exist $\iota_*, \iota^*, \hat{\iota}_*, \hat{\iota}^* \in
\mathbb{R}^+$, independent from the discretization size, such that
\begin{gather} \label{eq:stabilization}
    \begin{gathered}
        \iota_* a( \Pi_0 \bm{v}, \Pi_0 \bm{v} ) \leq
        s( (I-\Pi_0) \bm{v}, (I-\Pi_0) \bm{v} ) \leq \iota^* a( \Pi_0 \bm{v}, \Pi_0
        \bm{v} ) \qquad \forall \bm{v} \in \FIX{V_h}{V}(\Omega)\\
        \hat{\iota}_* \hat{a}( \Pi_0 \hat{\bm{v}}, \Pi_0 \hat{\bm{v}} ) \leq
        \hat{s}( (I-\hat{\Pi}_0) \hat{\bm{v}}, (I-\hat{\Pi}_0) \hat{\bm{v}} )
        \leq \hat{\iota}^*
        \hat{a}( \hat{\Pi}_0 \hat{\bm{v}}, \hat{\Pi}_0
        \hat{\bm{v}} ) \qquad \forall \hat{\bm{v}} \in \FIX{\hat{V}_h}{\hat{V}}(\Omega)
    \end{gathered}
\end{gather}
The stability term introduced previously for the fractures fulfils automatically the
request, however for highly heterogeneous fractures in the permeability a proper
scaling is recommended, \eg $\varsigma_i =
\norm{\FIX{\kappa^{-1}_i}{\lambda^{-1}_i}}_{L^\infty(\Omega_i)}$ for
each fracture $\Omega_i$.
To fulfil the second request in \cref{eq:stabilization} we can compute
explicitly the local matrices involved. Given a segment $E$ of length $h_E$ and
supposing that the effective permeability $\FIX{\hat{\kappa}}{\hat{\lambda}}$
is constant in $E$, we obtain
\begin{gather*}
    ( \hat{\Pi}^* )^\top \hat{G} \hat{\Pi}^* = \dfrac{h_E}{4
    \FIX{\hat{\kappa}}{\hat{\lambda}}}
    \begin{bmatrix}
        1 & -1 \\ -1 & 1
    \end{bmatrix}
    \quad \text{and} \quad
    ( I - \hat{D} \hat{\Pi}^* )^\top ( I - \hat{D} \hat{\Pi}^* ) = \dfrac{1}{2}
    \begin{bmatrix}
        1 & 1 \\ 1 & 1
    \end{bmatrix}.
\end{gather*}
We see that choosing $\hat{\varsigma} = h_E / \FIX{\hat{\kappa}}{\hat{\lambda}}$
is enough to scale properly the stability term for the intersection.

\begin{remark}
    It is important to note that if $\bm{u} = \Pi_0 \bm{u}$ or $\bm{u}= \Pi_0
    \bm{v}$ we have the identity: $a_h(\bm{u}, \bm{v}) = a(\bm{u}, \bm{v})$. The
    same is valid for the intersection flow.
\end{remark}


%


\section{The computational grid} \label{sec:coarsening}


Realistic fracture networks can have a highly complex geometry, and
correspondingly the construction of the computational grid is challenging.  In
particular, intersections between fractures add complexity, since for the types
of discretization considered herein, intersections are treated as constraints in
the gridding algorithm.  This often leads to a high number of cells, and
depending on intersection geometry, also low quality elements.  The requirement
that the grid conforms to intersections can be avoided by using specialized
numerical methods, however their implementation tends to be tedious.

Since the virtual element method can handle almost any polygon type, we make two modifications
of what can be considered a broadly used DFN meshing
algorithm, both aimed at alleviating the computational cost. Other approaches can
be found in \cite{Mustapha2007,Hyman2014,Karimi-Fard2016}, to name a few.
Consider the DFN composed by two intersecting fractures, $\Omega_i$ and
$\Omega_j$, presented in \cref{fig:mesh_two_fract}, the extension to several
fractures is straightforward.
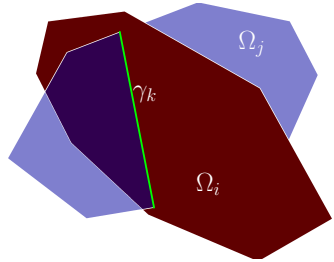
\begin{figure}[tbp]
    \centering
    \resizebox{0.33\textwidth}{!}{\fontsize{20pt}{7.2}\selectfont%
\begingroup%
  \makeatletter%
  \providecommand\color[2][]{%
    \errmessage{(Inkscape) Color is used for the text in Inkscape, but the package 'color.sty' is not loaded}%
    \renewcommand\color[2][]{}%
  }%
  \providecommand\transparent[1]{%
    \errmessage{(Inkscape) Transparency is used (non-zero) for the text in Inkscape, but the package 'transparent.sty' is not loaded}%
    \renewcommand\transparent[1]{}%
  }%
  \providecommand\rotatebox[2]{#2}%
  \ifx\svgwidth\undefined%
    \setlength{\unitlength}{253.68000421bp}%
    \ifx\svgscale\undefined%
      \relax%
    \else%
      \setlength{\unitlength}{\unitlength * \real{\svgscale}}%
    \fi%
  \else%
    \setlength{\unitlength}{\svgwidth}%
  \fi%
  \global\let\svgwidth\undefined%
  \global\let\svgscale\undefined%
  \makeatother%
  \begin{picture}(1,0.78851547)%
    \put(0,0){\includegraphics[width=\unitlength,page=1]{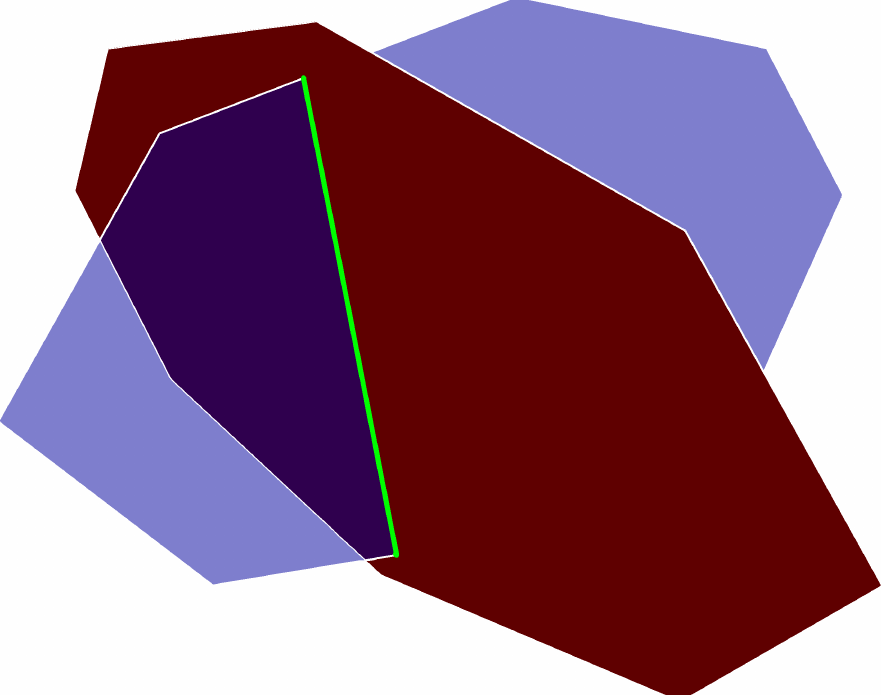}}%
    \put(0.58239336,0.20693578){\color[rgb]{1,1,1}\makebox(0,0)[lb]{\smash{$\Omega_i$}}}%
    \put(0.71304978,0.64739595){\color[rgb]{1,1,1}\makebox(0,0)[lb]{\smash{$\Omega_j$}}}%
    \put(0.38652614,0.49767331){\color[rgb]{1,1,1}\makebox(0,0)[lb]{\smash{$\gamma_k$}}}%
  \end{picture}%
\endgroup%
    }%
    \caption{The geometry of the two fractures and their intersection marked in
    light green, which are used in this section to present the meshing and
    coarsening strategy.}%
    \label{fig:mesh_two_fract}
\end{figure}
Our first modification is to build each mesh fracture separately and then link
them together through the intersection $\gamma_k$.
\FIX{The mesh of each fracture is created using}{This reduces the meshing
problem to a set of decoupled 2d domain with internal constraints from the
intersection lines, and allows us to apply established 2d meshing
software without adaptation for the coupling between domains. We have applied}
the library Triangle, see \cite{Shewchuk1996}, which is a fast and robust
triangular grid generator that allows for internal constraints. See
\cref{fig:mesh_single} on the left as an example.  The independent
meshing also means the fractures can be meshed in parallel.
\begin{figure}[tbp]
    \centering
    \includegraphics[width=0.475\textwidth]{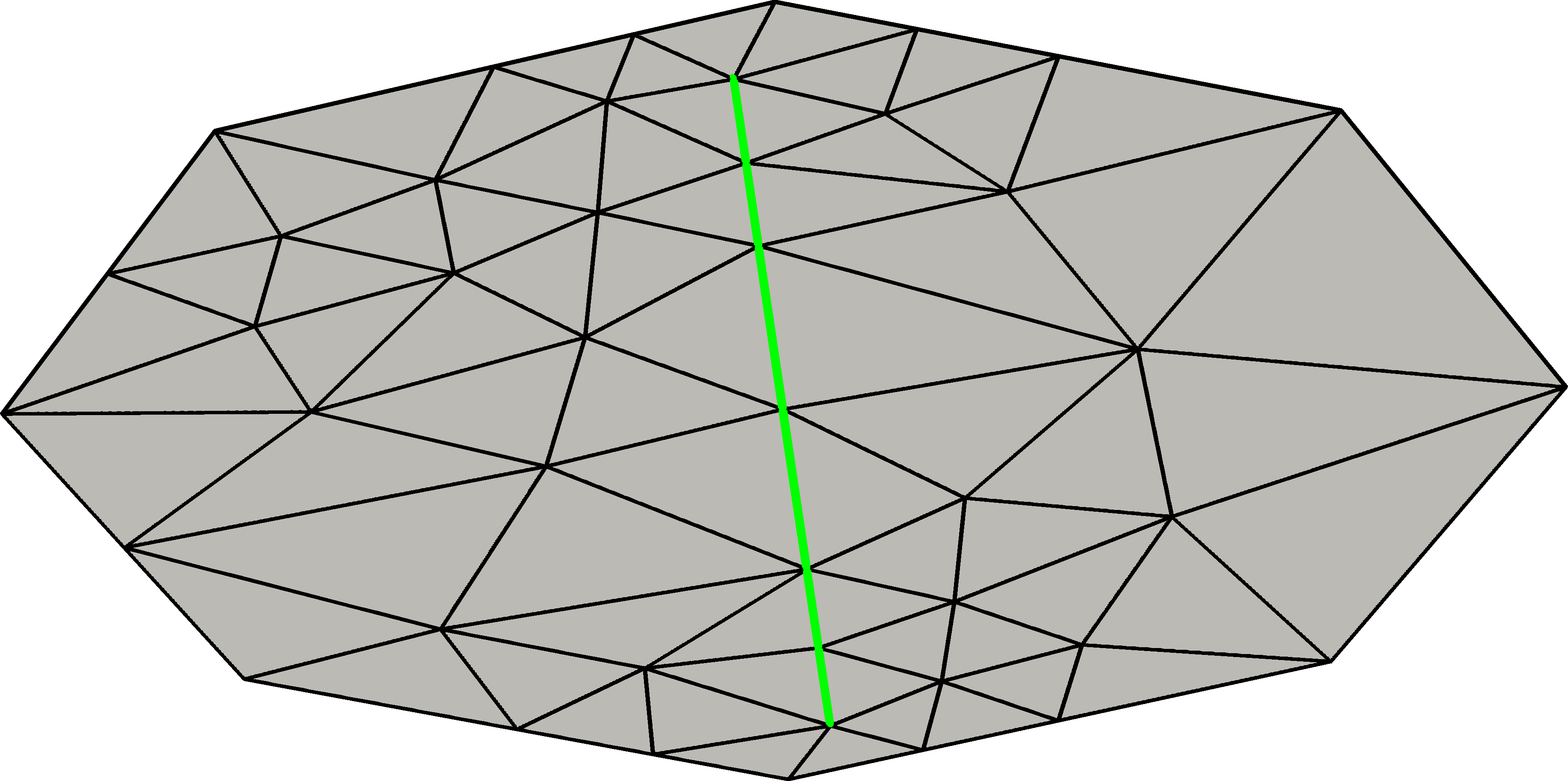}%
    \hfill
    \includegraphics[width=0.475\textwidth]{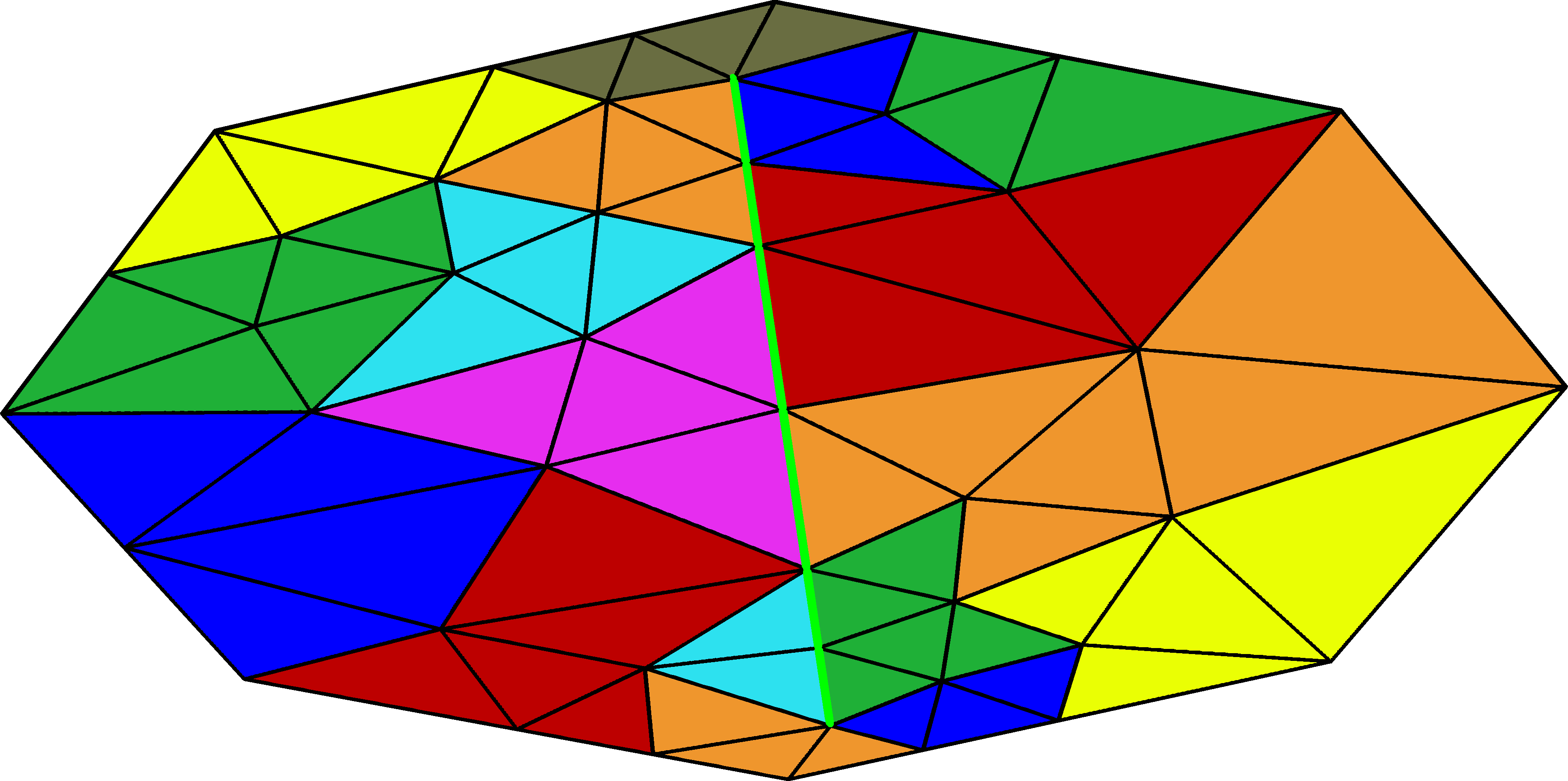}%
    \caption{Left: a triangulated fracture where the intersection is
    preserved. In general the conformity of the triangles across the
    intersection is not preserved. Right: Grouping of triangular cells
    after coarsening.  Clusters of triangles with the same colour will form
    a cell in the coarse mesh.}%
    \label{fig:mesh_single}
\end{figure}

When a mesh of a fracture is created, our second modification is to coarsen
the grid by cell agglomeration so as to lighten the total computational cost.
Our approach is motivated by the coarsening in algebraic multigrid methods,
 \eg \cite{Trottenberg2001}, and the coarsening is carried out
independently for each fracture.
As a measure of connectedness between, we consider a two-point flux approximation (TPFA)
discretization of \cref{eq:darcy_frac}, and denote the discretization matrix $A$.
To preserve fracture intersections in the coarse grid, we treat intersections
as a boundary for the TPFA discretization, and we explicitly prohibit a coarse
cell to cover both sides of the end of an intersection.
The coarsening ratio is determined by the parameter $c_{depth}$.
Details on the coarsening algorithm can be found in
\cref{sec:coarseapp}, or in the
literature on algebraic multigrid methods, see \eg \cite{Trottenberg2001}.

\cref{fig:mesh_single} on the right shows the clustering map after the
application of the algorithm.  We clearly see that the resulting elements may be
concave or even non star-shaped (but a finite union of star-shaped), \eg the big
orange element on the right part of the mesh.  It is worth to notice that the
process of creating and coarsening the meshes is embarrassingly parallel.  Once
the mesh of each fracture is created, we simply compute a co-refinement of the
edges lying on $\gamma_k$ from both fractures.  Splitting the edges of the
fractures at the intersection and considering a coherent numeration of the
elements and edges the global mesh is build up. Refer to
\cref{fig:mesh_two_fract_mesh} as an example of final mesh with and without the
coarsening strategy.  The effect of the parameter $c_{depth}$ is illustrated in
\cref{fig:iso_coarse_mesh}.  We also note that in anisotropic media the coarse
grid will to some extent adapt to the preferential flow directions, see
\cref{fig:ani_coarse_mesh}.
\begin{figure}[tbp]
    \centering
    \includegraphics[width=0.475\textwidth]{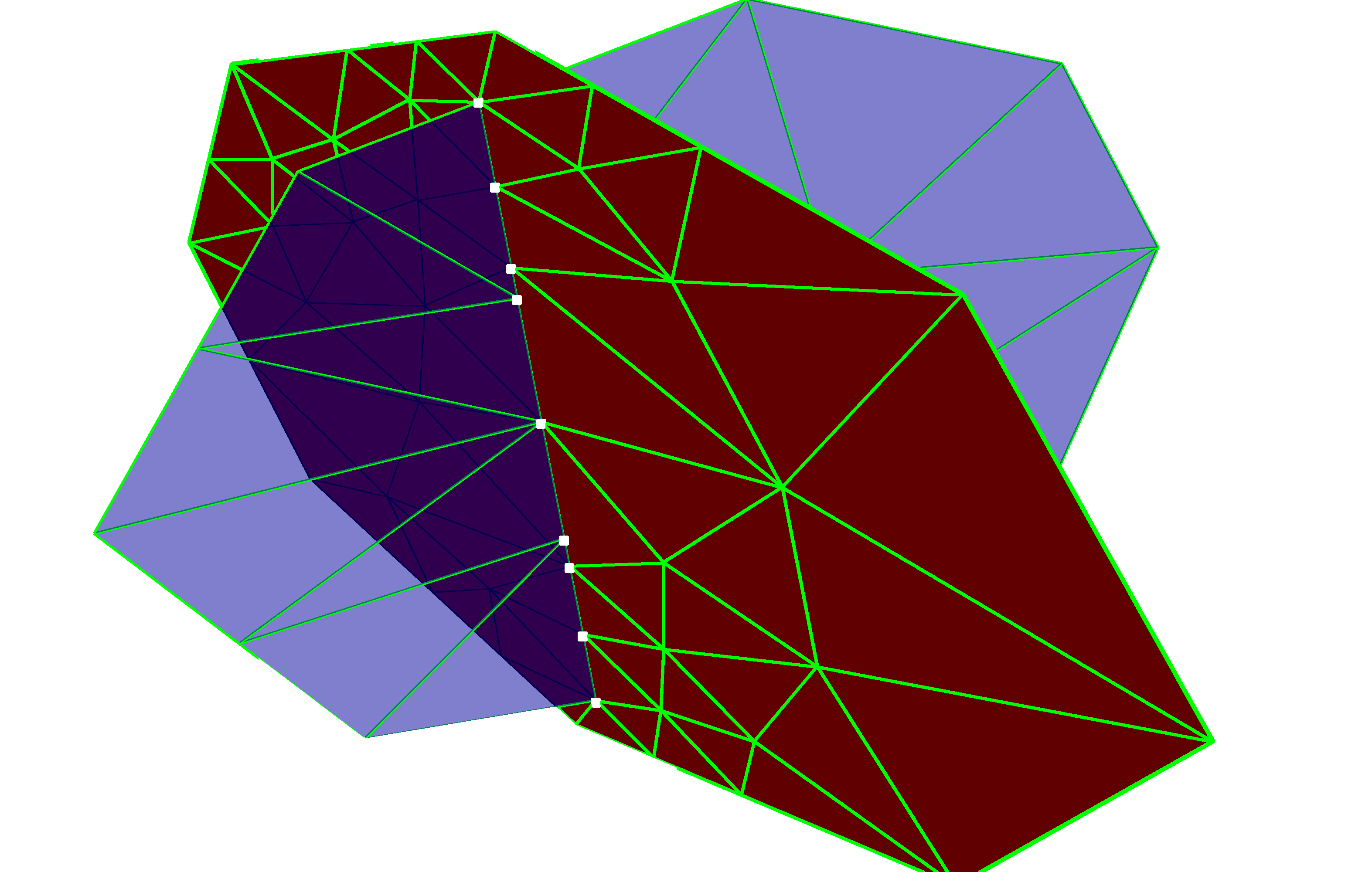}%
    \hfill
    \includegraphics[width=0.475\textwidth]{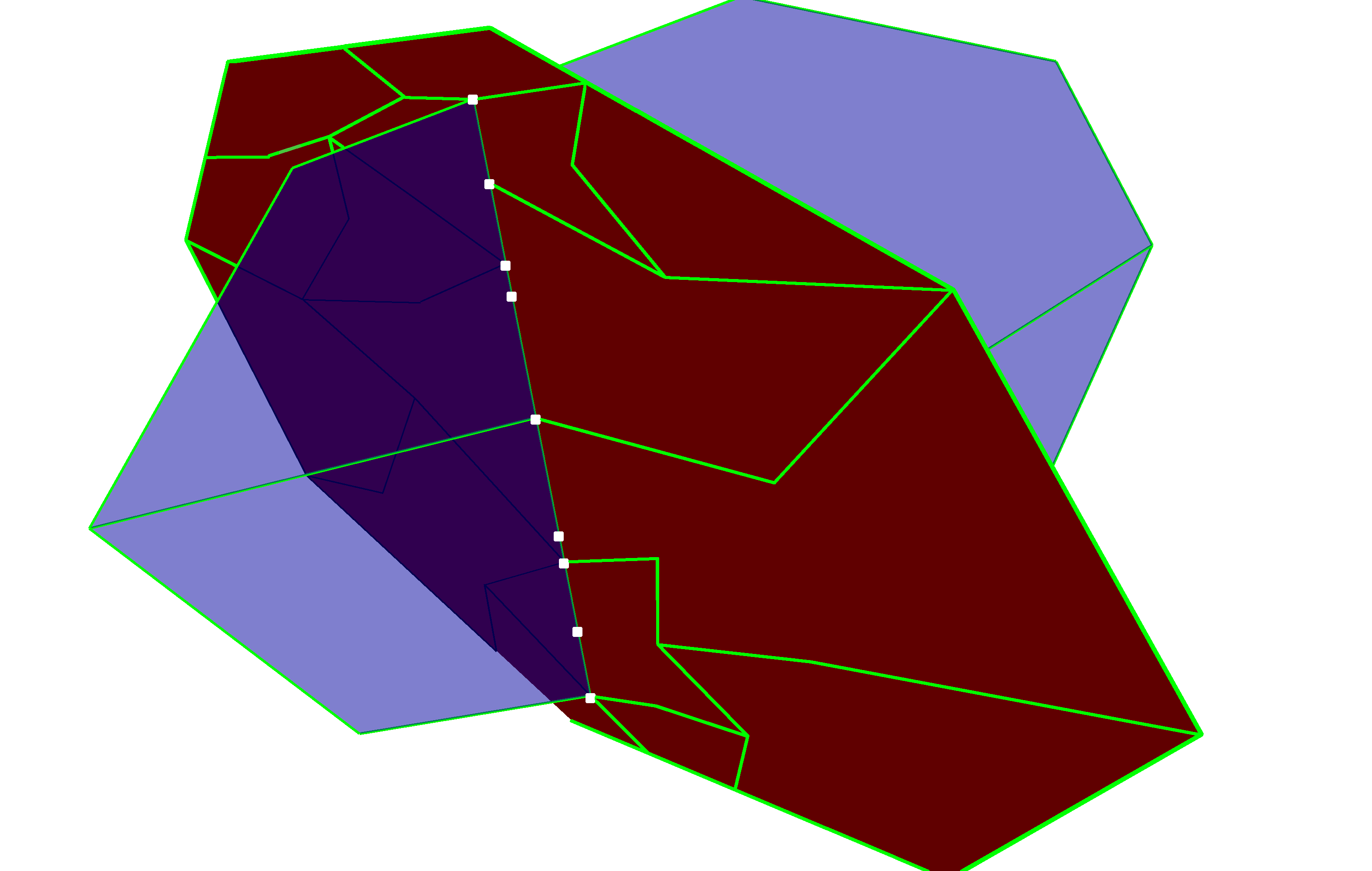}%
    \caption{The meshes of the two fractures linked together before (left) and
    after (right) coarsening. The points at the
    intersection are highlighted in white and preserved in the coarsening.
    The original grid has 93 elements and 153 edges, while the coarse grid has
    23 elements and 82 edges.
    The elements at the
    intersection are non-matching but conforming.}%
    \label{fig:mesh_two_fract_mesh}
\end{figure}
An example of the grid overlapped to the pressure field is represented in \cref{fig:real}
in the introduction.

%
%


\section{Examples} \label{sec:examples}


In this section we validate the models presented in the previous sections
through several tests and examples. In particular in \cref{subsec:coarsening} we
highlight the potentiality of the coarsening algorithm presented in
\cref{sec:coarsening} applied in our context. \cref{subsec:convergence}
contains several tests with a single as well as multiple fractures to put in
evidence the error decay and others properties of the numerical solution.
\FIX{}{In \cref{subsubsect:isect} we present a numerical example showing the importance
to use model presented in \cref{pb:dc}.}
Finally in \cref{subsubsec:realistic} we present the solution on a realistic
geometry.
Our implementation is carried out within the framework of the Matlab Reservoir Simulation Toolbox (MRST)
\cite{Lie2012}.

We make use of the following
\begin{gather*}
    \text{sparsity} = \dfrac{\text{nnz}(A)}{\text{size}(A)^2},
\end{gather*}
with $\text{nnz}(A)$ the number of non-zero values of the matrix $A$ and $\text{size}(A)$
the number of rows (or columns) of $A$, as a measure for the sparsity of a matrix.


\subsection{Coarsening} \label{subsec:coarsening}


In this subsection we investigate the ability of the coarsening algorithm,
introduced in \cref{sec:coarsening} to generate coarse meshes where the direction of the cells
reflects the underlying  physic. To that end, we
consider two problems with different effective permeability: isotropic and
anisotropic, the latter in a heterogeneous setting. In both cases we start from
a single fracture discretized by a triangular grid, then using the algorithm we
coarsen the mesh with different level of refinement: $c_{depth} = 1$, $c_{depth}
= 3$, and $c_{depth} = 5$. The fracture is the unit square rotate by $\pi/4$
along the $z$ axis.

In the first case the effective permeability is the identity tensor.
\cref{fig:iso_coarse_mesh} shows the meshes obtained with different level of
coarsening.
\begin{figure}[tbp]
    \centering
    \includegraphics[width=0.25\textwidth]{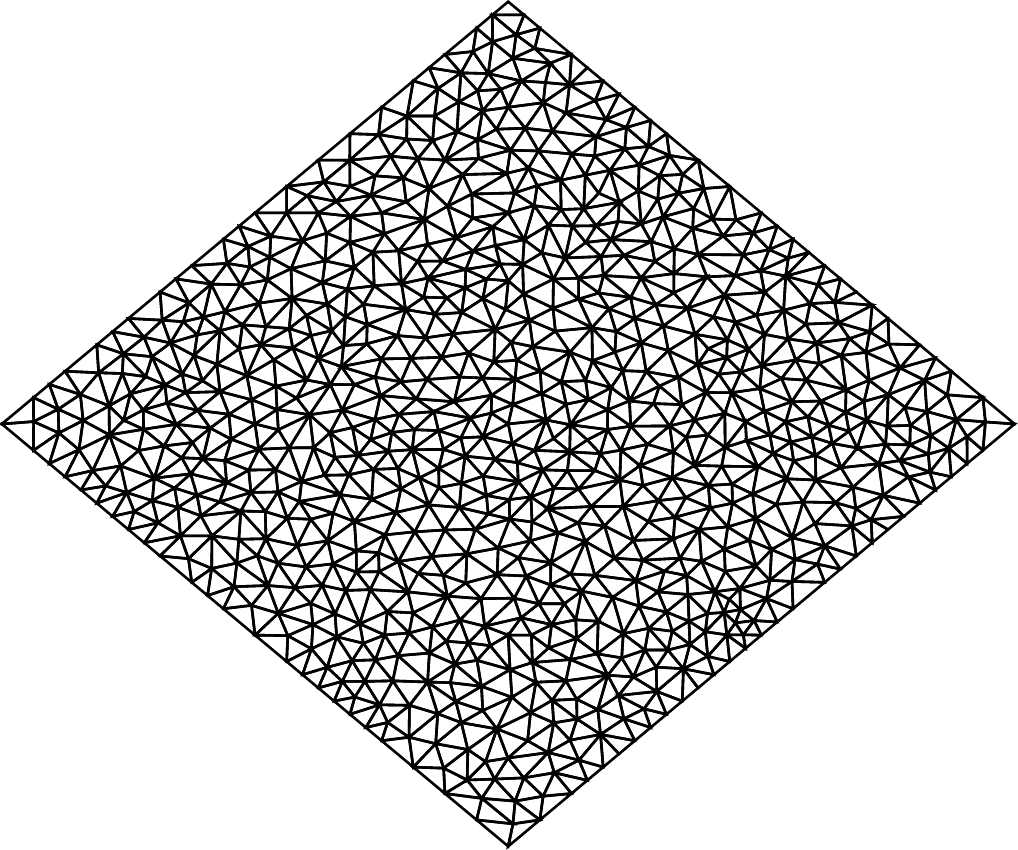}%
    \hfill%
    \includegraphics[width=0.25\textwidth]{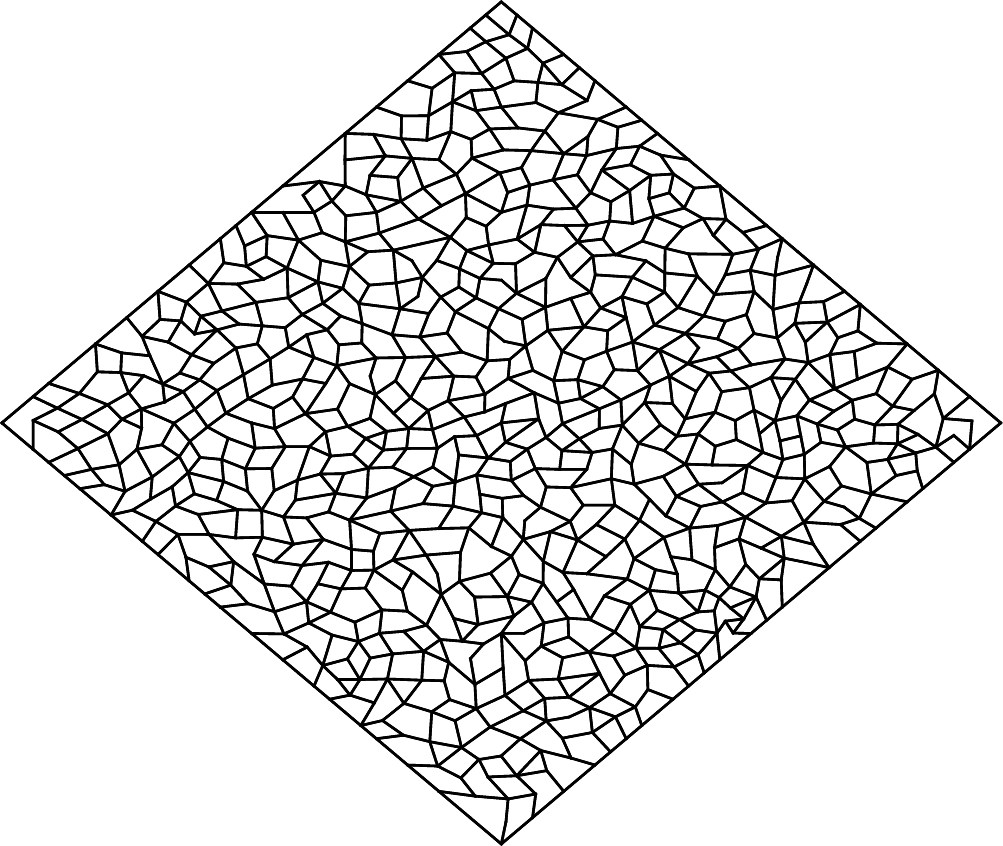}%
    \hfill%
    \includegraphics[width=0.25\textwidth]{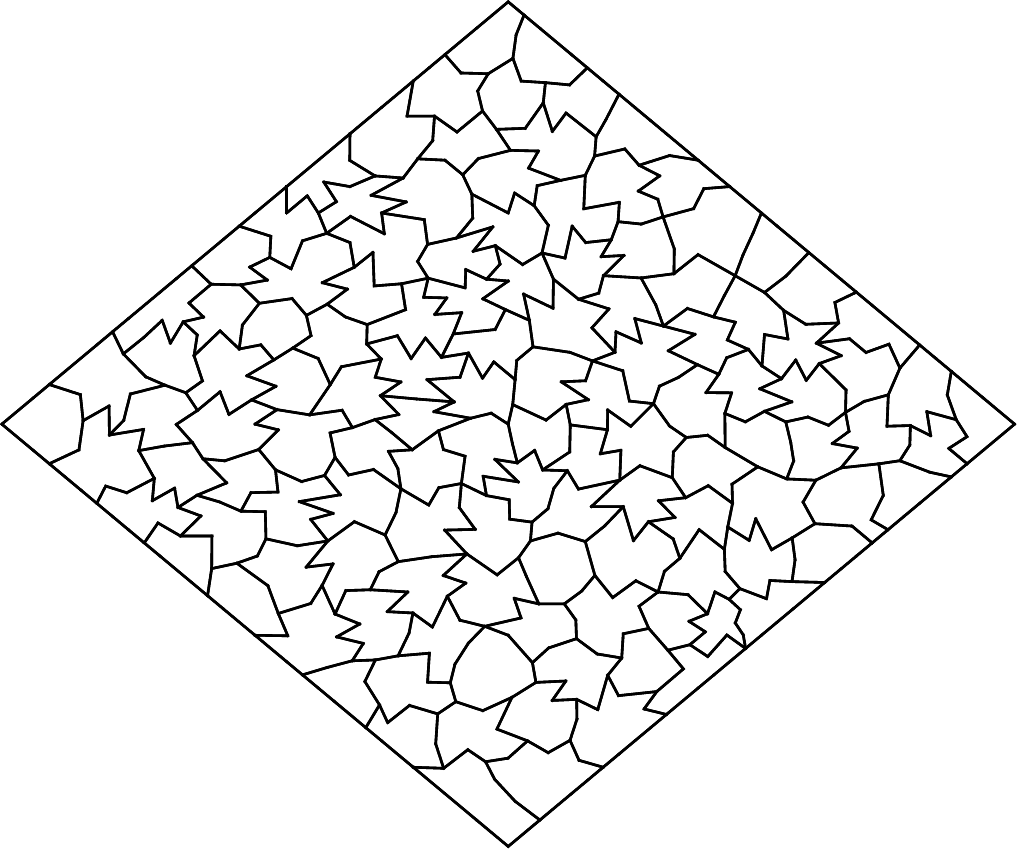}%
    \hfill%
    \includegraphics[width=0.25\textwidth]{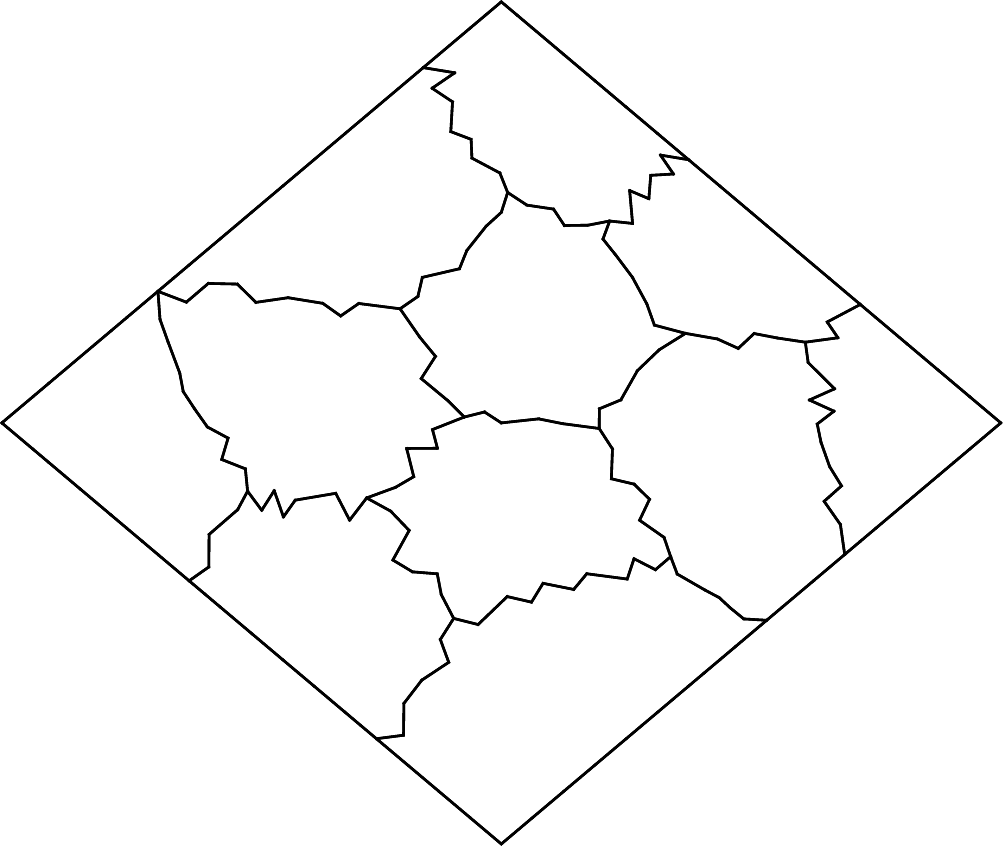}%
    \caption{Example of meshes used for the example in
    \cref{subsec:coarsening} with isotropic effective permeability tensor.
    From the left: triangular mesh, the coarse meshes with $c_{depth}=1$,
    $c_{depth}=3$, and $c_{depth}=5$, respectively.}%
    \label{fig:iso_coarse_mesh}
\end{figure}
The resulting cells of the meshes do not have a preferred alignment.  The relative
sparsity of the VEM matrix for each coarse level is: $1.4\cdot 10^{-3}$ for the
triangular grid, $3.9\cdot 10^{-3}$ for $c_{depth}=1$, $2.3\cdot 10^{-2}$ for
$c_{depth}=3$, and $2\cdot 10^{-1}$ for $c_{depth}=5$, which is a natural effect
of this coarsening algorithm.  Finally the number of edges (minimum, average,
and maximum) and cells for each level is: $(3,5,6)$ and 628 for $c_{depth}=1$,
$(8,12,15)$ and 122 for $c_{depth}=3$, and $(29,35,41)$ and 11 for
$c_{depth}=5$. The resulting meshes respect the theoretical requests for the
VEM.

The second case deals with an anisotropic and heterogeneous effective
permeability tensor.  We divide the fracture into four sectors: top-left part
($x \leq 0$ and $y > 0$), top-right part ($x > 0$ and $y > 0$), bottom-left part
($x \leq 0$ and $y \leq 0$), and bottom-right part ($x > 0$ and $y \leq 0$). In
the first and last sector we set $\FIX{\kappa}{\lambda}={\rm diag}[1, 100]$, while in the
others we have $\FIX{\kappa}{\lambda}={\rm diag}[100,1]$.
\cref{fig:ani_coarse_mesh} shows the
meshes obtained with different level of coarsening.
\begin{figure}[tbp]
    \centering
    \includegraphics[width=0.25\textwidth]{iso_tri}%
    \hfill%
    \includegraphics[width=0.25\textwidth]{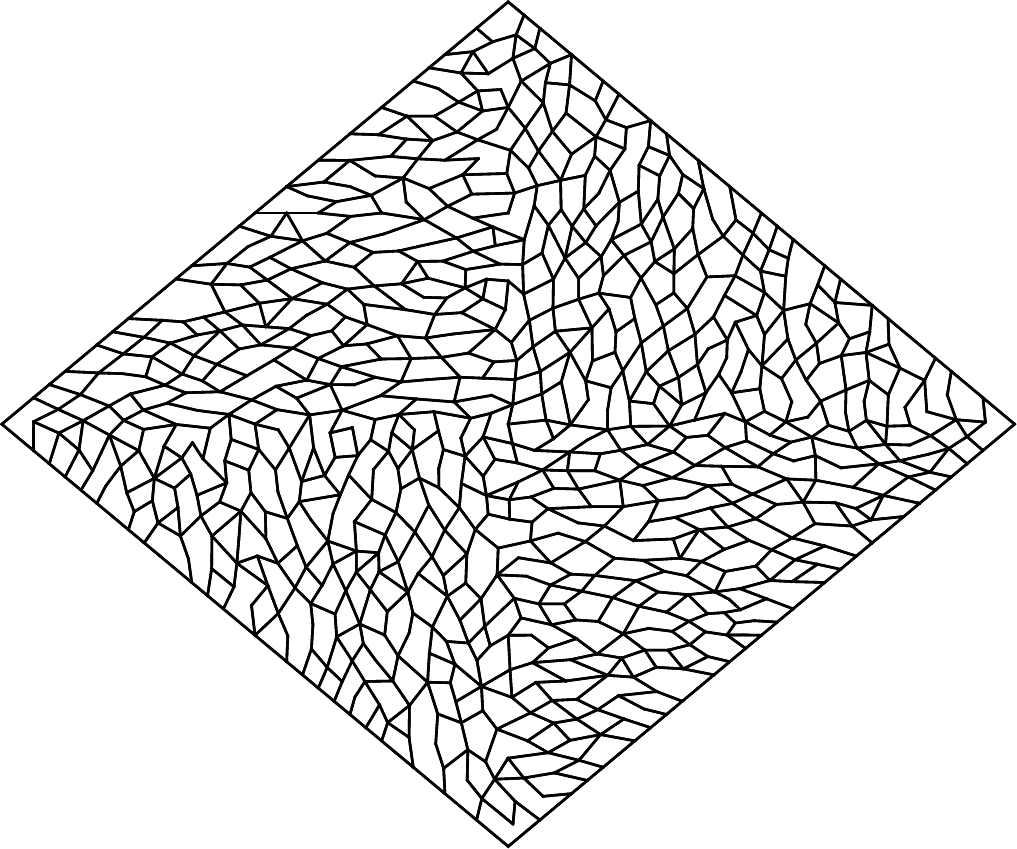}%
    \hfill%
    \includegraphics[width=0.25\textwidth]{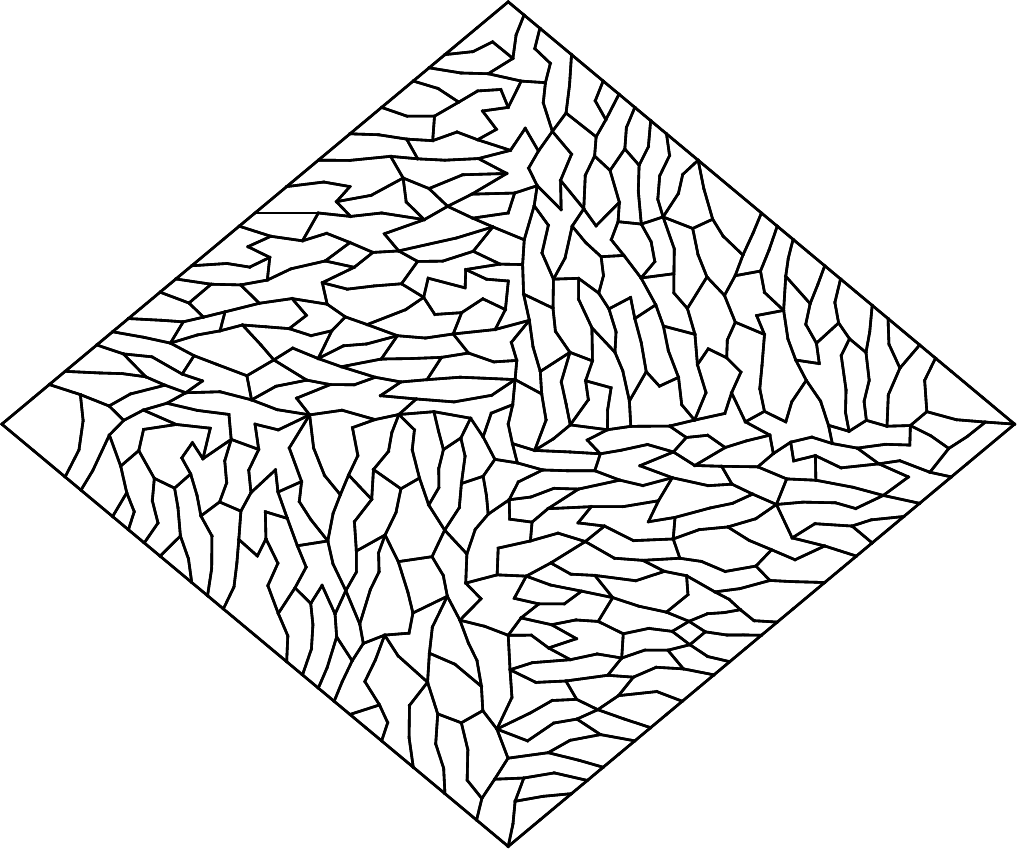}%
    \hfill%
    \includegraphics[width=0.25\textwidth]{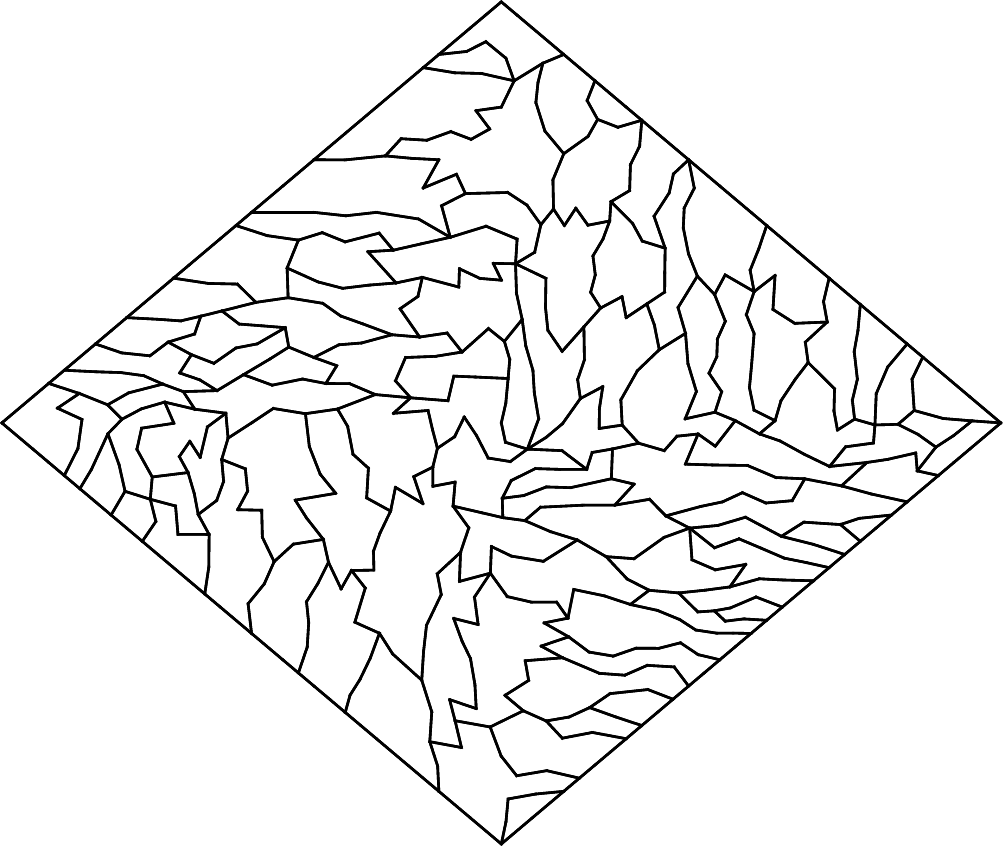}%
    \caption{Example of meshes used for the example in
    \cref{subsec:coarsening} for the anisotropic permeability tensor.
    From the left: triangular mesh, the coarse meshes with $c_{depth}=1$,
    $c_{depth}=3$, and $c_{depth}=5$, respectively.}%
    \label{fig:ani_coarse_mesh}
\end{figure}
The resulting cells of the meshes have a preferred alignment related to the
underlying permeability which respect both the anisotropy and the heterogeneity.
The sparsity of the VEM matrix for each coarse level is: $1.4\cdot 10^{-3}$ for
the triangular grid, $4.5\cdot 10^{-3}$ for $c_{depth}=1$, $1.2\cdot 10^{-2}$
for $c_{depth}=3$, and $3.1\cdot 10^{-2}$ for $c_{depth}=5$, which is a natural
effect of this coarsening algorithm.  Finally the number of edges (minimum,
average, and maximum) and cells for each level is: $(3,5,8)$ and 557 for
$c_{depth}=1$, $(4,8,12)$ and 236 for $c_{depth}=3$, and $(5,14,24)$ and 95 for
$c_{depth}=5$. Compared with the previous case we notice that the number of
cells and the variability in the number of edges is greater, which is needed to
preserve the anisotropy of the final mesh. The smaller sparsity of the matrices
is a direct consequence of this fact. Also in this case the resulting meshes
respect the theoretical requests for the VEM.

We conclude that the coarsening presented in \cref{sec:coarsening} is a valid
tool to create a coarse mesh which is adapted with the underlying physic. In the
next examples we consider this coarsening strategy in the computation of
error decay.


\subsection{Convergence \FIX{E}{e}vidence} \label{subsec:convergence}


In this part we present tests that show  numerical evidence of
convergence of the approximate to the exact solution as well as the order of
convergence. In the first part, \cref{subsubsec:single}, a single fracture is
considered with different types of meshes. In the second example two
intersecting fractures, with coupling conditions described in \cref{pb:cc} are
applied. Finally two fractures are considered where the flow is described also in their
intersection. In the following experiments we make use of the relative errors
for both the pressure and the projected velocity, we use the following expressions:
\begin{gather*}
    err(p) = \dfrac{\norm{ p - p_{ex}}_{\Omega} }{ \norm{ p_{ex} }_\Omega }
    \quad \text{and} \quad
    err( \Pi_0 \bm{u} ) = \dfrac{\norm{ \Pi_0 \bm{u} - \bm{u}_{ex}}_{\Omega} }{ \norm{
    \bm{u}_{ex} }_\Omega }
\end{gather*}
where $p_{ex}$ and $\bm{u}_{ex}$ are the exact solutions computed in the centre
of each cell. The same type of errors are computed also for the intersection when
the model described in \cref{pb:dc} is considered.
In addition, furthers details are included in \cref{sec:appendix}.


\subsubsection{Single \FIX{F}{f}racture} \label{subsubsec:single}


We consider a single fracture $\Omega$ constructed applying a rotation matrix to
the points of the unit square $[0,1]^2$. The rotation is $\pi/4$ along the $x$
axis. We assume unit effective permeability $\FIX{\kappa}{\lambda}$ and
source term equal to
\begin{gather*}
    f(x,y,z) = 7\, z - 4\, \sin\!\left(\pi\, y\right) + 2\, {\pi}^2\, y^2\,
    \sin\!\left(\pi\, y\right) - 8\, \pi\, y\, \cos\!\left(\pi\, y\right).
\end{gather*}
We set pressure boundary conditions such that the exact solution of the
problem is the following
\begin{gather*}
    p_{ex}(x,y,z) = x^2\, z + 4\, y^2\, \sin\!\left(\pi\, y\right) - 3\, z^3.
\end{gather*}
The exact Darcy velocity, in Cartesian coordinates, is
\begin{gather*}
    \bm{u}_{ex}(x,y,z) = \left[ \begin{array}{c} - 2\, x\, z\\ 0.5 ( 9\, z^2 -
    x^2 ) - 4\, y\,
    \sin\!\left(\pi\, y\right) - 2\, \pi\, y^2\, \cos\!\left(\pi\, y\right)\\
    0.5(9\, z^2 - x^2) - 4\, y\, \sin\!\left(\pi\, y\right) - 2\,
    \pi\, y^2\, \cos\!\left(\pi\, y\right) \end{array}\right].
\end{gather*}
It is possible to verify that $0 \leq p_{ex} \leq 1.337$ in $\Omega$.
We consider four different families of grids to analyse the order of
convergence for both $p$ and $\Pi_0 \bm{u}$, first build in the unit square and
then mapped in $\Omega$. We have: Cartesian grids, ``coarse grids'' where the
algorithm from \cref{sec:coarsening} is applied with $c_{depth} = 2$
on Cartesian grids, triangular grids, and ``random grids''
where, starting from Cartesian grids, the internal nodes are randomly moved. In
the latter a sanity check of the mesh is performed to avoid degenerate cases.
See \cref{fig:mesh} as an example. It is worth to notice that in the case of
``coarse grids'' and ``random grids'' the cells may not be convex.
\begin{figure}[tbp]
    \centering
    \includegraphics[width=0.33\textwidth]{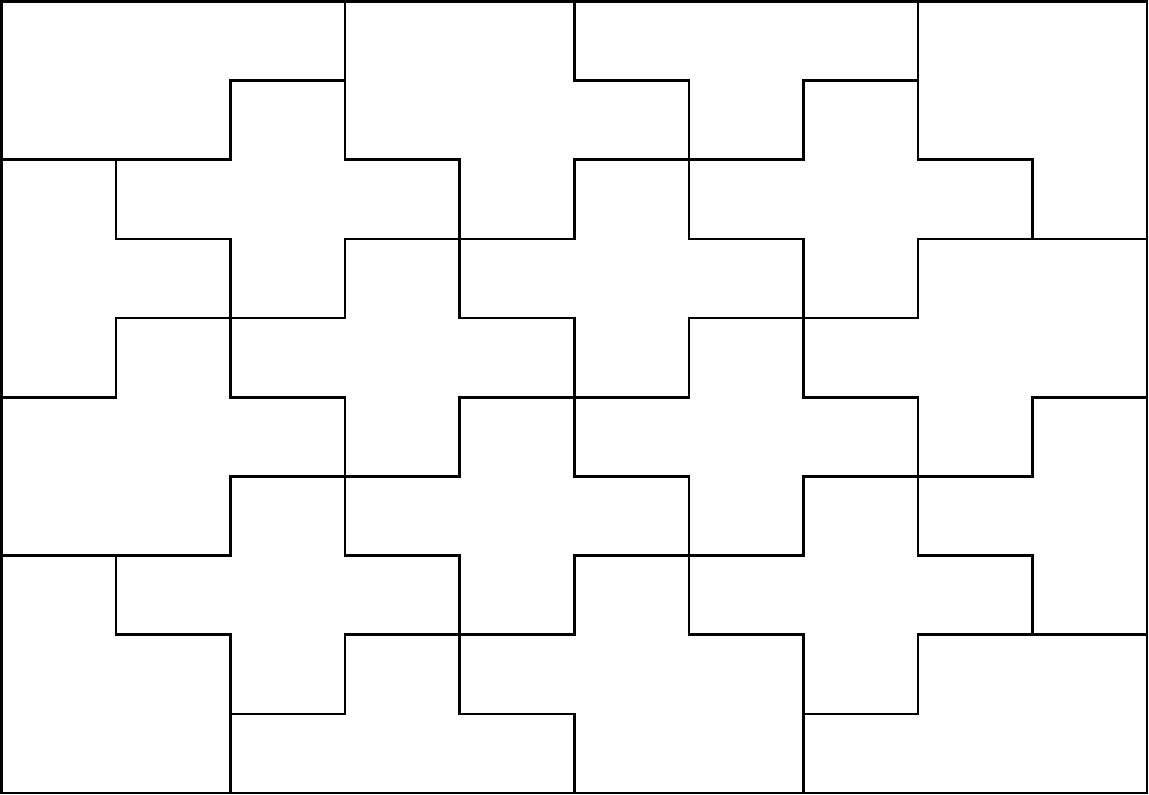}%
    \hfill%
    \includegraphics[width=0.33\textwidth]{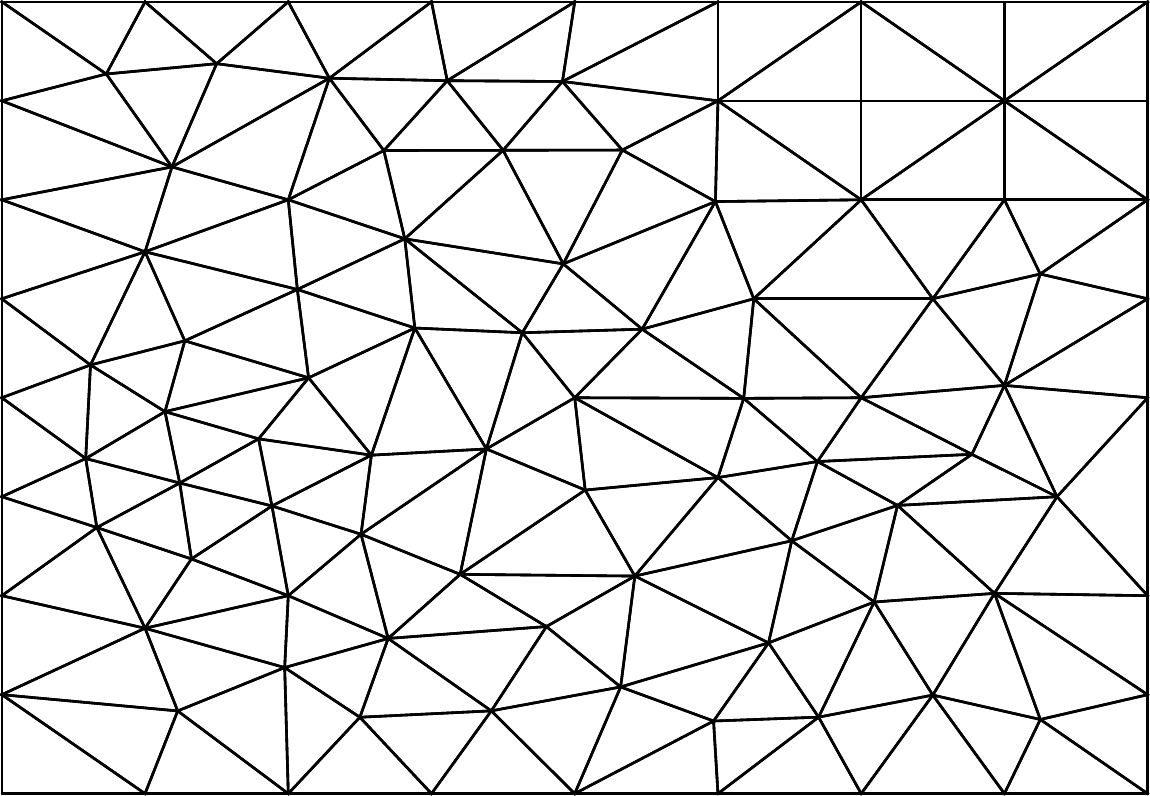}%
    \hfill%
    \includegraphics[width=0.33\textwidth]{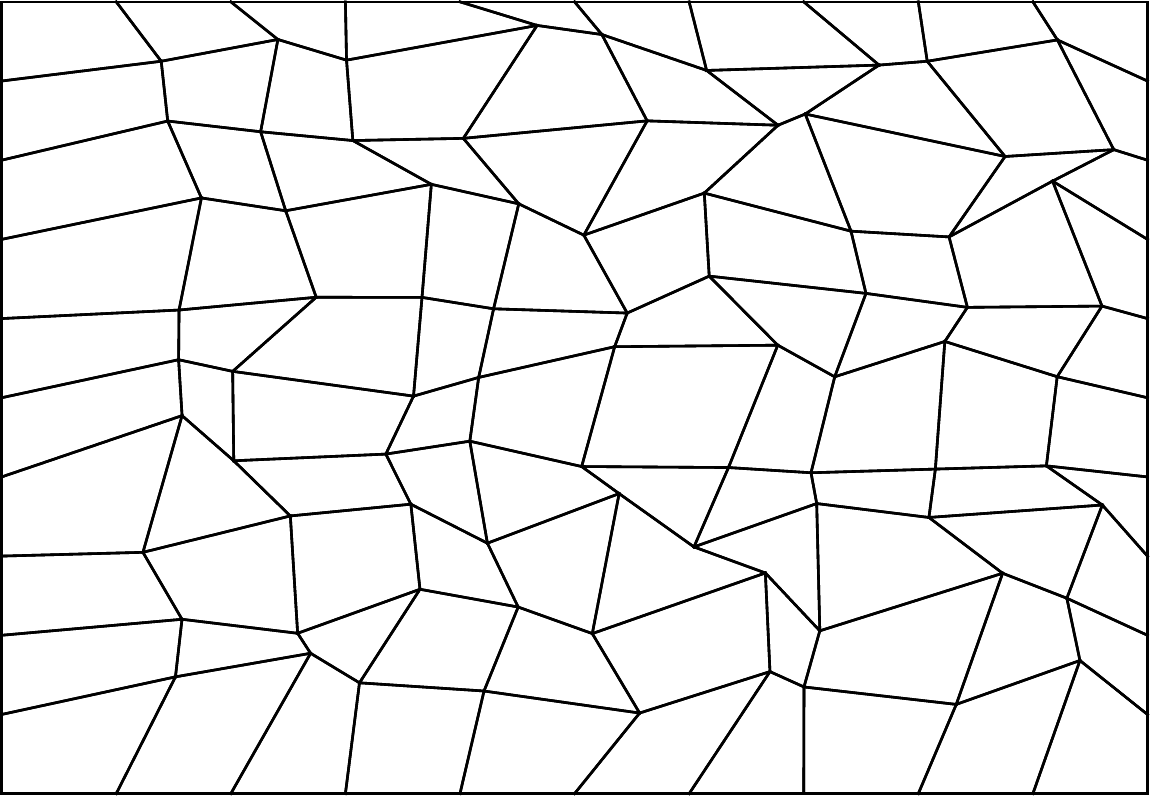}%
    \caption{Example of meshes used for the example in
    \cref{subsubsec:single}. From the left: coarse mesh, triangular mesh,
    and random mesh. The pattern of the coarse mesh remains similar for
    subsequent refinements. The meshes are presented in the $(x,y)$ plane. The
    Cartesian mesh is omitted.}%
    \label{fig:mesh}
\end{figure}
In \cref{fig:error1} we report the order of convergence for both $p$ and $\Pi_0
\bm{u}$ for the four mesh families and we can observe that, in all the cases,
the order of convergence respect the decay derived from the theory in
\cite{BeiraoVeiga2016}.
\cref{tag:errors_first} in \cref{sec:appendix} reports the detailed values for
this example, as well as the sparsity of the matrix, and the minimum and maximum
values of the solution. It is interesting to notice that in some cases the
discrete maximum or minimum principle is violated.
\begin{figure}[tbp]
    \centering
    \includegraphics{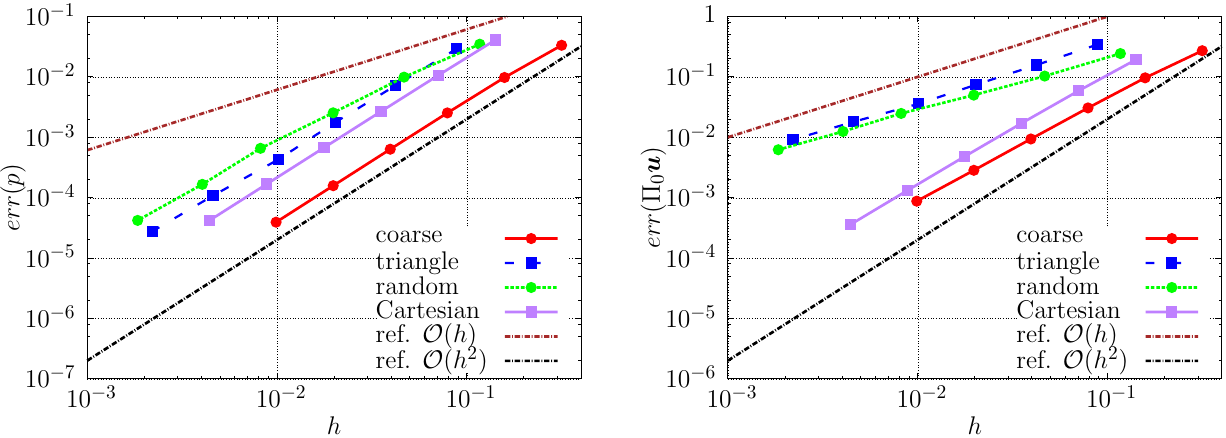}
    \caption{In the figure the error decay for both $p$, in the left plot, and
        $\Pi_0 \bm{u}$, in the right plot, for the example in
        \cref{subsubsec:single}. We add the references $\mathcal{O}(h)$ and
        $\mathcal{O}(h^2)$ to aid the comparison.}%
    \label{fig:error1}
\end{figure}


\subsubsection{Two \FIX{F}{f}ractures with \FIX{I}{i}ntersection} \label{subsubsec:two}


In the second example we analyse the error decay for two intersecting
fractures, where \cref{pb:cc} is considered. For this test we adapt the problem
in Subsubsection 5.3.1 of \cite{Benedetto2014} to our setting. Each fracture is
build from an ellipses discretized by 8 segments. We consider two
orthogonal fractures, depicted in \cref{fig:sol2}, defined as
\begin{gather*}
    \Omega_1 = \left\{ (x,y,z) \in \mathbb{R}^3: z^2 + 4y^2 \leq 1, x = 0
    \right\}\\
    \Omega_2 = \left\{ (x,y,z) \in \mathbb{R}^3: x^2 + 4y^2 \leq 1, z = 0
    \right\}.
\end{gather*}
\begin{figure}[tbp]
    \centering
    \includegraphics[width=0.25\textwidth]{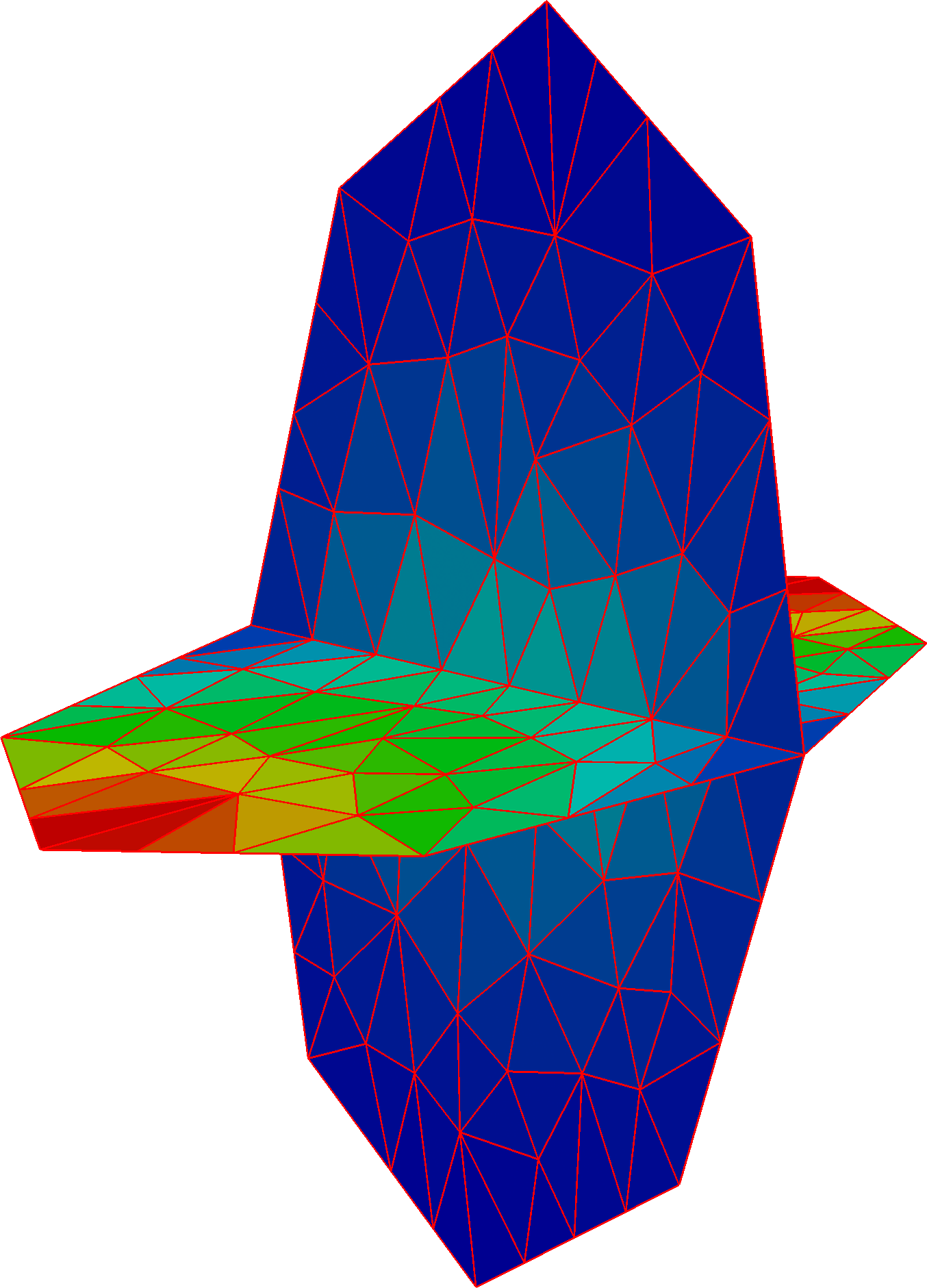}%
    \hfill%
    \includegraphics[width=0.25\textwidth]{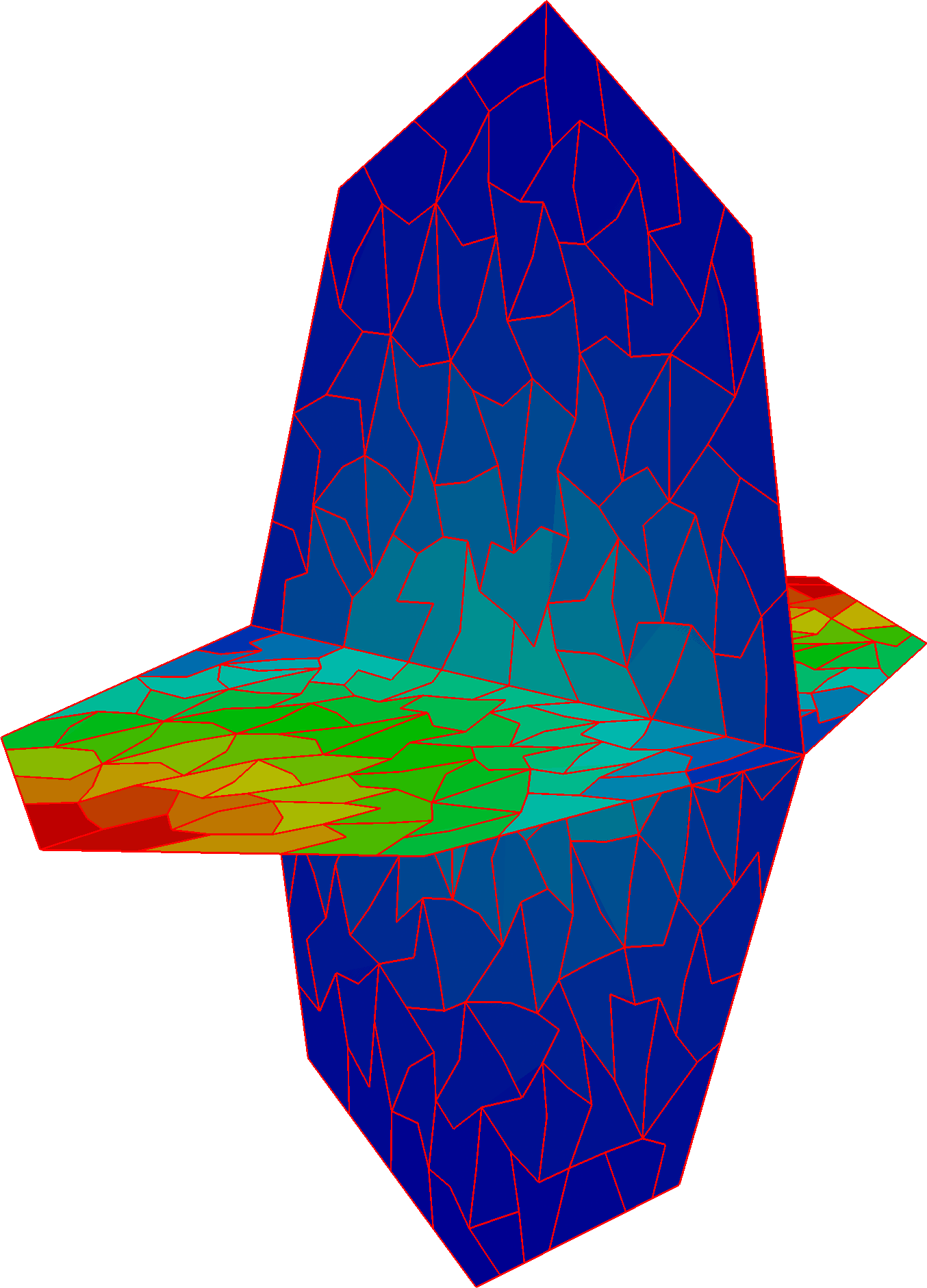}%
    \hfill%
    \includegraphics[width=0.25\textwidth]{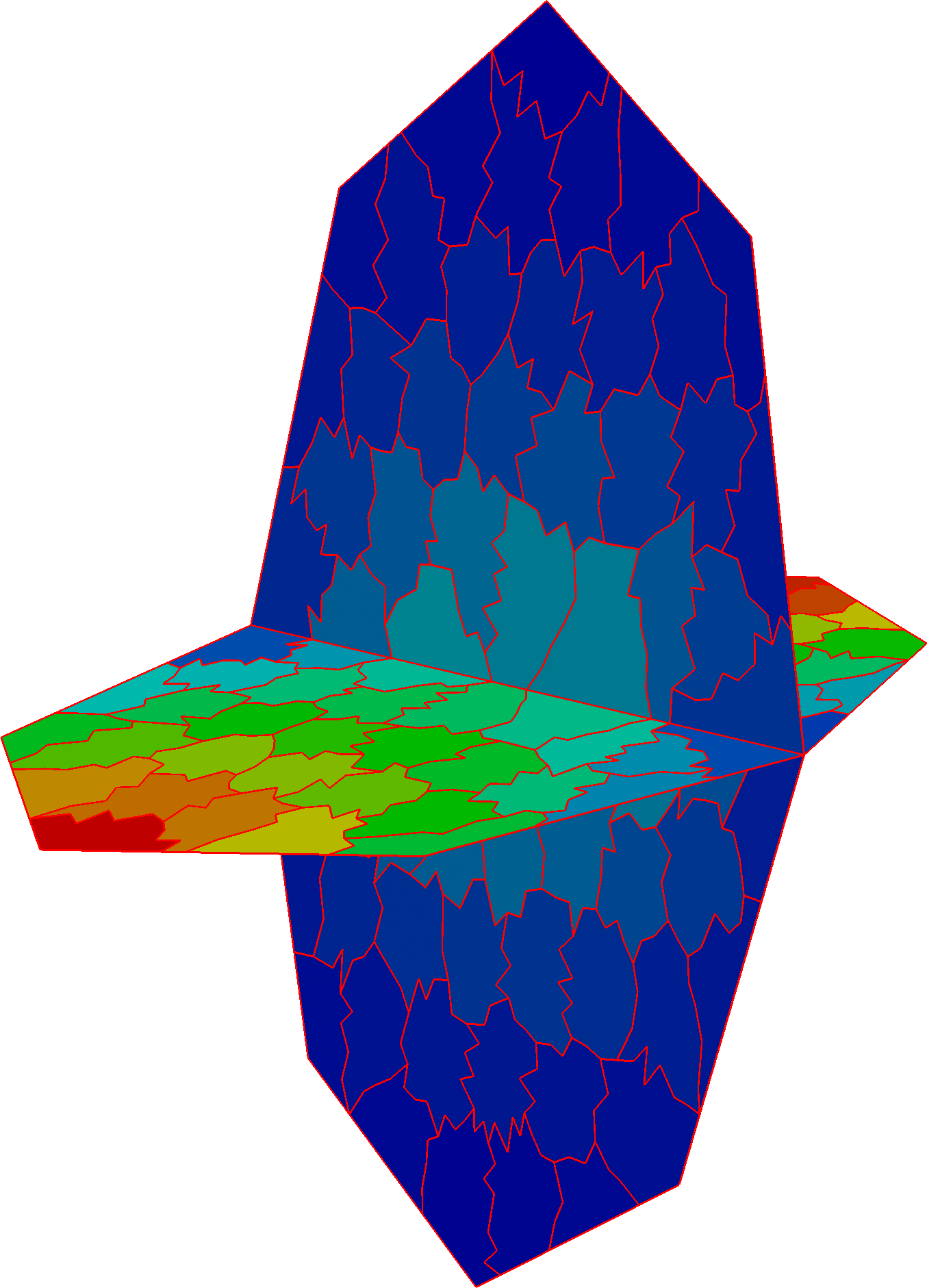}%
    \hfill%
    \includegraphics[width=0.25\textwidth]{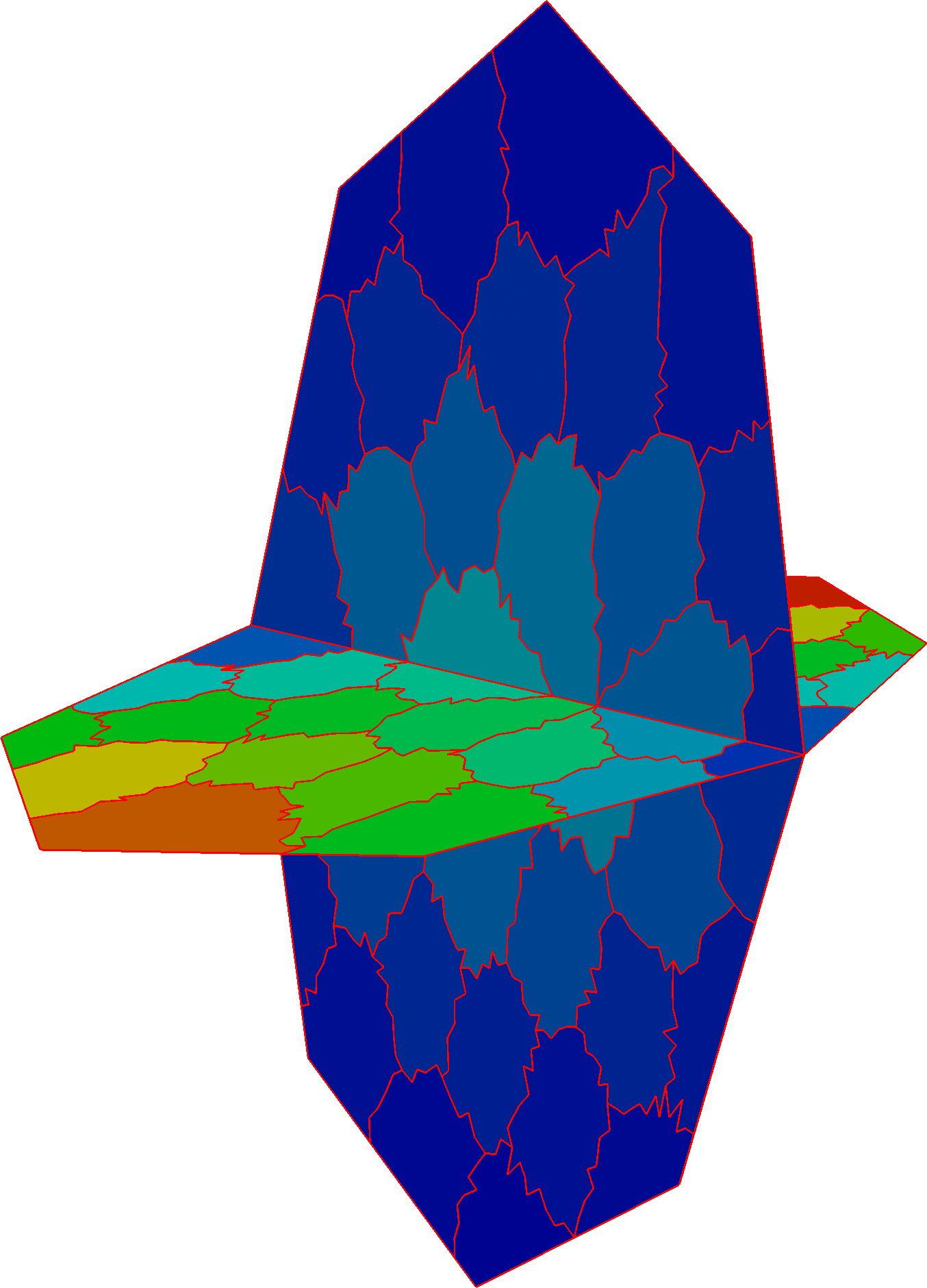}%
    \caption{Example of meshes and numerical solutions for the example in
    \cref{subsubsec:two}. From the left: triangular mesh, coarse mesh with
    $c_{depth}=2$, coarse mesh with $c_{depth}=4$, and coarse mesh with
    $c_{depth}=5$.
    The solution has range in $[0,3.6]$ with a ``Blue to Red Rainbow'' colour
    map.}%
    \label{fig:sol2}
\end{figure}
We assume pressure boundary condition on both $\partial \Omega_1$ and $\partial
\Omega_2$, unit effective permeability $\FIX{\kappa}{\lambda}$ and source as
\begin{gather} \label{eq:exact_sol_source}
    \begin{gathered}
        \FIX{q}{f}(x,y,z) =
        \begin{dcases*}
            8y( 1 - y ) - 8( z - 1 )^2 & for $z\geq 0$ \\
            8y( 1 - y ) - 8( z + 1 )^2 & for $z< 0$
        \end{dcases*} \text{in } \Omega_1\\
        \FIX{q}{f}(x,y,z) =
        \begin{dcases*}
            8y( 1 - y ) - 8( x + \zeta )^2 & for $x\geq 0$ \\
            8y( 1 - y ) - 8( x - \zeta )^2 & for $x< 0$
        \end{dcases*} \text{in } \Omega_2
    \end{gathered}
\end{gather}
With this data the exact solutions $p_{ex}$ and $\bm{u}_{ex}$ are
\begin{subequations} \label{eq:exact_sol}
\begin{gather}
    \begin{gathered}
        p_{ex}(x,y,z) =
        \begin{dcases*}
            4y(1-y)(z-1)^2 & for $z \geq 0$\\
            4y(1-y)(z+1)^2 & for $z < 0$
        \end{dcases*} \text{in } \Omega_1 \\
        p_{ex}(x,y,z) =
        \begin{dcases*}
            4y(1-y)(x+\zeta)^2 & for $x \geq 0$\\
            4y(1-y)(x-\zeta)^2 & for $x < 0$
        \end{dcases*} \text{in } \Omega_2
    \end{gathered}
\end{gather}
\begin{gather}
    \begin{gathered}
        \bm{u}_{ex} (x,y,z) =
        \begin{dcases*}
            \left[ 0; \,\,\, 8y(z-1)^2; \,\,\, -8y(1-y)(z-1) \right]^\top & for $z \geq 0$\\
            \left[ 0; \,\,\, 8y(z+1)^2; \,\,\, -8y(1-y)(z+1) \right]^\top & for $z < 0$
        \end{dcases*} \text{in } \Omega_1\\
        \bm{u}_{ex} (x,y,z) =
        \begin{dcases*}
            \left[ -8y(1-y)(x+\zeta); \,\,\,
            8y(x+\zeta)^2; \,\,\, 0 \right]^\top & for $x \geq 0$\\
            \left[ -8y(1-y)(x-\zeta); \,\,\, 8y(x-\zeta)^2; \,\,\, 0 \right]^\top & for $x < 0$
        \end{dcases*} \text{in } \Omega_2,
    \end{gathered}
\end{gather}
\end{subequations}
we assume $\zeta = 1$.
We consider four different families of discretization to analyse the order of
convergence for both $p$ and $\Pi_0 \bm{u}$ in both fractures.  We have:
triangular grids, ``coarse grids'' with $c_{depth} =2$,
$c_{depth}=4$, and $c_{depth} = 5$. In all the coarse cases we apply the
algorithm from \cref{sec:coarsening} on triangular grids.
See \cref{fig:sol2} as an example.  It is worth to notice that in the case of
``coarse grids'' the cells may not be convex and even they may be not
star-shaped.
\begin{figure}[tbp]
    \centering
    \includegraphics{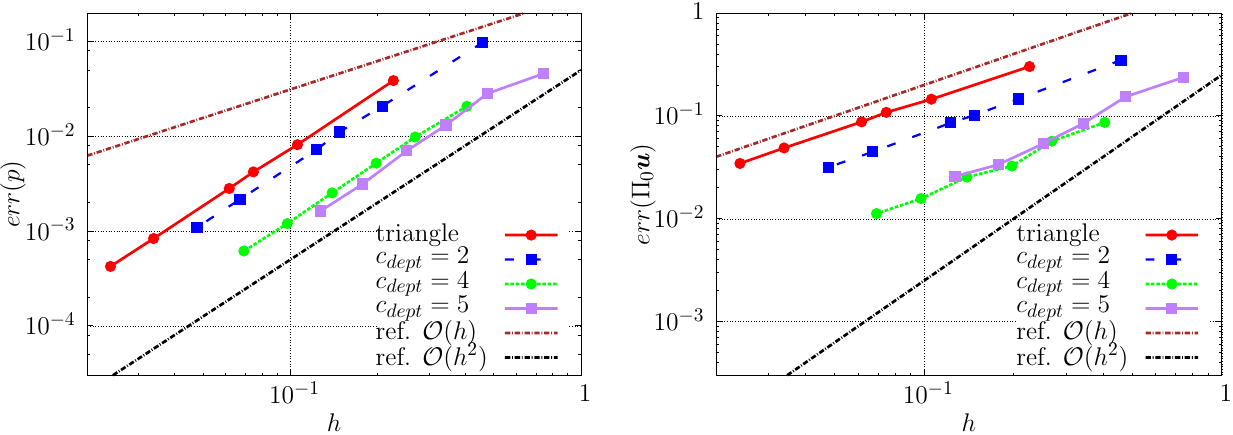}
    \caption{In the figure the error decay for both $p$, in the left plot, and
        $\Pi_0 \bm{u}$ for the first example in \cref{subsec:convergence}. We add
        the references $\mathcal{O}(h)$ and
        $\mathcal{O}(h^2)$ to aid the comparison.}%
    \label{fig:error2}
\end{figure}
In \cref{fig:error2} we report the $L^2$-relative errors for both the projected
velocity $\Pi_0 \bm{u}$ and the pressure $p$. We observe that the error decay
for the latter is quadratic and the former is linear with respect to the mesh
size. Moreover, for the same $h$, the error is higher for the triangular grid
and smaller when $c_{depth}$ increases. One possible explanation is that in the
first case each pressure \textit{d.o.f.} is linked with 3 velocity \dofs, while
for $c_{depth}=2$ the average number is 6, for $c_{depth}=4$ we have 20 edges,
and for $c_{depth}=5$ the average number of edges is 37. This fact can help the
accuracy for the meshes with more velocity \dofs. \FIX{}{Note that the grid size
$h$ reported is the average mesh size, which in some cases exhibits an erratic
reduction compared to the steady decrease of the grid size in the triangular
grid.} \cref{tag:errors_second} in \cref{sec:appendix} reports the detailed
values for this example, as well as the sparsity of the matrix, the minimum and
maximum values of the solution, number of edges for each cell.

\begin{remark}
    In this test the meshes at the one co-dimensional intersection are matching
    just because the procedure to construct the two meshes is the same, except
    the mapping from the reference domain.
\end{remark}


\subsubsection{Two \FIX{F}{f}ractures with \FIX{I}{i}ntersection \FIX{F}{f}low} \label{subsubsec:two_inters}


In this example we analyse the error decay for the coupled system \cref{pb:dc},
where two fractures intersect. The fractures $\Omega_1$ and $\Omega_2$ are the
same as the test shown in \ref{subsubsec:two}, while the intersection $\gamma$
is defined as $\gamma = \left\{ (x,y,z) \in \mathbb{R}^3: x = 0,\, z = 0, \, y
\in [0,1] \right\} $. The analytical solution of the problem
$p_{ex}$ and $\bm{u}_{ex}$ is
\cref{eq:exact_sol} with $\zeta = -1$ for the fractures and
$\hat{p}_{ex}$ and $\hat{\bm{u}}_{ex}$ for the intersection
\begin{gather*}
    \hat{p}_{ex}(x,y,z) = 5y(1-y) \quad \text{in } \gamma
    \quad \text{and} \quad
    \hat{\bm{u}}_{ex}(x,y,z) = \left[ 0; \,\,\, 5+10y; \,\,\,0 \right]^\top
    \quad \text{in } \gamma.
\end{gather*}
We consider pressure boundary conditions on both fractures as well as on the
intersections. The effective permeability for the fractures and for the
intersection is unitary, while the normal effective permeability for the
intersection is equal to $\FIX{\widetilde{\kappa}}{\widetilde{\lambda}} = 8$.
The source terms are chosen
as \cref{eq:exact_sol_source} with $\zeta = -1$ for the fractures and
$\FIX{\hat{q}}{\hat{q}}$
for the intersection
\begin{gather*}
    \FIX{\hat{q}}{\hat{q}}(x,y,z) = 10 + 32y(1-y) \quad \text{in } \gamma.
\end{gather*}
We consider a family of triangular meshes and two of coarse meshes. The latter
are constructed from a triangular mesh with different level of coarsening:
$c_{depth} = 2$ and $c_{depth} = 4$.
\cref{fig:sol3} shows a graphical representation of the solution obtained, where also
the pressure in the intersection is reported.
\begin{figure}[tbp]
    \centering
    \includegraphics[width=0.25\textwidth]{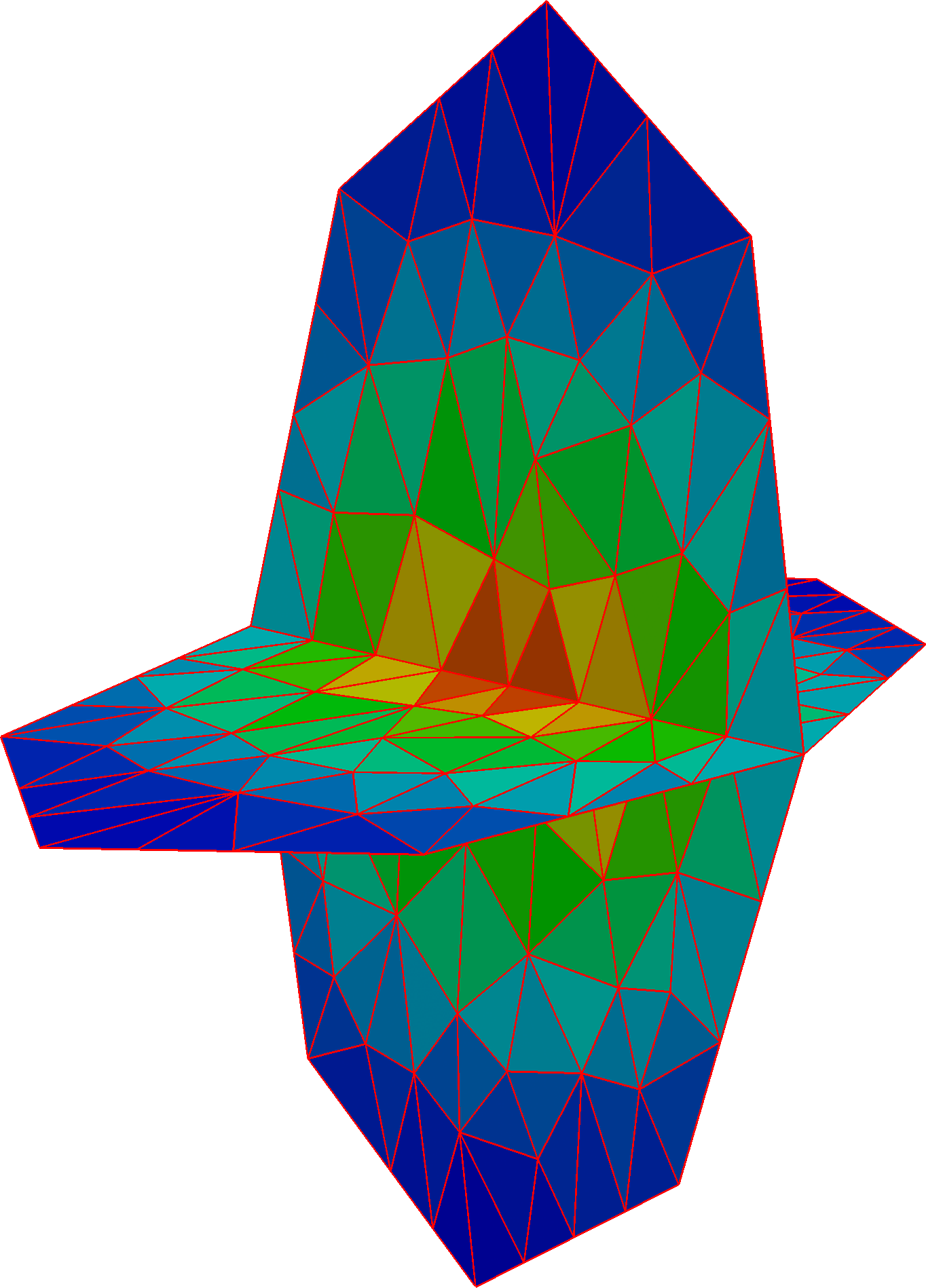}%
    \hfill%
    \includegraphics[width=0.25\textwidth]{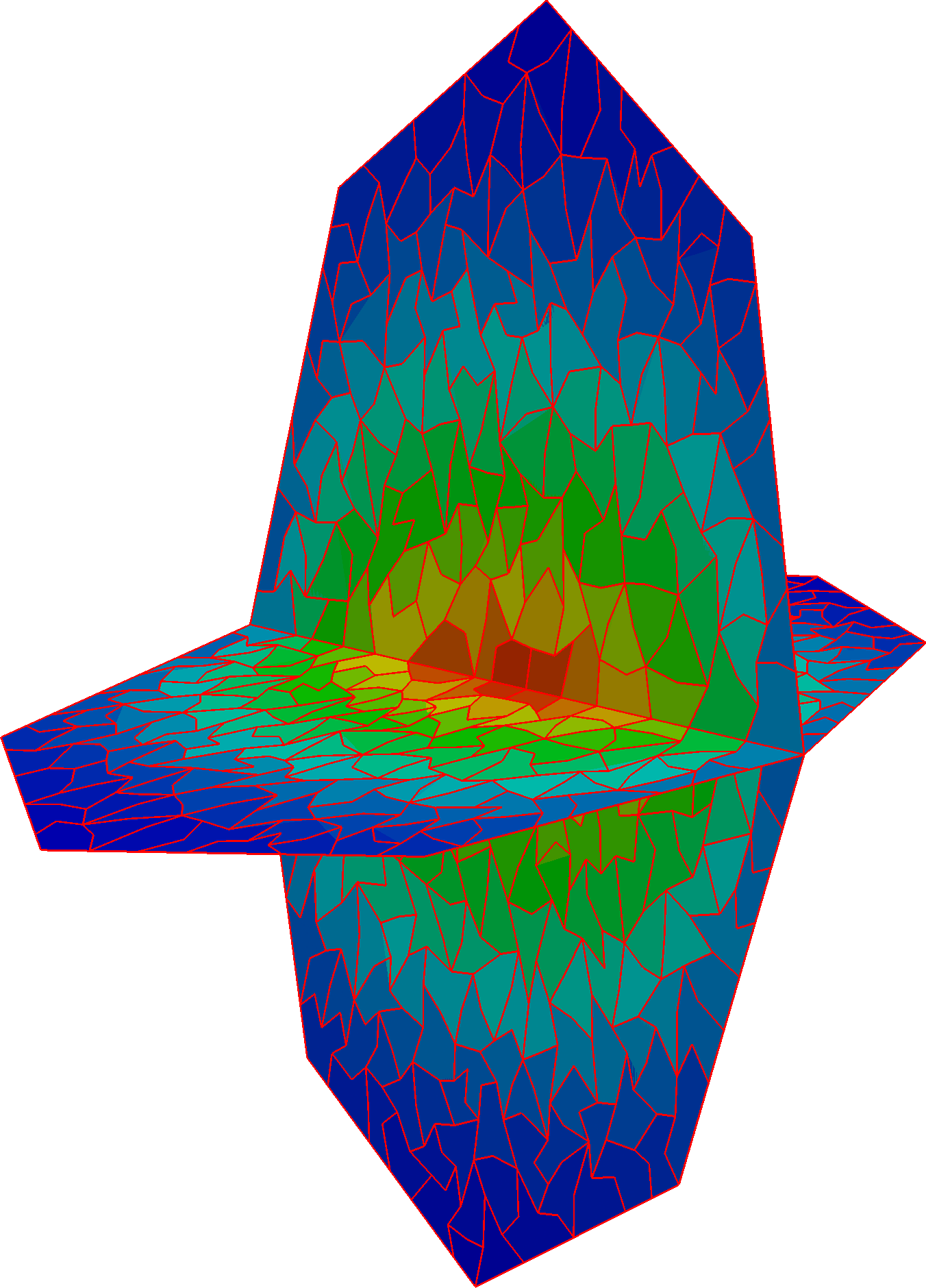}%
    \hfill%
    \includegraphics[width=0.25\textwidth]{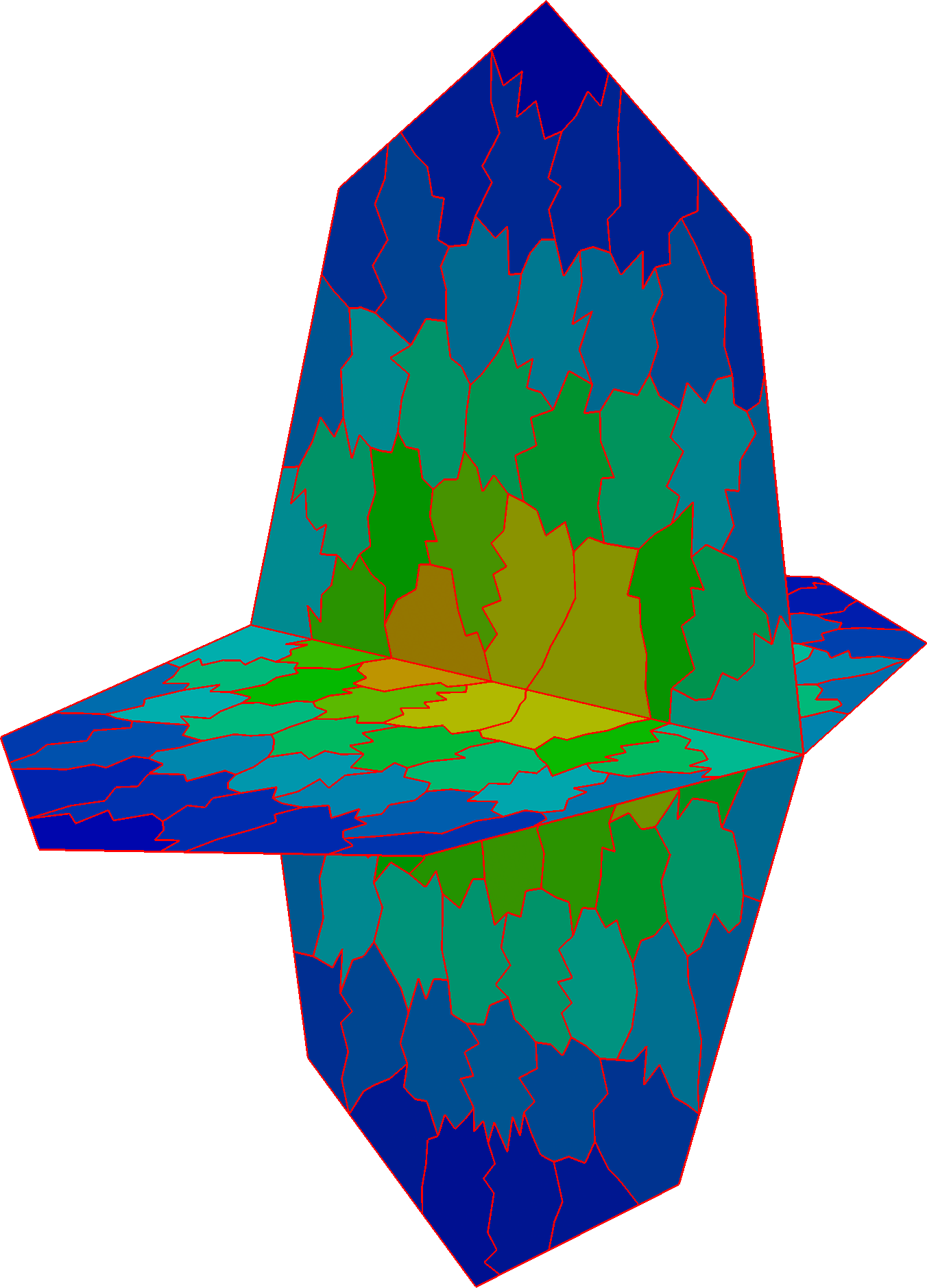}%
    \hfill%
    \includegraphics[width=0.015\textwidth]{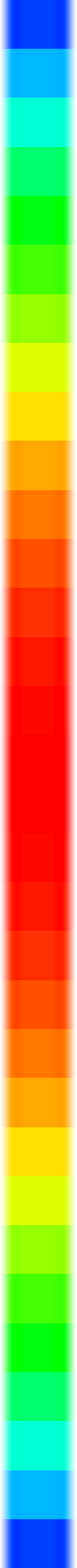}%
    \hfill%
    \caption{Example of meshes and numerical solutions for the example in
    \cref{subsubsec:two_inters}. From the left: triangular mesh, coarse mesh with
    $c_{depth}=2$, coarse mesh with $c_{depth}=4$, and pressure in $\gamma$ from
    the $c_{depth} = 2$ case.
    The solution has range in $[0,1]$ in the fractures and in $[0,1.25]$ in the intersection,
    a ``Blue to Red Rainbow'' colour map is used.}%
    \label{fig:sol3}
\end{figure}
In \cref{fig:error3} we report the $L^2$-relative errors for both the projected
velocity $\Pi_0 \bm{u}$ and the pressure $p$. We observe that the error decay
for the latter is quadratic and the former is linear with respect to the mesh
size, for the three families of mashes.
\begin{figure}[tbp]
    \centering
    \includegraphics{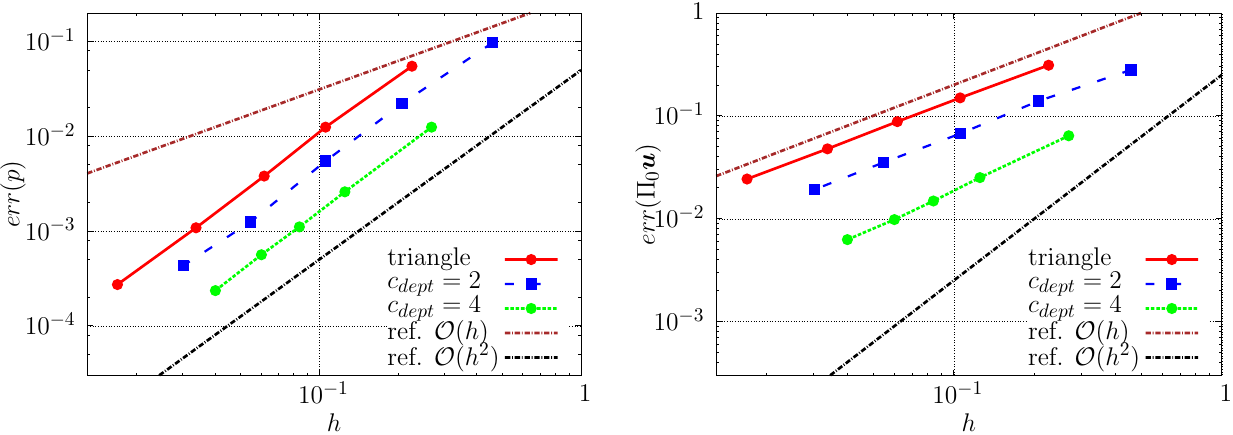}
    \caption{In the figure the error decay for both $p$, in the left plot, and
        $\Pi_0 \bm{u}$ for the third example in \cref{subsec:convergence}. We add
        the references $\mathcal{O}(h)$ and
        $\mathcal{O}(h^2)$ to easy the comparison.}%
    \label{fig:error3}
\end{figure}
In \cref{fig:error3_I} we report the
$L^2$-relative errors for both the projected velocity $\hat{\Pi}_0 \hat{\bm{u}}$ and the
pressure $\hat{p}$ for the solution in the intersection $\gamma$. The mesh size is now
referred to the intersection mesh. Also in this case
the slope of the pressure error is nearly quadratic for the three families, but
also the velocity (especially for the triangular and coarse-$c_{depth} = 4$
families) has an error that behave quadratically. This is a possible super-convergence
property of the numerical scheme in one space dimension.
\begin{figure}[tbp]
    \centering
    \includegraphics{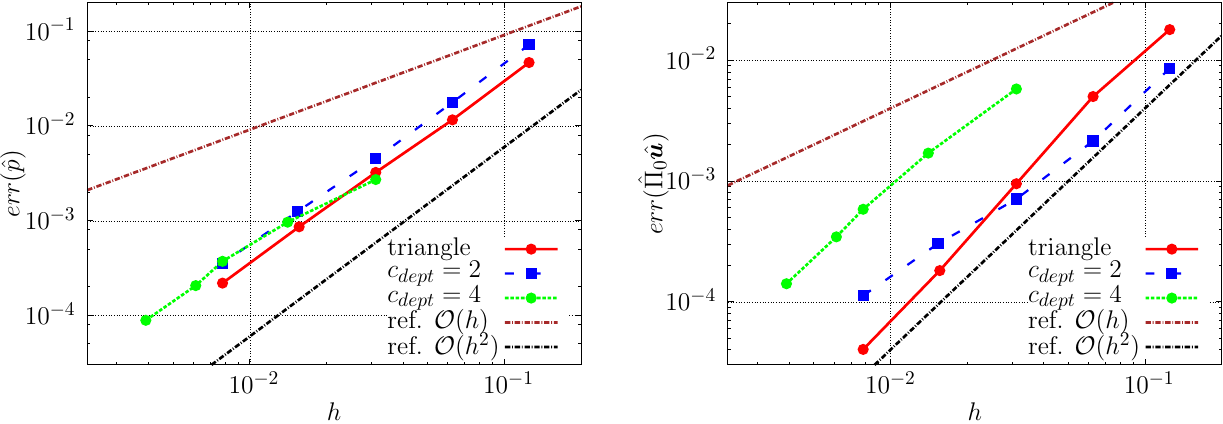}
    \caption{In the figure the error decay for both $\hat{p}$, in the left plot, and
        $\hat{\Pi}_0 \hat{\bm{u}}$ for the third example in \cref{subsec:convergence}. We add
        the references $\mathcal{O}(h)$ and
        $\mathcal{O}(h^2)$ to easy the comparison.}%
    \label{fig:error3_I}
\end{figure}
\cref{tag:errors_third} and \cref{tag:errors_thirdInters} in
\cref{sec:appendix} reports the detailed values for this example, as well as the
sparsity of the matrix, the minimum and maximum values of the solution, number
of edges for each cell.

We conclude the section commenting that in all the tests performed the error decay,
for both the pressure and velocity, is coherent with the analyses presented in
\cite{BeiraoVeiga2016}.


\subsection{Modeling of intersection flow}\label{subsubsect:isect}
\FIX{}{%
Having established the convergence properties of the virtual element method with
and without flow in the intersection, we next focus on the modeling aspects of
\cref{pb:cc} and \cref{pb:dc}. To that end, we consider a network of four
fractures, defined by


\begin{eqnarray*}
\Omega_1 = &&\left\{ (x,y,z) \in \mathbb{R}^3: 0\leq x \leq 1, y=0, 0 \leq z \leq 1
\right\}\\
\Omega_2 = &rot\big(&\left\{ (x,y,z) \in \mathbb{R}^3: 1/5\leq x-z\leq 6/5, -1/5 \leq x+z\leq 4/5, y=0
\right\}, \\
& &2\pi/3, (1/2, 0, 0), (0, 0, 1) \big)\\
\Omega_3 =& rot\big(&\left\{ (x,y,z) \in \mathbb{R}^3: \nu\leq x \leq 1/2 - \nu, y=0,
1/2 + \nu\leq z \leq 1 + \nu,
\right\}, \\
& &\pi/6, (1/2, 0, 1/2), (1, 0, -1)\big) \\
\Omega_4 =& rot\big(&\left\{ (x,y,z) \in \mathbb{R}^3: 1/2 + \nu \leq x \leq 1 + \nu, y=0,
1/2 + \nu\leq z \leq 1 +\nu,
\right\}, \\
&&\pi/6, (1/2, 0, 1/2), (1, 0, -1)\big).
\end{eqnarray*}

Here we have introduced the short-form $\nu=1/(5\sqrt{2})$ and $rot(\omega,
\theta, v, w)$ denotes the rotation of domain $\omega$ with an angle $\theta$
around the line parametrized as $l(t)=v + wt$.  The rotation angles are
immaterial to the solution provided they are non-zero; the values chosen are
motivated by clarity of visualization.

A source of unit strength is assigned in $(1/2, 0, 1/2)$, that is, in
$\Omega_1$.  For the other fractures, we assign homogeneous Dirichlet
conditions on the faces \FIX{}{in the corners furthest away from the injection points}.
On the remaining boundary edges we impose no-flux
conditions.
In terms of sources and boundary conditions, the problem is symmetric with respect to
the fractures $\Omega_2$, $\Omega_3$ and $\Omega_4$.
We note that, due to the Neumann condition, this problem is outside
the class considered in \cref{sec:well_posedness}.

In the intersection between $\Omega_1$ and
$\Omega_2$, the effective normal permeability $\tilde{\lambda}_m$ is set to
$10^{-7}$. The intersection between $\Omega_1$ and $\Omega_3$ has tangential
effective permeability $\hat{\lambda}_k=10^{-10}$, effectively employing
\cref{pb:cc} here, while for the intersection of $\Omega_1$ and $\Omega_4$,
$\hat{\lambda}_k=10^{10}$.  All other permeabilities have unit value.
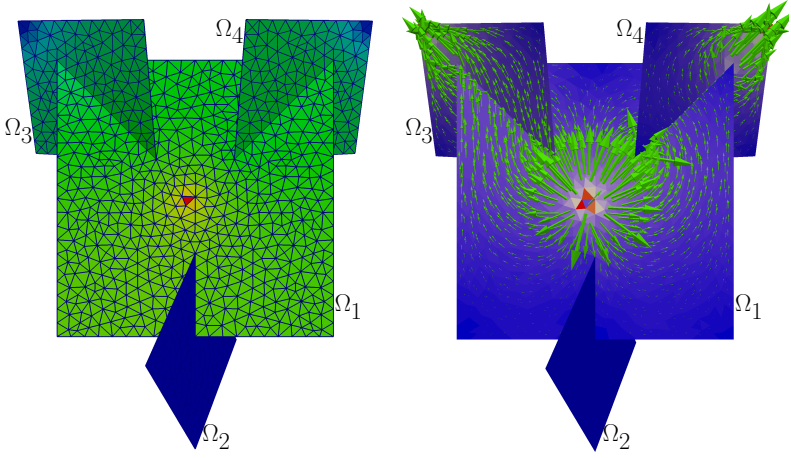
\begin{figure}
    \centering
    \resizebox{0.4\textwidth}{!}{\fontsize{45pt}{17.2}\selectfont%
\begingroup%
  \makeatletter%
  \providecommand\color[2][]{%
    \errmessage{(Inkscape) Color is used for the text in Inkscape, but the package 'color.sty' is not loaded}%
    \renewcommand\color[2][]{}%
  }%
  \providecommand\transparent[1]{%
    \errmessage{(Inkscape) Transparency is used (non-zero) for the text in Inkscape, but the package 'transparent.sty' is not loaded}%
    \renewcommand\transparent[1]{}%
  }%
  \providecommand\rotatebox[2]{#2}%
  \ifx\svgwidth\undefined%
    \setlength{\unitlength}{400.00000961bp}%
    \ifx\svgscale\undefined%
      \relax%
    \else%
      \setlength{\unitlength}{\unitlength * \real{\svgscale}}%
    \fi%
  \else%
    \setlength{\unitlength}{\svgwidth}%
  \fi%
  \global\let\svgwidth\undefined%
  \global\let\svgscale\undefined%
  \makeatother%
  \begin{picture}(1,1.15232976)%
    \put(0,0){\includegraphics[width=\unitlength,page=1]{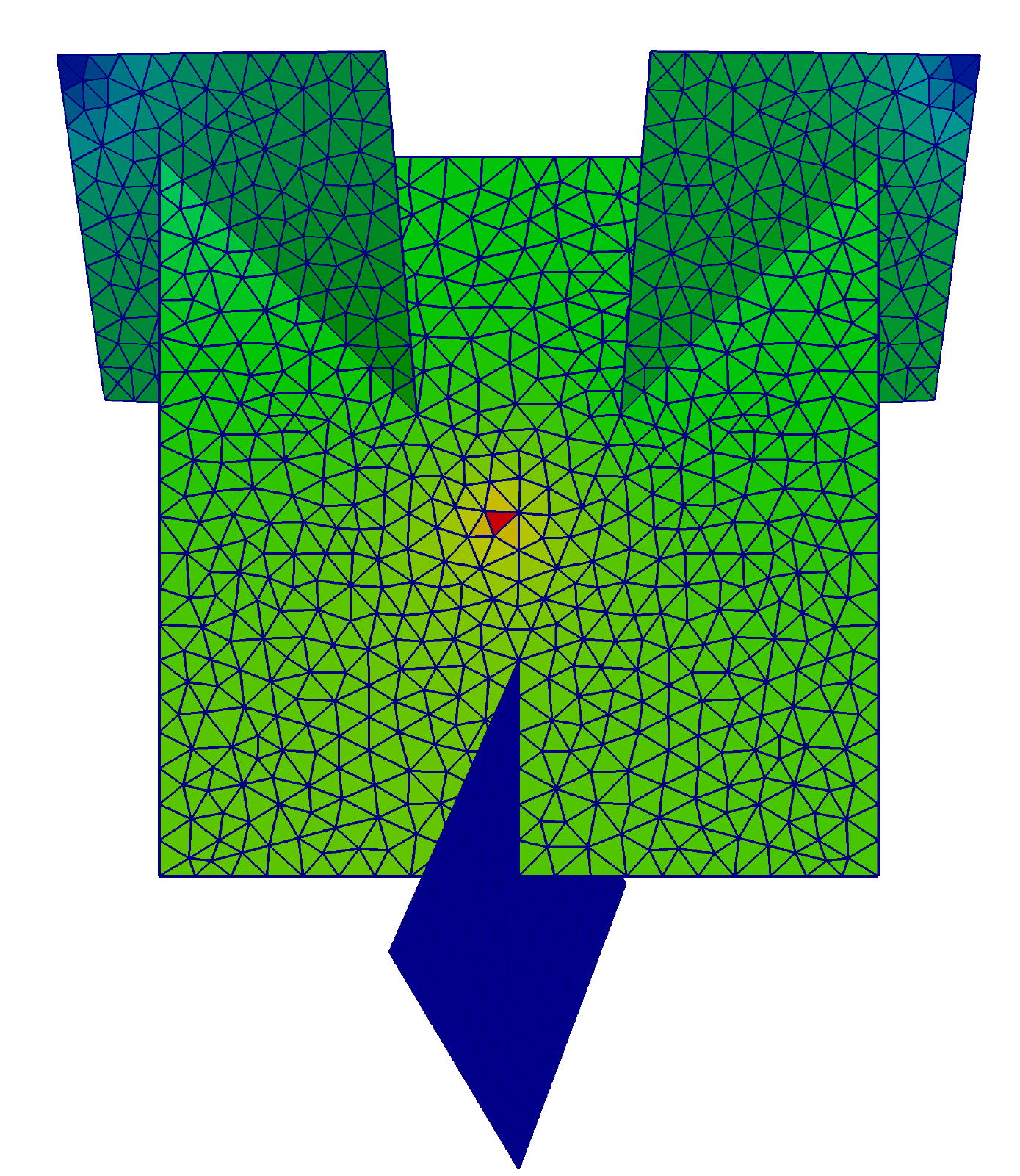}}%
    \put(0.52676236,0.02560891){\color[rgb]{0,0,0}\makebox(0,0)[lb]{\smash{$\Omega_2$}}}%
    \put(0.86434599,0.35819347){\color[rgb]{0,0,0}\makebox(0,0)[lb]{\smash{$\Omega_1$}}}%
    \put(0.02241929,0.80772125){\color[rgb]{0,0,0}\makebox(0,0)[lb]{\smash{$\Omega_3$}}}%
    \put(0.56109824,1.06232108){\color[rgb]{0,0,0}\makebox(0,0)[lb]{\smash{$\Omega_4$}}}%
  \end{picture}%
\endgroup%
        }
    \resizebox{0.4\textwidth}{!}{\fontsize{45pt}{17.2}\selectfont%
\begingroup%
  \makeatletter%
  \providecommand\color[2][]{%
    \errmessage{(Inkscape) Color is used for the text in Inkscape, but the package 'color.sty' is not loaded}%
    \renewcommand\color[2][]{}%
  }%
  \providecommand\transparent[1]{%
    \errmessage{(Inkscape) Transparency is used (non-zero) for the text in Inkscape, but the package 'transparent.sty' is not loaded}%
    \renewcommand\transparent[1]{}%
  }%
  \providecommand\rotatebox[2]{#2}%
  \ifx\svgwidth\undefined%
    \setlength{\unitlength}{400.00000961bp}%
    \ifx\svgscale\undefined%
      \relax%
    \else%
      \setlength{\unitlength}{\unitlength * \real{\svgscale}}%
    \fi%
  \else%
    \setlength{\unitlength}{\svgwidth}%
  \fi%
  \global\let\svgwidth\undefined%
  \global\let\svgscale\undefined%
  \makeatother%
  \begin{picture}(1,1.15232976)%
    \put(0,0){\includegraphics[width=\unitlength,page=1]{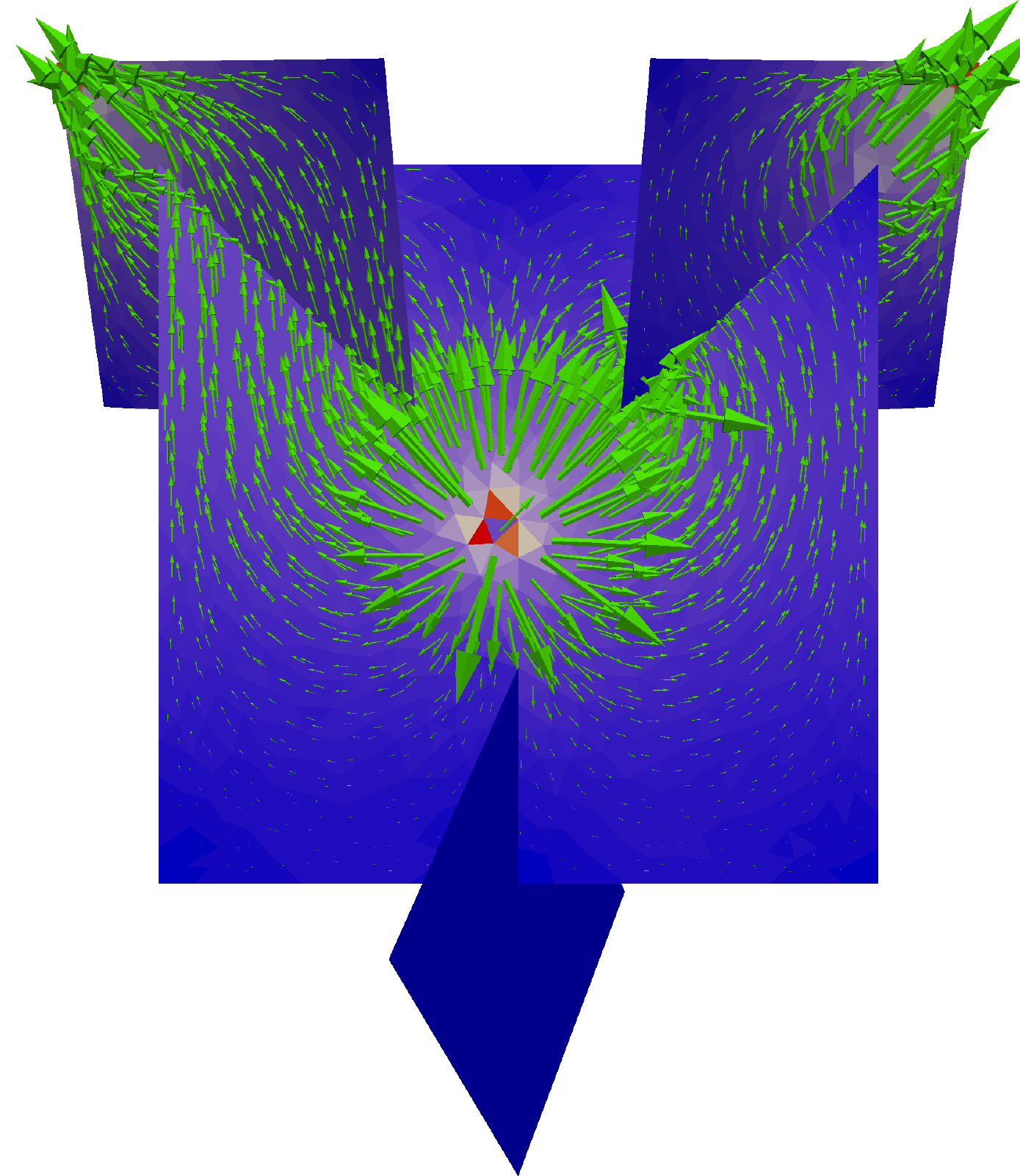}}%
    \put(0.52676236,0.02560885){\color[rgb]{0,0,0}\makebox(0,0)[lb]{\smash{$\Omega_2$}}}%
    \put(0.864346,0.35819352){\color[rgb]{0,0,0}\makebox(0,0)[lb]{\smash{$\Omega_1$}}}%
    \put(0.02241928,0.80772129){\color[rgb]{0,0,0}\makebox(0,0)[lb]{\smash{$\Omega_3$}}}%
    \put(0.56109824,1.0623211){\color[rgb]{0,0,0}\makebox(0,0)[lb]{\smash{$\Omega_4$}}}%
  \end{picture}%
\endgroup%
        }
    \caption{Left: Pressure profile obtained for the family of four intersecting
    fractures. Right: Cell-wise velocity magnitudes, and also velocity vectors.
    Arrows close to the source and the outlets are not shown.}
    \label{fig:four_fractures}
\end{figure}
Pressure and velocity profiles are shown in \cref{fig:four_fractures}. Due to
the low $\tilde{\lambda}_m$, $\Omega_2$ is effectively sealed off from the other
fractures, with a pressure that is virtually equal to the outlet value in the
entire fracture.  The pressure difference between $\Omega_3$ and $\Omega_4$ is
negligible. However, as seen from the velocity field, flow into $\Omega_4$ is
channelized into the intersection. For $\Omega_3$ there is no flow in the
intersection, but a wider sweep of the fracture plane itself compared to
$\Omega_4$.
}


\subsection{Realistic DFN} \label{subsubsec:realistic}


\FIX{The previous examples established the convergence properties of our numerical
methods.}{} In this final example, we illustrate the robustness of the approach
for general fracture geometries.  To that end, we consider a stochastically
generated network of 60 fractures using the software described in
\cite{FracSim3D}.  The network is shown in \cref{fig:ex_realistic_geom}.  The
fracture permeability is set to unity, and for simplicity we only consider the
model presented in \cref{pb:cc}.  Each fracture is discretized by 8 boundary
edges, we impose no-flux boundary condition on all the edges of the fractures
except for two of them. A pressure $p_1 = 1$ and $p_2 = 0$ is thus imposed.
Again we note that the no-flow conditions are not covered by the analysis in
\cref{sec:well_posedness}.

\begin{figure}[tbp]
    \centering
    \includegraphics{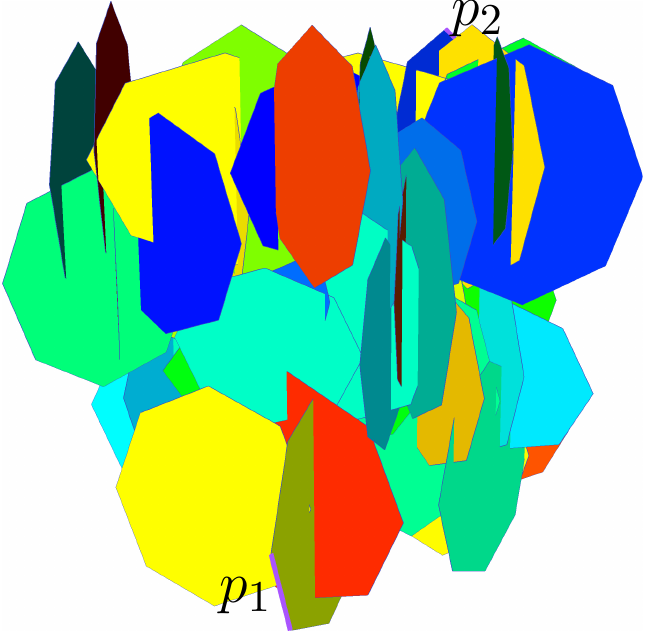}
    \caption{Geometry used for the example in \cref{subsubsec:realistic}. Each
    fracture is coloured by its identification number and the pressure boundary
    conditions are highlight using magenta segments.}%
    \label{fig:ex_realistic_geom}
\end{figure}

The stochastic generation of the fracture network leads to intersection
configurations \FIX{that in general may be difficult in terms of meshing.  This
is however not the case for our approach, where the computational mesh is
defined by triangulations of the individual fratures independently, and then
applying the coarsening algorthm with $c_{depth}=2$.  The resulting  grid, with
$24444$ cells and $78574$ edges, is depicted in \cref{fig:ex_realistic}.} {that
may pose difficulties for meshing, in particular if the fracture planes are
coupled during grid creation.  Our approach of independent gridding of each
fracture avoids the linkage of nodes in different fracture planes, but it still
create small cells close to short constraints and almost parallel constraints.
The coarsening algorithm, here applied with $c_{depth}=2$ increases the cell
size. The polygonal mesh has $24444$ cells and $78574$ edges, compared with
$106809$ cells for the triangular grid. The impact of the coarsening is
illustrated in \cref{fig:coarsening_realistic_zoom}.
\begin{figure}
    \centering
    \includegraphics[width=0.4\textwidth]{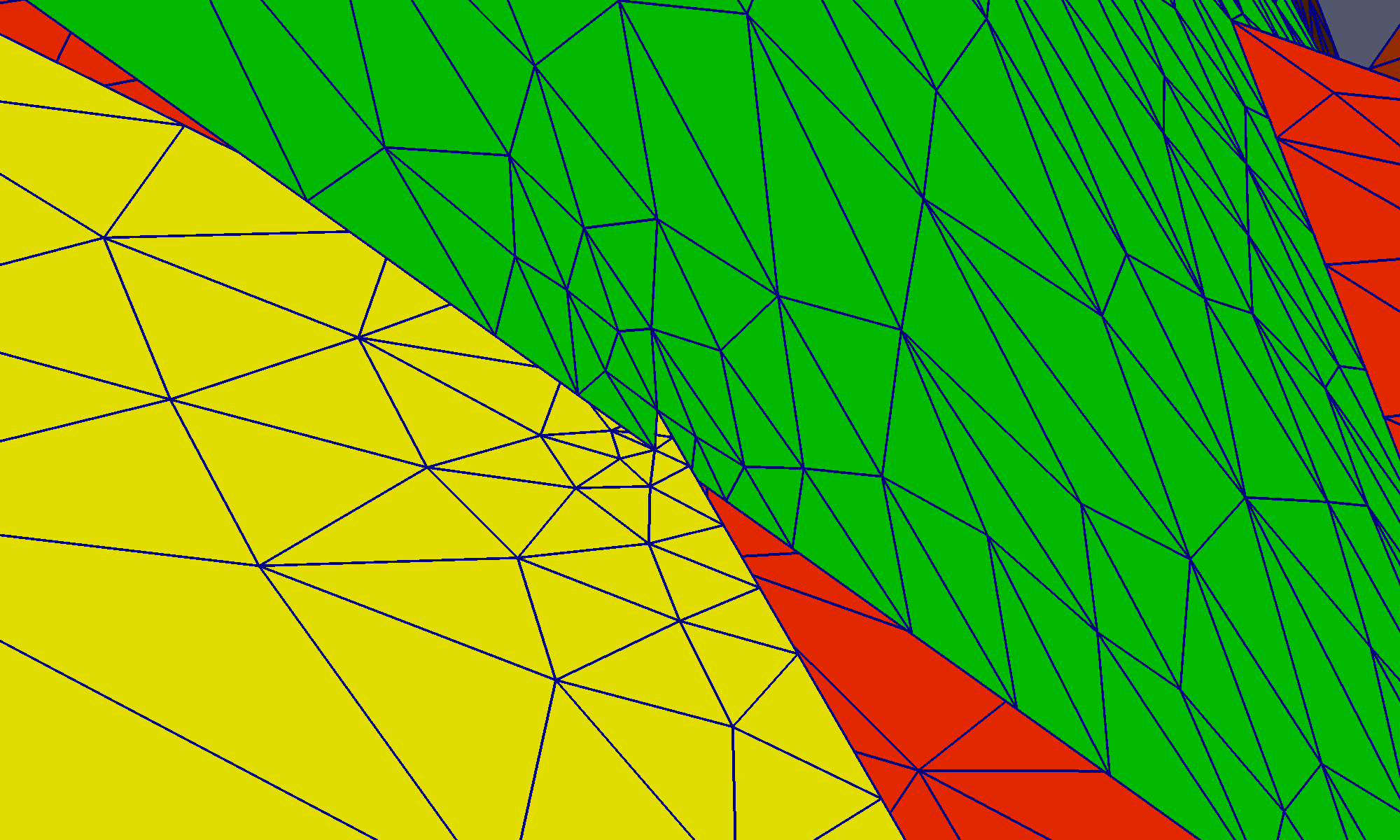}
    \includegraphics[width=0.4\textwidth]{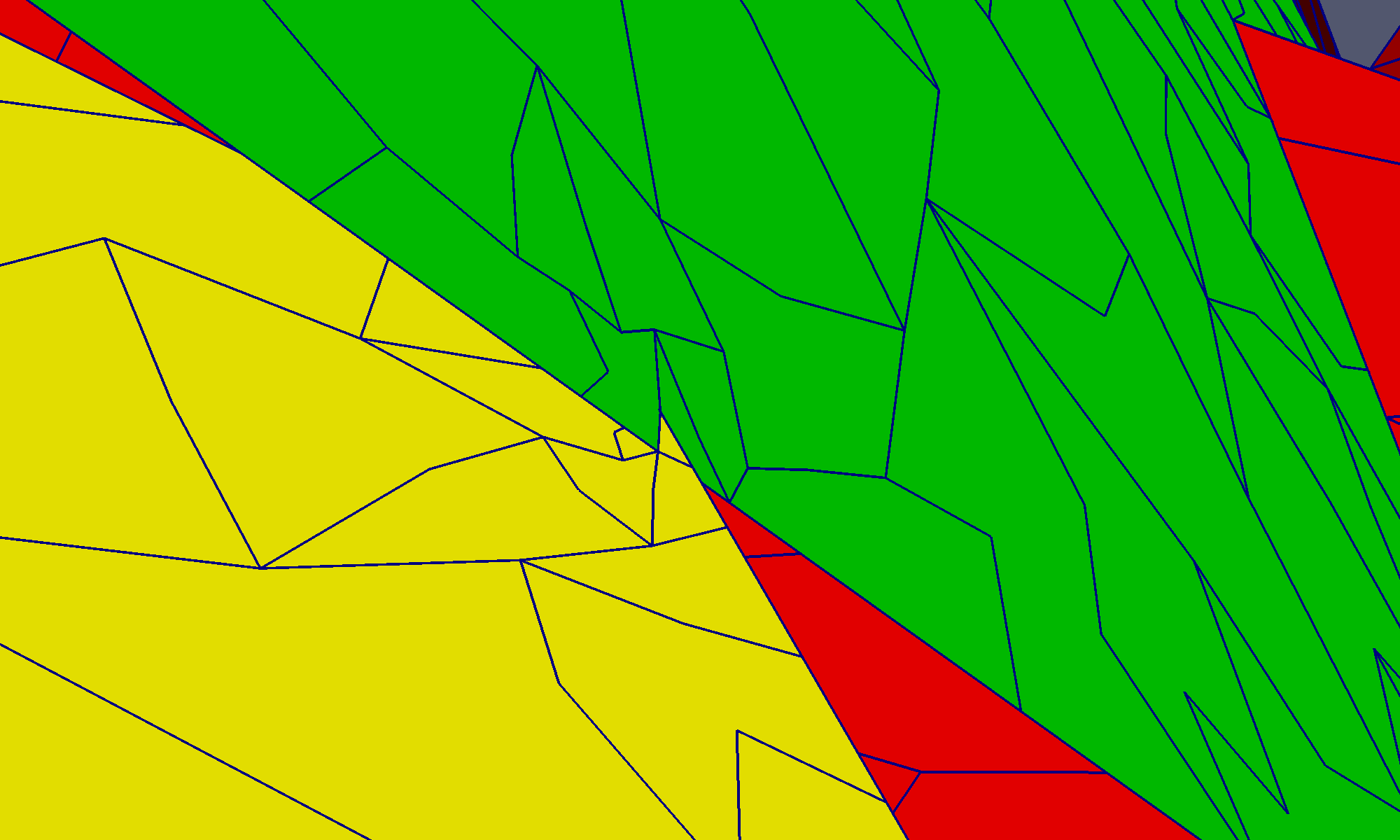}
    \includegraphics[width=0.4\textwidth]{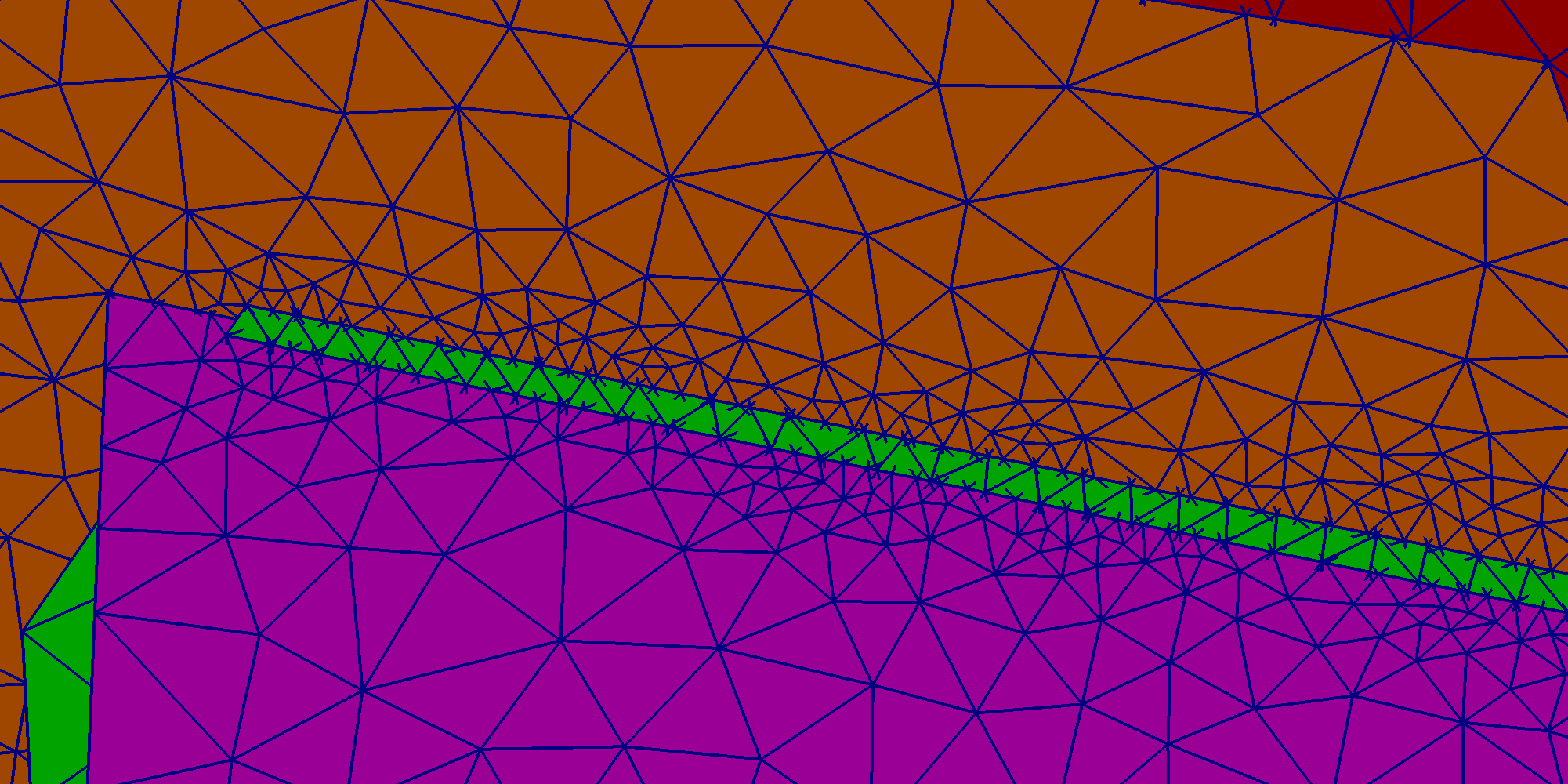}
    \includegraphics[width=0.4\textwidth]{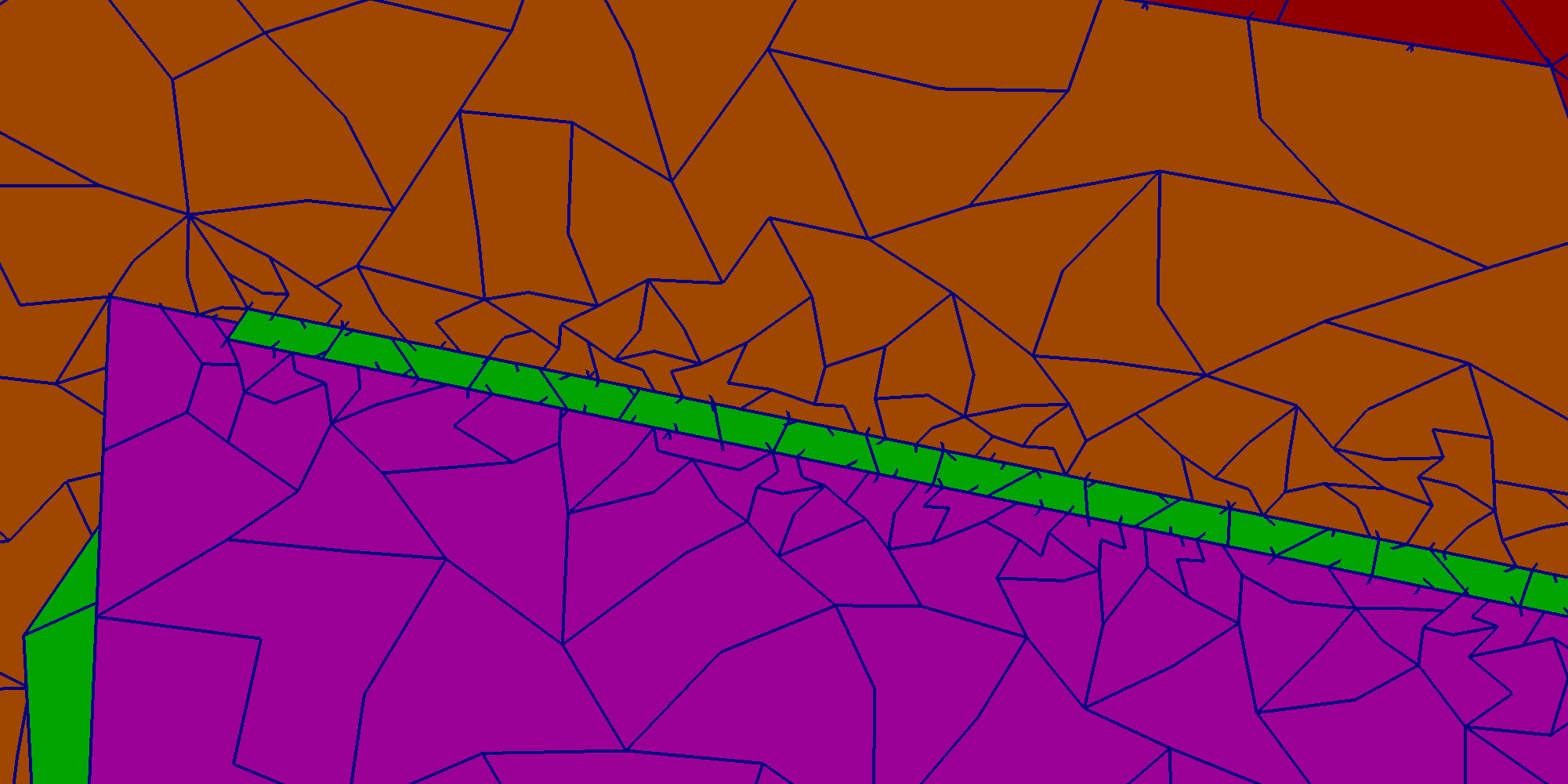}
    \caption{Details of the mesh used in \cref{subsubsec:realistic}. Left
    column: Original triangular grid. Right: Polygonal grid obtained after
    coarsening.}%
    \label{fig:coarsening_realistic_zoom}
\end{figure}
}

\begin{figure}[tbp]
    \centering
    \includegraphics[width=0.495\textwidth]{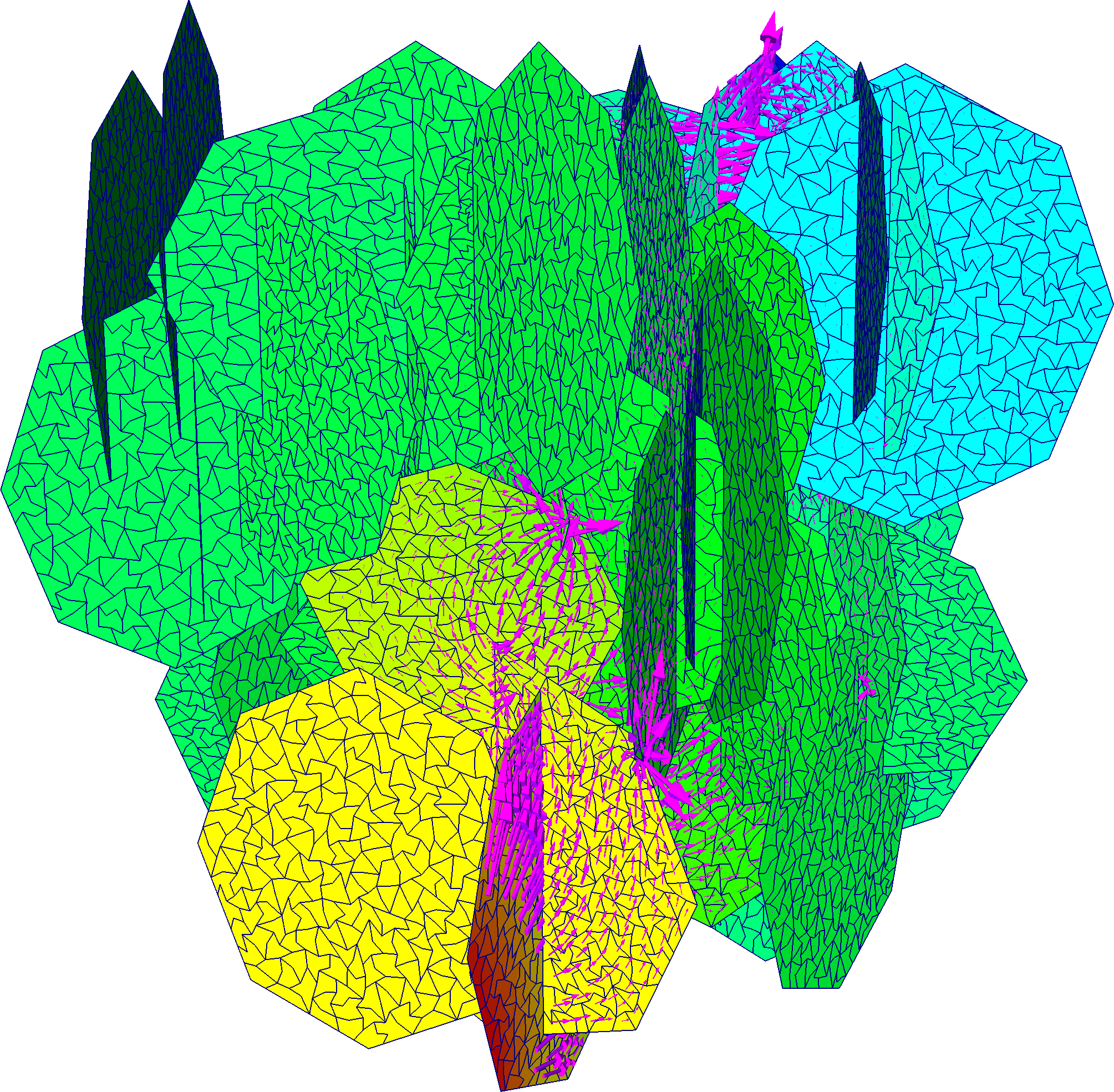}%
    \hfill%
    \raisebox{1cm}{\includegraphics[width=0.495\textwidth]{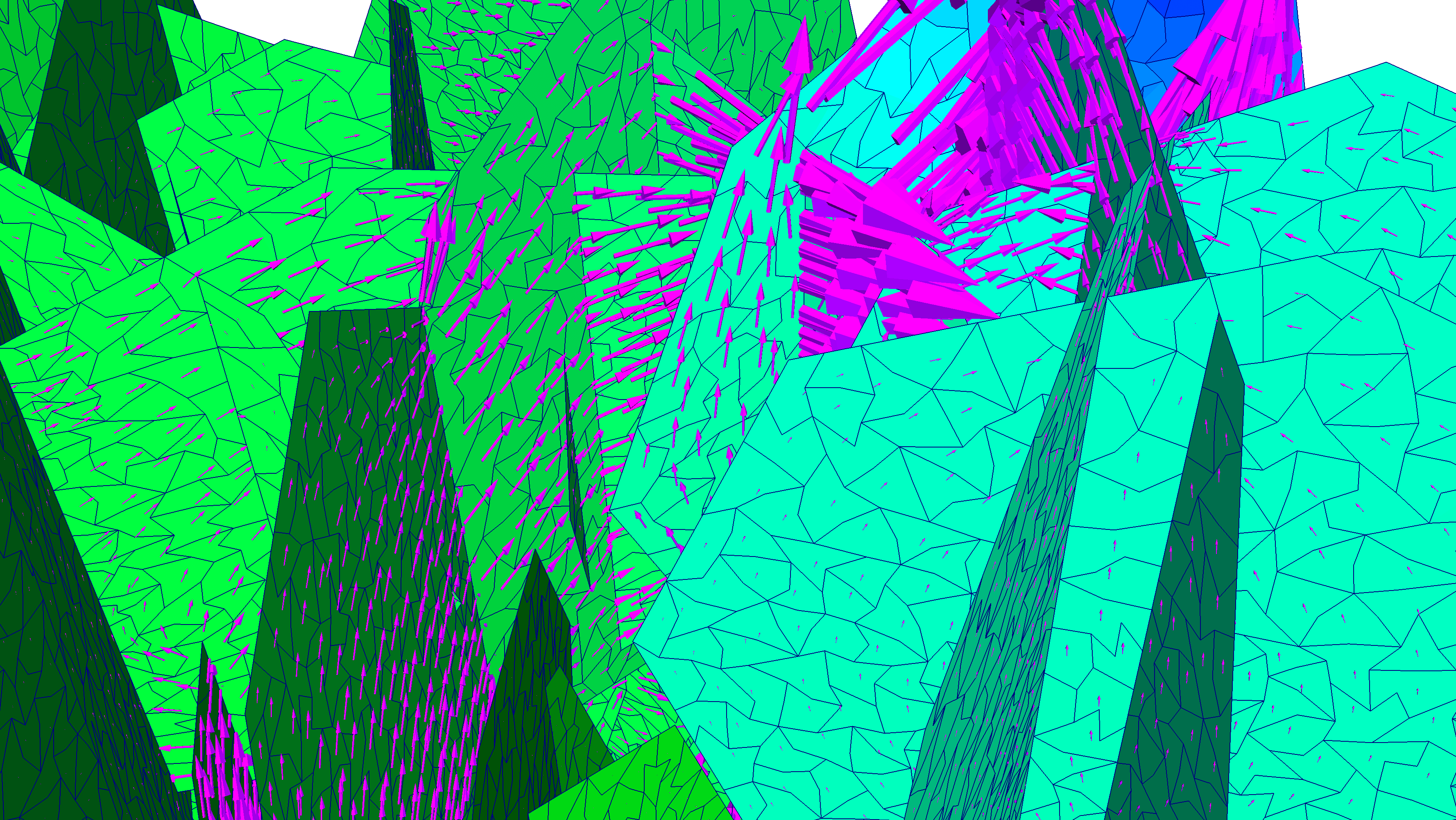}}%
    \caption{Solution obtained with a realistic discrete fracture network. The
    figure depicts both pressure and velocity, the latter represented by magenta
    arrows. In the right the full system is represented while in the left a
    detail close to the outflow boundary is shown.  The pressure has range in $[0,1]$ and
    a ``Blue to Red Rainbow'' colour map is used. The velocity is restricted to
    ease the interpretation and the arrows are scaled with the velocity
    magnitude.}%
    \label{fig:ex_realistic}
\end{figure}
\cref{fig:ex_realistic} shows the pressure solution and the projected
velocity field, together a zoom-in in the vicinity of the outflow.
As indicated by the figures, the numerical method captures high flow velocities in the
in- and outflow fractures, as well as interaction between the fracture planes.
\FIX{ We can conclude that the proposed scheme effectively handles also complex
fracture geometries and the coarsening algorithm is able to reduce the number of
degrees of freedom without affect the quality of the solution.} { It is
worthwhile to consider how other conservative methods could have discretized the
network in \cref{fig:ex_realistic_geom}. A key property of the virtual element
method, and the closely related mimetic finite difference method, is the ability
to handle hanging nodes and general polygon grids.  A discretization by \eg
mixed finite elements would have required not only the triangular grid with
about four times as many cells, but also techniques to eliminate hanging nodes
at the intersections, which may further increase the cell number. Another
possibility is to consider the class of finite volume methods which are,
generally, more robust with respect to the shape of the cells. See for instance,
\cite{brenner:hal-01192740}.}

%
%


\section{Conclusions} \label{sec:conclusions}


In this paper we presented a novel approximation for a complex network of
fractures using reduced models to describe fracture flows. The reduced models
are a reasonable approximation when the thickness of each fracture is some orders
of magnitude smaller then its other characteristic sizes and smaller than the
typical size of the surrounding rock matrix. In particular two models are
presented to describe the flow at the intersection of fractures, with the
possibility to allow a tangential flow along the intersection. \FIX{Moreover
the choice of the virtual finite element method allows us to handle, in a much
robust and accurate way, different problem configurations and, partially, solve
some of the difficulties related to the mesh generation.}{The virtual element
method's ability to handle general polygons with hanging nodes allows us to
relax constraints on the gridding of fracture intersections, and also apply grid
coarsening.} In the examples we saw the error decay of the solution, both
pressure and Darcy velocity, behave as expected in all the cases. Finally we
noticed that the solution behaves as expected and no evidence of serious
contraindications is present, thus the algorithm presented seems a promising
tool for the numerical approximation for this type of problems.

%
%


\section*{Acknowledgment}

    We acknowledge financial support for the ANIGMA project from the Research
    Council of Norway (project no. 244129/E20) through the ENERGIX program.
    The authors warmly thank the others components of the
    ANIGMA project team:
    Eivind Bastesen,
    Inga Berre,
    Simon John Buckley,
    Casey Nixon,
    David Peacock,
    Atle Rotevatn,
    P{\aa}l N{\ae}verlid S{\ae}vik,
    Luisa F. Zuluaga.
    The authors wish to thank also:
    Louren\c{c}o Beir\~{a}o da Veiga,
    Wietse Boon,
    Franco Dassi,
    Luca Formaggia,
    Eren U{\c c}ar.

%
%

\appendix

\section{Details of coarsening algorithm} \label{sec:coarseapp}

Here we describe the coarsening algorithm in some detail.
For further information, confer \cite{Trottenberg2001}.
Define two parameters $\epsilon_{str} \in (0,1)$ and
$c_{depth} \in \mathbb{N}$. An element $i$, called $C$-element,
in the original mesh will be connected, in the coarse grid, to some elements
in the set $N_i$, called $F$-elements, those affect the error at $i$ most. Since the
algorithm is matrix base the previous request aim to find the $F$-elements $j$
such that $\abs{[A]_{ij}}$ is the largest in some sense. We say that a degree of
freedom $i$ is ``strongly negatively coupled'' (SNC) to $j$ if
\begin{gather*}
    - [A]_{ij} \geq \epsilon_{str} \max_{[A]_{ik}<0} \abs{[A]_{ik}}
    \quad \text{and we have} \quad
    S_i \defeq \left\{ j \in N_i: i \text{ is SNC to } j \right\}.
\end{gather*}
Since the relation of being SNC is not symmetric we introduce $S_i^\top \defeq
\left\{ j: i \in S_j \right\}$, which is the set of elements which are strongly
coupled to $i$. Starting from a $C$-element $i$ the algorithm connect all the
$F$-elements $j$ that are strongly coupled to $i$ to create a coarse element.
The process is repeated considering another $C$-element of the mesh. As
highlight in \cite{Trottenberg2001}, the selection of a new $C$-element $i$ is
based on the measure of importance $\lambda_i$ in the set of undecided elements
$U$, \ie nor $C$ or $F$-elements. We define
\begin{gather*}
    \lambda_i \defeq \sharp\left( S_i^\top \cap U \right) + 2 \sharp \left(
    S_i^\top \cap F \right) \quad \text{with} \quad i \in U,
\end{gather*}
where $F$ is the set of $F$-elements. The next $C$-element will be the element
with greater $\lambda_i$. The presented algorithm is repeated for
$c_{depth}$-times where the initial mesh is the coarse mesh obtained at the
previous step. The greater $c_{depth}$ the bigger will be elements in the final coarse
grid.

\begin{remark}
    As point out in \cite{Trottenberg2001} a reasonable value for
    $\epsilon_{str}$ is $0.25$. We note also that in the case of permeability
    with strong anisotropy the final mesh is able to better represent the
    underling physic of the problem. See \cref{subsec:coarsening} for a
    discussion.
\end{remark}
\begin{remark}
    To avoid pathological shapes of elements during the coarsening process at
    the tip of the intersection we impose a priori the triangles, which share the
    tip of the intersection and have an edge on the intersection, as $C$-elements.
\end{remark}

\section{Tables} \label{sec:appendix}

In this section we present the detailed values for the examples in
\cref{sec:examples}. In \cref{tag:errors_first} are reported the values for the
test of \cref{subsubsec:single}. In \cref{tag:errors_second} are reported the
values for the test of \cref{subsubsec:two}.

\begin{table}[tbp] \footnotesize
    \begin{tabular}{|@{}*{1}{>{\centering\arraybackslash}p{.025\textwidth}}@{}*{9}{>{\centering\arraybackslash}p{.108\textwidth}@{}}|}
        \hline
        \multicolumn{10}{|c|}{Coarse mesh family}\\
        \hline
        $\sharp$ & $h$& $err(p)$ & $\mathcal{O}(p)$ & $err({\Pi_0 \bm{u}})$ &
        $\mathcal{O}(\Pi_0 \bm{u})$ & $\min p$ & $\max p$ & size & sparsity \\
        \hline
        1 & 3.162e-1 & 3.340e-2 & -     & 2.713e-1 & -      & 3.434e-3 & 1.228  & 152    & 1.158e-1 \\
        2 & 1.581e-1 & 9.774e-3 & 1.773 & 9.639e-2 & 1.493 & 4.958e-4 & 1.2642 & 584    & 3.321e-2 \\
        3 & 7.906e-2 & 2.537e-3 & 1.946 & 3.067e-2 & 1.652 & 7.474e-5 & 1.2947 & 2288   & 8.885e-3 \\
        4 & 3.953e-2 & 6.348e-4 & 1.998 & 9.360e-3 & 1.712 & 1.049e-5 & 1.3193 & 9056   & 2.292e-3 \\
        5 & 1.976e-2 & 1.577e-4 & 2.009 & 2.839e-3 & 1.721 & 1.393e-6 & 1.3308 & 36032  & 5.806e-4 \\
        6 & 9.882e-3 & 3.921e-5 & 2.008 & 8.742e-4 & 1.699 & 1.795e-7 & 1.3335 & 143744 & 1.467e-4 \\
        \hline
    \end{tabular}
    \begin{tabular}{|@{}*{1}{>{\centering\arraybackslash}p{.025\textwidth}}@{}*{9}{>{\centering\arraybackslash}p{.108\textwidth}@{}}|}
        \hline
        \multicolumn{10}{|c|}{Cartesian mesh family}\\
        \hline
        $\sharp$ & $h$& $err(p)$ & $\mathcal{O}(p)$ & $err({\Pi_0 \bm{u}})$ &
        $\mathcal{O}(\Pi_0 \bm{u})$ & $\min p$ & $\max p$ & size & sparsity \\
        \hline
        1 & 1.414e-1 & 4.099e-2 & -     & 1.936e-1 & -     & -7.991e-4 & 1.283  & 320   & 1.793e-2 \\
        2 & 7.071e-2 & 1.061e-2 & 1.949 & 5.923e-2 & 1.709 & -2.346e-5 & 1.3055 & 1240  & 4.646e-3 \\
        3 & 3.535e-2 & 2.682e-3 & 1.984 & 1.715e-2 & 1.788 & 5.599e-7  & 1.3211 & 4880  & 1.234e-3 \\
        4 & 1.768e-2 & 6.728e-4 & 1.995 & 4.807e-3 & 1.835 & 2.797e-7  & 1.3291 & 19360 & 3.162e-4  \\
        5 & 8.839e-3 & 1.684e-4 & 1.999 & 1.319e-3 & 1.866 & 4.767e-8  & 1.3329 & 77120 & 7.96e-5 \\
        6 & 4.419e-4 & 4.211e-5 & 1.999 & 3.567e-4 & 1.887 & 6.738e-9  & 1.335 & 307840 & 2.02e-5 \\
        \hline
    \end{tabular}
    \begin{tabular}{|@{}*{1}{>{\centering\arraybackslash}p{.025\textwidth}}@{}*{9}{>{\centering\arraybackslash}p{.108\textwidth}@{}}|}
        \hline
        \multicolumn{10}{|c|}{Random mesh family}\\
        \hline
        $\sharp$ & $h$& $err(p)$ & $\mathcal{O}(p)$ & $err({\Pi_0 \bm{u}})$ &
        $\mathcal{O}(\Pi_0 \bm{u})$ & $\min p$ & $\max p$ & size & sparsity \\
        \hline
        1 & 1.171e-1 & 3.4794e-2 & -     & 2.434e-1 & -     & -8.777e-4 & 1.2988 & 320    & 2.168e-2 \\
        2 & 4.675e-2 & 9.8991e-3 & 1.367 & 1.032e-1 & 0.933 &  1.684e-6 & 1.3159 & 1240   & 5.749e-3 \\
        3 & 1.975e-2 & 2.5395e-3 & 1.581 & 4.993e-2 & 0.844 &  8.565e-6 & 1.328  & 4880   & 1481e-3 \\
        4 & 8.188e-3 & 6.5716e-4 & 1.535 & 2.481e-2 & 0.794 &  1.920e-6 & 1.3335 & 19360  & 3761e-4 \\
        5 & 4.04e-3  & 1.6608e-4 & 1.947 & 1.247e-2 & 0.974 &  2.948e-7 & 1.3353 & 77120  & 9.475e-5 \\
        6 & 1.843e-3 & 4.2e-5    & 1.752 & 6.218e-3 & 0.887 & 2.4239e-08& 1.3362 & 307840 & 2.378e-5  \\
        \hline
    \end{tabular}
    \begin{tabular}{|@{}*{1}{>{\centering\arraybackslash}p{.025\textwidth}}@{}*{9}{>{\centering\arraybackslash}p{.108\textwidth}@{}}|}
        \hline
        \multicolumn{10}{|c|}{Triangular mesh family}\\
        \hline
        $\sharp$ & $h$& $err(p)$ & $\mathcal{O}(p)$ & $err({\Pi_0 \bm{u}})$ &
        $\mathcal{O}(\Pi_0 \bm{u})$ & $\min p$ & $\max p$ & size & sparsity \\
        \hline
        1 & 8.86e-2  & 2.928e-2 & -     & 3.435e-1 &  -    & -1.449e-3 & 1.3089 & 428    & 1.224e-2 \\
        2 & 4.221e-2 & 7.238e-3 & 1.885 & 1.572e-1 & 1.024 & -1.971e-4 & 1.3157 & 1627   & 3.266e-3 \\
        3 & 2.029e-2 & 1.775e-3 & 1.919 & 7.426e-2 & 1.024 & -1.849e-5 & 1.3296 & 6616   & 8.098e-4 \\
        4 & 1.014e-2 & 4.265e-4 & 2.055 & 3.597e-2 & 1.045 & -1.793e-6 & 1.3319 & 26405  & 2.037e-4  \\
        5 & 4.594e-3 & 1.078e-4 & 1.738 & 1.807e-2 & 0.869 & -2.125e-7 & 1.3348 & 105000 & 5.133e-5 \\
        6 & 2.204e-3 & 2.771e-5 & 1.849 & 9.167e-3 & 0.924 & -5.644e-8 & 1.3358 & 408068 & 1.322e-05 \\
        \hline
    \end{tabular}
    \caption{For each table we report the values for the discretization ($h$),
    errors ($err(p)$ and $err({\Pi_0 \bm{u}})$) and order of convergence
    ($\mathcal{O}(p)$ and $\mathcal{O}(\Pi_0 \bm{u})$) for the example in
    \cref{subsec:convergence}.
    The last columns are devoted to the minimum and maximum principle, number of
    rows (or columns) of the matrix and sparsity.  For the name used in each table, consider
    the terminology reported in the aforementioned subsection.}%
    \label{tag:errors_first}
\end{table}


\begin{table}[tbp] \footnotesize
    \begin{tabular}{|@{}*{1}{>{\centering\arraybackslash}p{.025\textwidth}}@{}*{10}{>{\centering\arraybackslash}p{.097\textwidth}@{}}|}
        \hline
        \multicolumn{11}{|c|}{Triangular mesh family}\\
        \hline
        $\sharp$ & $h$& $err(p)$ & $\mathcal{O}(p)$ & $err({\Pi_0 \bm{u}})$ &
        $\mathcal{O}(\Pi_0 \bm{u})$ & $\sharp$ faces & $\min p$ & $\max p$ & size & sparsity \\
        \hline
        1 & 2.261e-1 & 3.87e-2  & -     & 3.022e-1 & -      & 3,3,3 & 5.949e-3 & 3.6581 & 646   & 7.96e-3 \\
        2 & 1.058e-1 & 8.160e-3 & 2.048 & 1.452e-1 & 0.964 & 3,3,3 & 5.063e-4 & 3.8417 & 3014  & 1.75e-3 \\
        3 & 7.461e-2 & 4.209e-3 & 1.898 & 1.082e-1 & 0.843 & 3,3,3 & 5.492e-4 & 3.8612 & 6070  & 8.76e-4 \\
        4 & 6.163e-2 & 2.818e-3 & 2.099 & 8.737e-2 & 1.120  & 3,3,3 & 3.625e-4 & 3.8900 & 8902  & 5.99e-4 \\
        5 & 3.382e-2 & 8.337e-4 & 2.029 & 4.881e-2 & 0.970 & 3,3,3 & 9.553e-5 & 3.9453 & 29626 & 1.81e-4 \\
        6 & 2.404e-2 & 4.235e-4 & 1.984 & 3.447e-2 & 1.019 & 3,3,3 & 3.394e-5 & 3.9620 & 58233 & 9.23e-5 \\
        \hline
    \end{tabular}
    \begin{tabular}{|@{}*{1}{>{\centering\arraybackslash}p{.025\textwidth}}@{}*{10}{>{\centering\arraybackslash}p{.097\textwidth}@{}}|}
        \hline
        \multicolumn{11}{|c|}{Coarse mesh family with $c_{depth} = 2$}\\
        \hline
        $\sharp$ & $h$& $err(p)$ & $\mathcal{O}(p)$ & $err({\Pi_0 \bm{u}})$ &
        $\mathcal{O}(\Pi_0 \bm{u})$ & $\sharp$ faces & $\min p$ & $\max p$ & size & sparsity \\
        \hline
        1 & 4.586e-1 & 9.771e-2 & -     & 3.484e-1 & -     & 5,6,9 & 5.135e-3 & 3.5731 & 268   & 3.659e-2 \\
        2 & 2.070e-1 & 2.080e-2 & 1.945 & 1.46e-1  & 1.094 & 5,6,9 & 3.256e-3 & 3.6826 & 1192  & 9.215e-3 \\
        3 & 1.474e-1 & 1.11e-2  & 1.85  & 1.016e-1 & 1.066 & 5,6,9 & 7.392e-4 & 3.8070 & 2366  & 4.725e-3 \\
        4 & 1.229e-1 & 7.212e-3 & 2.368 & 8.558e-2 & 0.943 & 5,6,9 & 1.016e-3 & 3.8260 & 3422  & 3.304e-3 \\
        5 & 6.705e-2 & 2.164e-3 & 1.987 & 4.469e-2 & 1.073 & 5,6,9 & 4.978e-4 & 3.8962 & 11238 & 1.029e-3 \\
        6 & 4.775e-2 & 1.093e-3 & 2.011 & 3.168e-2 & 1.013 & 5,6,9 & 1.828e-4 & 3.9359 & 21943 & 5.328e-4 \\
        \hline
    \end{tabular}
    \begin{tabular}{|@{}*{1}{>{\centering\arraybackslash}p{.025\textwidth}}@{}*{10}{>{\centering\arraybackslash}p{.097\textwidth}@{}}|}
        \hline
        \multicolumn{11}{|c|}{Coarse mesh family with $c_{depth} = 4$}\\
        \hline
        $\sharp$ & $h$& $err(p)$ & $\mathcal{O}(p)$ & $err({\Pi_0 \bm{u}})$ &
        $\mathcal{O}(\Pi_0 \bm{u})$ & $\sharp$ faces & $\min p$ & $\max p$ & size & sparsity \\
        \hline
        1 & 4.047e-1 & 2.075e-2 & -     &  8.634e-2 & -     & 17,20,25 & 3.824e-2 & 3.3361 & 758   & 3.960e-2 \\
        2 & 2.686e-1 & 9.821e-3 & 1.825 &  5.670e-2 & 1.026 & 13,19,24 & 9.755e-3 & 3.6202 & 1528  & 2.002e-2 \\
        3 & 1.973e-1 & 5.200e-3 & 2.06  &  3.255e-2 & 1.798 & 14,20,26 & 6.993e-3 & 3.6896 & 2792  & 1.168e-2 \\
        4 & 1.392e-1 & 2.525e-3 & 2.072 &  2.540e-2 & 0.711 & 13,20,26 & 2.891e-3 & 3.7927 & 5354  & 6.346e-3 \\
        5 & 9.758e-2 & 1.198e-3 & 2.099 &  1.569e-2 & 1.357 & 14,20,27 & 1.282e-3 & 3.8650 & 10537 & 3.283e-3 \\
        6 & 6.917e-2 & 6.154e-4 & 1.936 &  1.121e-2 & 0.976 & 13,20,26 & 6.152e-4 & 3.9082 & 20554 & 1.734e-3 \\
        \hline
    \end{tabular}
    \begin{tabular}{|@{}*{1}{>{\centering\arraybackslash}p{.025\textwidth}}@{}*{10}{>{\centering\arraybackslash}p{.097\textwidth}@{}}|}
        \hline
        \multicolumn{11}{|c|}{Coarse mesh family with $c_{depth} = 5$}\\
        \hline
        $\sharp$ & $h$& $err(p)$ & $\mathcal{O}(p)$ & $err({\Pi_0 \bm{u}})$ &
        $\mathcal{O}(\Pi_0 \bm{u})$ & $\sharp$ faces & $\min p$ & $\max p$ & size & sparsity \\
        \hline
        1 & 7.449e-1 & 4.609e-2 & -     & 2.371e-1 & -     & 29,38,48 & 9.882e-2 & 2.4769 & 520   & 9.531e-2  \\
        2 & 4.761e-1 & 2.827e-2 & 1.091 & 1.538e-1 & 0.967 & 28,36,44 & 6.912e-2 & 3.1075 & 890   & 5.795e-2  \\
        3 & 3.429e-1 & 1.313e-2 & 2.336 & 8.447e-2 & 1.825 & 18,35,49 & 2.227e-2 & 3.4482 & 1644  & 3.221e-2  \\
        4 & 2.515e-1 & 7.052e-3 & 2.005 & 5.368e-2 & 1.462 & 25,37,47 & 2.055e-2 & 3.5211 & 2942  & 2.003e-2  \\
        5 & 1.773e-1 & 3.138e-3 & 2.316 & 3.355e-2 & 1.345 & 21,37,49 & 9.645e-3 & 3.6361 & 5725  & 1.056e-2  \\
        6 & 1.27e-1  & 1.630e-3 & 1.962 & 2.577e-2 & 0.79  & 24,37,44 & 2.642e-3 & 3.8138 & 10908 & 5.841e-3  \\
        \hline
    \end{tabular}
    \caption{For each table we report the values for the discretization ($h$),
    errors ($err(p)$ and $err({\Pi_0 \bm{u}})$) and order of convergence
    ($\mathcal{O}(p)$ and $\mathcal{O}(\Pi_0 \bm{u})$) for the example in
    \cref{subsubsec:two}.  The last columns are devoted to the number of faces
    for each cell (minimum, average, and maximum, respectively), the minimum and
    maximum principle, number of rows of the matrix and sparsity.  For the name
    used in each table, consider the terminology reported in the aforementioned
    subsection.}%
    \label{tag:errors_second}
\end{table}


\begin{table}[tbp] \footnotesize
    \begin{tabular}{|@{}*{1}{>{\centering\arraybackslash}p{.025\textwidth}}@{}*{10}{>{\centering\arraybackslash}p{.097\textwidth}@{}}|}
        \hline
        \multicolumn{11}{|c|}{Triangular mesh family}\\
        \hline
        $\sharp$ & $h$& $err(p)$ & $\mathcal{O}(p)$ & $err({\Pi_0 \bm{u}})$ &
        $\mathcal{O}(\Pi_0 \bm{u})$ & $\sharp$ faces & $\min p$ & $\max p$ & size & sparsity \\
        \hline
        1 & 2.261e-1 & 5.494e-2 & -    & 3.112e-1 & -     & 3,3,3 & 6.051e-3 & 9.139e-1 & 655    & 7.871e-3 \\
        2 & 1.058e-1 & 1.249e-2 & 1.95 & 1.497e-1 & 0.963 & 3,3,3 & 5.141e-4 & 9.606e-1 & 3031   & 1.746e-3 \\
        3 & 6.163e-2 & 3.807e-3 & 2.2  & 8.795e-2 & 0.985 & 3,3,3 & 3.632e-4 & 9.723e-1 & 8935   & 5.974e-4 \\
        4 & 3.382e-2 & 1.082e-3 & 2.1  & 4.784e-2 & 1.01  & 3,3,3 & 9.558e-5 & 9.809e-1 & 29691  & 1.807e-4 \\
        5 & 1.695e-2 & 2.722e-4 & 2    & 2.434e-2 & 0.978 & 3,3,3 & 3.867e-5 & 9.913e-1 & 117311 & 4.588e-5 \\
        \hline
    \end{tabular}
    \begin{tabular}{|@{}*{1}{>{\centering\arraybackslash}p{.025\textwidth}}@{}*{10}{>{\centering\arraybackslash}p{.097\textwidth}@{}}|}
        \hline
        \multicolumn{11}{|c|}{Coarse mesh family with $c_{depth} = 2$}\\
        \hline
        $\sharp$ & $h$& $err(p)$ & $\mathcal{O}(p)$ & $err({\Pi_0 \bm{u}})$ &
        $\mathcal{O}(\Pi_0 \bm{u})$ & $\sharp$ faces & $\min p$ & $\max p$ & size & sparsity \\
        \hline
        1 & 4.586e-1 & 9.705e-2 & -    & 2.802e-1 & -     & 5,6,8  & 4.628e-3 & 8.183e-1 & 277    & 3.499e-2 \\
        2 & 2.070e-1 & 2.211e-2 & 1.86 & 1.393e-1 & 0.88  & 5,6,9  & 3.253e-3 & 9.152e-1 & 1209   & 9.036e-3 \\
        3 & 1.058e-1 & 5.478e-3 & 2.08 & 6.718e-2 & 1.09  & 5,6,9  & 3.466e-4 & 9.638e-1 & 4573   & 2.466e-3 \\
        4 & 5.461e-2 & 1.246e-3 & 2.24 & 3.519e-2 & 0.977 & 5,6,9  & 5.242e-4 & 9.777e-1 & 16761  & 6.928e-4 \\
        5 & 3.033e-2 & 4.337e-4 & 1.79 & 1.928e-2 & 1.02  & 5,6,10 & 5.833e-5 & 9.889e-1 & 54617  & 2.147e-4 \\
        \hline
    \end{tabular}
    \begin{tabular}{|@{}*{1}{>{\centering\arraybackslash}p{.025\textwidth}}@{}*{10}{>{\centering\arraybackslash}p{.097\textwidth}@{}}|}
        \hline
        \multicolumn{11}{|c|}{Coarse mesh family with $c_{depth} = 4$}\\
        \hline
        $\sharp$ & $h$& $err(p)$ & $\mathcal{O}(p)$ & $err({\Pi_0 \bm{u}})$ &
        $\mathcal{O}(\Pi_0 \bm{u})$ & $\sharp$ faces & $\min p$ & $\max p$ & size & sparsity \\
        \hline
        1 & 2.686e-1 & 1.249e-2 & -    &  6.404e-2 & -    & 13,19,24 & 9.662e-3 & 8.496e-1 & 1561  &  1.927e-2\\
        2 & 1.253e-1 & 2.592e-3 & 2.06 &  2.509e-2 & 1.23 & 14,20,27 & 3.294e-3 & 9.399e-1 & 6649  &  5.112e-3\\
        3 & 8.399e-2 & 1.107e-3 & 2.13 &  1.489e-2 & 1.31 & 14,20,30 & 7.368e-4 & 9.630e-1 & 14385 &  2.438e-3\\
        4 & 6.012e-2 & 5.624e-4 & 2.02 &  9.809e-3 & 1.25 & 12,20,17 & 7.374e-4 & 9.797e-1 & 27537 &  1.294e-3\\
        5 & 4.013e-2 & 2.352e-4 & 2.16 &  6.26e-3  & 1.11 & 13,20,27 & 1.825e-4 & 9.831e-1 & 60851 &  5.944e-4\\
        \hline
    \end{tabular}
    \caption{For each table we report the values for the discretization ($h$),
    errors ($err(p)$ and $err({\Pi_0 \bm{u}})$) and order of convergence
    ($\mathcal{O}(p)$ and $\mathcal{O}(\Pi_0 \bm{u})$) for the example in
    \cref{subsubsec:two_inters}.  The last columns are devoted to the number of faces
    for each cell (minimum, average, and maximum, respectively), the minimum and
    maximum principle, number of rows of the matrix and sparsity.  For the name
    used in each table, consider the terminology reported in the aforementioned
    subsection.}%
    \label{tag:errors_third}
\end{table}

\begin{table}[tbp] \footnotesize
    \centering
    \begin{tabular}{|@{}*{1}{>{\centering\arraybackslash}p{.025\textwidth}}@{}*{7}{>{\centering\arraybackslash}p{.097\textwidth}@{}}|}
        \hline
        \multicolumn{8}{|c|}{Triangular mesh family}\\
        \hline
        $\sharp$ & $h$& $err(\hat{p})$ & $\mathcal{O}(\hat{p})$ &
        $err({\hat{\Pi}_0 \hat{\bm{u}}})$ &
        $\mathcal{O}(\hat{\Pi}_0 \hat{\bm{u}})$ & $\min \hat{p}$ & $\max \hat{p}
        $ \\
        \hline
        1 & 1.250e-1 & 4.665e-2 & -    & 1.793e-2 & -    & 3.423e-1 & 1.269 \\
        2 & 6.250e-2 & 1.158e-2 & 2.01 & 5.010e-3 & 1.84 & 1.647e-1 & 1.254 \\
        3 & 3.125e-2 & 3.229e-3 & 1.84 & 9.500e-4 & 2.4  & 8.042e-2 & 1.252 \\
        4 & 1.562e-2 & 8.601e-4 & 1.91 & 1.812e-4 & 2.39 & 3.966e-2 & 1.250 \\
        5 & 7.812e-3 & 2.187e-4 & 1.97 & 4.028e-5 & 2.17 & 1.968e-2 & 1.250 \\
        \hline
    \end{tabular}
    \begin{tabular}{|@{}*{1}{>{\centering\arraybackslash}p{.025\textwidth}}@{}*{7}{>{\centering\arraybackslash}p{.097\textwidth}@{}}|}
        \hline
        \multicolumn{8}{|c|}{Coarse mesh family with $c_{depth} = 2$}\\
        \hline
        $\sharp$ & $h$& $err(\hat{p})$ & $\mathcal{O}(\hat{p})$ &
        $err({\hat{\Pi}_0 \hat{\bm{u}}})$ &
        $\mathcal{O}(\hat{\Pi}_0 \hat{\bm{u}})$ & $\min \hat{p}$ & $\max \hat{p}$ \\
        \hline
        1 & 1.250e-1 & 7.198e-2 & -    & 8.557e-3 & -     & 3.524e-1 & 1.303 \\
        2 & 6.250e-2 & 1.769e-2 &  2.02& 2.143e-3 & 2     & 1.658e-1 & 1.262 \\
        3 & 3.125e-2 & 4.511e-3 &  1.97& 7.088e-4 & 1.6   & 8.061e-2 & 1.253 \\
        4 & 1.538e-2 & 1.246e-3 &  1.82& 3.032e-4 & 1.2   & 3.967e-2 & 1.251 \\
        5 & 7.812e-3 & 3.541e-4 &  1.86& 1.124e-4 & 1.47  & 1.968e-2 & 1.250 \\
        \hline
    \end{tabular}
    \begin{tabular}{|@{}*{1}{>{\centering\arraybackslash}p{.025\textwidth}}@{}*{7}{>{\centering\arraybackslash}p{.097\textwidth}@{}}|}
        \hline
        \multicolumn{8}{|c|}{Coarse mesh family with $c_{depth} = 4$}\\
        \hline
        $\sharp$ & $h$& $err(\hat{p})$ & $\mathcal{O}(\hat{p})$ &
        $err({\hat{\Pi}_0 \hat{\bm{u}}})$ &
        $\mathcal{O}(\hat{\Pi}_0 \hat{\bm{u}})$ & $\min \hat{p}$ & $\max \hat{p}$ \\
        \hline
        1 & 3.125e-2 & 2.717e-3 & -    &  5.785e-3 & -     & 8.066e-2 & 1.246 \\
        2 & 1.408e-2 & 9.627e-4 & 1.3  &  1.703e-3 & 1.53  & 3.965e-2 & 1.248 \\
        3 & 7.812e-3 & 3.718e-4 & 1.61 &  5.831e-4 & 1.82  & 1.968e-2 & 1.249 \\
        4 & 6.135e-3 & 2.059e-4 & 2.45 &  3.448e-4 & 2.17  & 1.968e-2 & 1.25  \\
        5 & 3.906e-3 & 8.863e-5 & 1.87 &  1.411e-4 & 1.98  & 9.803e-3 & 1.25  \\
        \hline
    \end{tabular}
    \caption{For each table we report the values for the discretization ($h$),
    errors ($err(\hat{p})$ and $err({\hat{\Pi}_0 \hat{\bm{u}}})$) and order of convergence
    ($\mathcal{O}(\hat{p})$ and $\mathcal{O}(\hat{\Pi}_0 \hat{\bm{u}})$) for the example in
    \cref{subsubsec:two_inters}. The mesh size is
    referred to the intersection mesh. For the name
    used in each table, consider the terminology reported in the aforementioned
    subsection.}%
    \label{tag:errors_thirdInters}
\end{table}



\end{document}